\documentclass[12pt]{amsart}

\usepackage{amscd,amssymb,amsfonts,amsmath,amstext,amsthm}

\usepackage{chngcntr}
\counterwithin{table}{section}

\newcommand{\ot}{\otimes}
\newcommand{\bt}{\bowtie}

\newcommand{\ds}{\displaystyle}

\def\BZ{{\mathbb Z}}
\def\BC{{\mathbb C}}
\def\BR{{\mathbb R}}
\def\BN{{\mathbb N}}

\newtheorem{thm}{Theorem}[subsection]
\newtheorem{df}[thm]{Definition}
\newtheorem{remark}[thm]{Remark}
\newtheorem{ex}[thm]{Example}
\newtheorem{cor}[thm]{Corollary}
\newtheorem{prop}[thm]{Proposition}
\newtheorem{lem}[thm]{Lemma}
\newtheorem{qu}[thm]{Question}

\newcommand{\num}{\refstepcounter{thm}}

\title{Computing Higher Indicators for the Double of a Symmetric Group}

\author{Rebecca Courter}
\address{University of Southern California, Los Angeles, CA 90089-1113}
\curraddr{Pasadena City College, Pasadena, CA 91106}
\email{rebeccep@usc.edu, recourter@pasadena.edu}

\begin{document}

\thanks{The author was supported by NSF grant DMS 1001547.}

\begin{abstract} In this paper we explicitly determine all indicators for the Drinfel'd doubles of the symmetric group acting upon up to 10 objects.  We explore when distinct characters give exactly the same indicators and when the indicators have a zero value.  We find that the indicators are all non-negative integers, which supports our conjecture that the indicators for the Drinfel'd double of any symmetric group will be non-negative integers, just as they are for the symmetric groups themselves.
\end{abstract}
 
\maketitle

%---------------------------------Section 1: Introduction--------------------------------------------------

\section{Introduction}

For an irreducible representation of a group over the complex numbers, the classical Frobenius-Schur indicator determines whether or not the representation is defined over the reals.  One may also define higher indicators for representations of a group, so that the classical one is the second indicator.  The classical Frobenius-Schur indicator extends to any semisimple Hopf algebra [LM], and the higher indicators also extend to Hopf algebras [KSZ2]. 

Indicators, including their higher analogues, are invariants that are proving very useful in the study of Hopf algebras.  They have been used in classification problems [K] [NS1]; in studying possible dimensions of the representations of a semisimple Hopf algebra [KSZ1]; and in determining the prime divisiors of the exponent of a Hopf algebra [KSZ2] [NS2].  Moreover, the indicator is invariant under equivalence of monoidal categories [MaN].  Another area where indicators are proving useful is conformal field theory; see the work of Bantay [B1] [B2].  The notion of higher indicators has also been extended to more general categories [NS1] [NS2] [NS3], where quasi-Hopf algebras play a unifying role [N1] [N2].

Frobenius and Schur gave a formula for the classic indicator $\nu$ of a representation of a finite group $G$ [Se].  They showed that for any irreducible representation, $\nu$ was always $1,0,\text{ or }-1$, and $\nu = 1$ occured precisely when the representation could be defined over $\BR$.  Similarily [LM] gave a version of this formula for a Hopf algebra $H$, and again for $V$ a simple $H$-module, with character $\chi$ and indicator $\nu(V) = \nu(\chi)$, the only possible values of $\nu(V)$ are 0, 1, and -1.  A little earlier the formula was extended for the special case of Kac algebras over $\BC$ in [FGSV].   The symmetric group $S_n$ has been an object of interest for years.  In the 1940's Young showed that the indicator for any ireducible representation of any symmetric group $S_n$ is 1.  This fact was later extended to the Hopf algebra $D(S_n)$ in [KMM].

The higher indicators of modules of finite groups are denoted $\nu_m$ for any $m\in \BZ$, in which the classical Frobenius-Schur indicator $\nu$ is equal to $\nu_2$.  [KSZ2] extended the higher indicator formula to Hopf algebras.  It is well known that the higher indicators for modules over groups are integers, but for Hopf algebras in general the values may involve roots of unity.  However, for $G$ a finite group, its Drinfel'd double $D(G)$ is a nice Hopf algebra, with a braiding on its module category, so one should expect it to behave better.  

In his dissertation [Ke], Keilberg examined the Drinfel'd double of any dihedral group, semi-dihedral group, and some other classes of non-abelian groups of order $pq$.  He proved that all indicators of these Hopf algebras are integers by explicitly finding their values.  He also proved that for the Drinfel'd double of the diheadral group, the indicators are also non-negative.  As seen before with the classical indicator, the symmetric group is also a very nice object to work with.  Scharf proved that all higher indicators of a simple $S_n$-module are \it non-negative \rm integers [S].  Whether or not this stronger statement of non-negativity is true for a simple module of $D(S_n)$, was the motivating question for this work.  This thesis gives positive supporting evidence for this conjecture.
  
All the indicators of $D(S_n)$ in this thesis were calculated using GAP, which stands for ``Groups, Algorithms, Programming."  GAP is a mathematical software system for computational discrete algebra, with particular emphasis on computational Group Theory.  [GAP]

This dissertation is organized as follows: in Section 2 we first review general definitions and theorems, and describe how we compute the higher indicators.  We then consider how automorphisms of $G$ affect characters and indicators of $D(G)$.  In Section 3 we find a more computer-efficient formula to calculate the indicators, we show that we need only find a few higher indicators in order to find them all, and we make observations about what our calculations show.  We derive lemmas and propositions describing when representation have the same set of indicator values (that is, when they are I-equivalent), why certain indicators are 0, and when the set we use to calculate the indicators is not empty.  In Section 4 we present the indicators and specific details about how they were calculated for $D(S_3)$ through $D(S_6)$.  In Section 5 we present the GAP functions written and used to do our calculations, as well as a few others used to give additional information.  Appendix A contains the character tables used to compute the indicators for $D(S_3)$ through $D(S_6)$.  Due to the large number of tables required to display the indicators for $D(S_7)$ through $D(S_{10})$, those are recorded in Appendix B which is not included in this paper: Appendix B can be found at \verb+http://www.pasadena.edu/files/syllabi/recourter_31371.pdf+.  Also due to the large amount of data computed, no details about how the indicators of $D(S_7)$ through $D(S_{10})$ were calculated are given (as was done for the smaller doubles in Section 4).

%\newpage

%-------------------Section 2. Preliminaries---------------------

\section{Preliminaries}

\subsection{}
\bf Definitions and Notation \rm   %-------------------2.1-------------------

Let $k$ be an algebraically closed field of characterist 0.

% Definition 2.1.1  \ref{hopf}
\begin{df}\label{hopf} \rm A finite dimensional \it Hopf algebra \rm over $k$ is a sextuple ($H,\text{m},\text{u},\Delta, \epsilon, S$) where $H$ is a $k$-vector space, ($H,\text{m},\text{u}$) is an algebra with associative multiplication $\text{m}: H \otimes H \rightarrow H$ and unit $\text{u}: k \rightarrow H$, ($H, \Delta, \epsilon)$ is a coalgebra with coassociative comultiplication $\Delta: H \rightarrow H \otimes H $ and counit $\epsilon: H \rightarrow k$ such that both $\Delta$ and $\epsilon$ are algebra maps, and the antipode $S: H \rightarrow H$ is an inverse for the identity under convolution. We use the sigma notation for comultiplication,
$$\Delta (h) = \sum h_1 \otimes h_2, \text{ for } h \in H.$$
Thus in sigma notation, $S$ satisfies for each $h \in H$, 
$$\sum S(h_1)h_2 = \sum h_1 S(h_2) = \epsilon(h)1.$$ 
\end{df}
\noindent For a reference see [Mo].

To define the Drinfel'd double of a group we must first define the pieces that make up the double.

%Definition 2.1.2  \ref{groupalgebra}
\begin{df}\label{groupalgebra} \rm For any group $G$, the \it group algebra \rm over the complex numbers is the collection of formal sums $$\BC G = \left \{\ \sum \alpha_g g \ \ | \ \ \alpha_g \in \BC\ \right \},$$ 
where addition is defined by $\alpha g + \beta g = (\alpha+\beta)g$, for all $\alpha, \beta \in \BC$, $g \in G$ and extended component wise, and multiplication is defined by $\alpha g \cdot \beta h = \alpha\beta (gh)$, for all $\alpha, \beta \in \BC$, $g, h \in G$ and extended component wise.\end{df}

%Definition 2.1.3  \ref{dual}
\begin{df} \label{dual}\rm
The \it dual space \rm of $\BC G$, notated by $(\BC G)^*$ or $\BC^G$, is the vector space of all $\BC$-linear maps from $\BC G \rightarrow \BC$.  Its basis is the collection of ``coordinate functions" $\{p_g\ \ | \  g \in G\}$ where for $x\in G$, 
$$p_g(x) \ = \ \delta_{g,x} \ = \ \left\{ \begin{array}{rl} 1 & g=x\\ 0 & g\neq x\end{array} \right..$$\normalsize
For $g,h,x \in G, \alpha \in \BC$, it's addition is defined by $(p_g + p_h)(x) = p_g(x) + p_h(x)$, and scalar multiplication defined by $(\alpha p_g)(x) = p_g(\alpha x) = \alpha \cdot p_g(x)$.  $\BC^G$ is also an algebra with multiplication defined by $(p_g \cdot p_h)(x) = p_g(x) \cdot p_h(x)= \delta_{g,h}p_g(x)$.
 \end{df}

\noindent We remark that both $\BC G$ and $\BC^G$ are coalgebras and Hopf algebras.  The coalgebra structure of $\BC G$ is given by $\Delta(\sum \alpha_g g) = \sum\alpha_g \Delta(g) = \sum\alpha_g(g \otimes g) = \sum\alpha_g g \otimes g$, and $\epsilon(\sum\alpha_g g) = \sum\alpha_g \epsilon(g) = \sum\alpha_g$.  The coalgebra structure of $\BC^G$ is given by $\ds \Delta(p_g) = \sum_{uv=g} p_u \otimes p_v$, and the counit is $\epsilon(p_g) = 1$.

We can further define an action $\rightharpoonup :G \otimes \BC^G \rightarrow \BC^G$ of $G$ on $\BC^G$ by, 
$$x \rightharpoonup p_g := p_{xgx^{-1}}.$$\normalsize
  With this action, we can now define the Drinfel'd double of a group.

%Definition 2.1.4 \ref{double}
\begin{df}\label{double} \rm [Mo]  The \it Drinfel'd double \rm of a group $G$, denoted by $D(G) = \BC^G \bowtie \BC G$ is a semisimple Hopf algebra with basis elements written as $p_g \bowtie x $, where $g,x \in G$ and $p_g \in \BC^G$.  Multiplication is defined by 
\begin{align*}  \left( p_k \bowtie z \right) \cdot \left( p_h \bowtie y \right):& = p_k\ (z \rightharpoonup p_h) \bowtie zy \ =\  p_k\ p_{zhz^{-1}} \bowtie zy\\ & = \delta_{k,zhz^{-1}}\ p_k \bowtie zy. \end{align*} \normalsize
Comultiplication is defined by 
$$ \Delta(p_g \bt x) : = \sum_{h \in G} (p_h \bt x) \ot (p_{h^{-1}g} \bt x). $$\normalsize 
The counit is given by $ \epsilon(p_g \bowtie x) = \delta_{g,1}$,
 and the antipode is given by $S(p_g \bt x):= p_{x^{-1}g^{-1}x} \bt x^{-1}$. 
\end{df}

A further equivalent notation to use for $D(G)$ is the smash product $\BC^G \# \BC G$, where the elements can be written as $p_g \# x$, or simply $p_gx$.  For a reference see [KMM]. \newline

Before we can define indicators, we must first discuss representations and characters.  For a reference see [Se].  For an arbitrary group $G$:

%Definition 2.1.5
\begin{df} \rm A \it representation of $G$ over $\BC$ \rm is a group homomorphism \newline $\rho: G \rightarrow GL_n(\mathbb{C})$.  The degree of $\rho$ is $n$.  
\end{df}

%Definition 2.1.6
\begin{df} A \it group algebra representation \rm is an extension of a group representation $\rho$ to the group algebra $\BC G$ and is given by $\tilde{\rho} : \BC G \rightarrow M_n(\BC)$, where $$\tilde{\rho}\left(\sum \beta_i\ g_i \right) \ = \ \sum \beta_i\ \rho(g_i).$$ \end{df}

Any degree $n$ group representation $\rho$ determines an $n$ dimensional $\BC G$-module (and hense a $G$-module) $V \subseteq \BC G$ with basis $\{ v_1, \ldots, v_n\}$.  Any element $v \in V$ written in terms of the basis $v = \alpha_1v_1 + \ldots + \alpha_nv_n$ can be written as a column vector [$\alpha_i$].  Then the module action is matrix multiplication given by $g \cdot v  := \ \rho(g) [\alpha_i] \in V$.  Similarily, group algebra representations determine $\BC G$-modules.  Since the representation determines the module, we also refer to the module $V$ as a representation.  

%Definition 2.1.7
\begin{df} \rm A representation is \it irreducible \rm if its module is simple.\end{df}

%Definition 2.1.8    \ref{character}
\begin{df}\label{character} \rm The \it character \rm $\chi$ of a representation $\rho$ of $G$ is the matrix trace of the image of the representation in $GL_n(\BC)$.  That is, $\chi(g) = \text{trace}(\rho(g))$.  The \it degree \rm of $\chi$ is the degree of $\rho$, and if $\rho$ is one-dimensional, then we call $\chi$ a linear character.  If $\rho$ is irreducible, then we also call $\chi$ an irreducible character.
\end{df}

%Definition 2.1.9   \ref{FSdef}
\begin{df}\label{FSdef}\rm  Let $V$ be a representation of the group $G$, with character $\chi$, and $m \geq 0$.  Then the \it $m^{\text{th}}$ Frobenius-Schur indicator \rm of $V$ is given by 
 $$\nu_m(V) = \nu_m(\chi): =\ \chi \left (\ \frac{1}{|G|}\ \sum_{g \in G} \ g^m\ \right ) \  = \ \frac{1}{|G|}\ \sum_{g \in G}\ \chi (g^m) \ \in \BC.$$\normalsize 
\end{df}

This definition also clearly extends to representations of a group algebra, but to define the Frobenius-Schur indicators on a Hopf algebra such as $D(G)$, we must first extend the notion of a representation. 

\subsection{}
\bf Representations and Indicators of $D(G)$ \rm   %-------------------2.2-------------------

For a finite group $G$, [KSZ2] gave a formula for the higher Frobenius-Schur indicators of $D(G)$, which we include at the end of this section.  The indicators are defined on representations of $D(G)$, and since any representation can be built from irreducible representations, we must be able to find all the irreducible representations of $D(G)$ before we can give the formula for finding the higher indicators.

To find the irreducible representations of $D(G)$, we do the following: \newline
1.   Choose one $g$ in each distinct conjugacy class of $G$. \newline
2.   For this $g$, let $C_G(g)$ be the centralizer of $g \in G$. \newline
3.   Let $V$ be a simple $\BC C_G(g)$-module corresponding to an irreducible representation of $C_G(g)$. \newline
4.   Define $\hat{V} = \BC G \bigotimes_{\BC C_G(g)} V$, the induced module.  Note:  $\hat{V}$ is a $\BC G$-module, but not simple as a $\BC G$-module. \newline
5.   Define a specific action of $\BC^G = (\BC G)^*$ on $\hat{V}$ so then $\hat{V}$ is a $D(G)$-module.  (See below) \newline
6.   With this action $\hat{V}$ is a simple $D(G)$-module (or corresponds to an irreducible representation) and all irreducible representations for $D(G)$ correspond to one of these $\hat{V}$'s.

Steps 4, 5, and 6 are explicitly described in the following lemma.

%Lemma 2.2.1   \ref{repsofDG}
\begin{lem}\label{repsofDG}\rm [KMM] \it Let $H=D(G)$ be the Drinfel'd double of $G$, and let $C_G(g)$ be the centralizer of $g \in G$. Let $V$ be a left $\BC C_G(g)$-module, and let $\hat{V} = \BC G \bigotimes_{\BC C_G(g)} V$.  Then $\hat{V}$ is a left $H$-module, via the action
$$(p_h \bowtie x) \cdot [y \otimes v]\ =\ \delta_{xygy^{-1},h} \ xy \otimes v$$\normalsize
 for $h\in G,\ x,y \in C_G(g),$ and $v \in V$.  Furthermore, if $V$ is a simple left $\BC C_G(g)$-module, then $\hat{V}$ is a simple left $H$-module.  Conversely every simple left $H$-module is isomorphic to $\hat{V}$ for some simple module $V$ of $\BC C_G(g)$, where $g$ ranges over a choice of one element in each conjugace class of $G$.
\end{lem}
 
For a finite dimensional semisimple Hopf algebra $H$, since $\Delta$ is coassociative, we may define $$\Delta^2(h) = (\Delta \otimes \text{id}) \circ \Delta(h) = (\text{id} \otimes \Delta) \circ \Delta(h) = \sum h_1\otimes h_2 \otimes h_3$$ and inductively define $\ds \Delta^n(h) =  (\Delta \otimes \text{id}) \circ \Delta^{n-1}(h) = \sum h_1 \otimes \cdots \otimes h_{n+1}$ for $n \geq 2$.  
We also define $h^{[n]} = \text{m}^{n-1} \circ \Delta^{n-1}(h) = \sum h_1 \cdots h_n$ for any $h \in H$. 

Recall an integral in a Hopf algebra is a non-zero invariant under multiplication, that is $t \neq 0 \in H$ is an \it integral \rm if $ht = \epsilon(h)t = th$ for all $h \in H$.
 Since $H$ is semisimple, by Maschke's theorem $\epsilon(t) \neq 0$.  We will let $\Lambda$ denote the unique integral of $H$ with $\epsilon(\Lambda) = 1$.  

%Equation 2.2.2    \ref{lambda}
It is not difficult to see that when $H = D(G)$,
\begin{equation}\num \label{lambda}\Lambda:=\ p_1  \bt  \left(\ \frac{1}{|G|}\ \sum_{g \in G} g\  \right)\ = \ \frac{1}{|G|}\ \sum_{g \in G} p_1 \bt g, \end{equation}\normalsize
For an example, see [Mo, Thm. 10.3.12].

Using the definition of $h^{[m]}$ above, we see
 $$\Lambda^{[m]}:=\ \sum \ \Lambda_1 \Lambda_2 \cdot \cdot \cdot \Lambda_m.$$ \normalsize
Now we are ready to give the definition of higher indicators on representations of the Drinfel'd double of a group.

%Definition 2.2.3   \ref{def of FS ind}
\begin{df}\label{def of FS ind} \rm
Given a simple $D(G)$-module $V$ and its character $\chi$ the \it $m^\text{th}$ Frobenius-Schur indicator \rm of $V$ (or $\chi$) is 
$$\nu_m(V) = \nu_m(\chi) = \chi(\Lambda^{[m]}), \ \ m \in \BN.$$\normalsize
\end{df}

In order to work with these indicators we need a little more notation and some formulas for $\Lambda^{[m]}$.

Our notation for the following definition is a modification of the notation used in [KSZ2, 7.2].

%Definition 2.2.4
\begin{df}\label{Gmug} \rm Let $G$ be a finite group.  For any $x, y \in G$ and $m \in \BN$, define
$$\begin{array}{rl}
G_m(g) & = \left \{\ x\in G\ : \ \ds \prod_{j=0}^{m-1} x^{-j} g x^j = 1\ \right \} \\
G_m(g,y) & = \left \{ x\in G\ : \ \ds \prod_{j=0}^{m-1} x^{-j} g x^j = 1 \text{ and } x^m = y \right \} \\
z_m(g,y) & =\ |\ G_m(g,y)\ |.
\end{array}$$\normalsize
\end{df}
From [KSZ2] we are taking $F=G$ and $k=1$, which makes our $G_m(g,y)$ and $z_m(g,y)$ precisely their $G_{m,1}(g,y)$ and $z_{m,1}(g,y)$.

%Proposition 2.2.5
\begin{prop} \rm [KSZ2, 7.3] \it Let $G$ be a finite group and let $\Lambda$ be the integral of $D(G)$ in Equation \ref{lambda}.  Then 
$$\begin{array}{rl}
\ds \Lambda^{[m]} & =\ds \ \frac{1}{|G|}\  \sum_{g,y \in G} \ z_m(g,y) \ p_g \bt y \\
 & = \ds \ \frac{1}{|G|}\  \sum_{\substack{ x \in G_m(g), \\ g\in G}} p_g \bt x^m.
\end{array}$$\normalsize
\end{prop}

%Corollary 2.2.6      \ref{old indicator formula}
\begin{cor}\label{old indicator formula}\rm [KSZ2, 7.4] \it Let  $\chi$ be an irreducible character of $D(G)$ induced from an irreducible character $\eta$ of $C_G(u)$ as described in Definition \ref{character} and Lemma \ref{repsofDG}.  Then 
$$\begin{array}{rl}
\ds \nu_m(\chi) & =\ds \ \frac{1}{|G|}\  \sum_{\substack{ x \in G_m(g),\\ g\in G}}\ \chi( p_g \bt x^m) \\
 & = \ds \ \frac{1}{|C_G(u)|}\  \sum_{y  \in C_G(u)} \ z_m(u,y) \ \eta(y).
\end{array}$$\normalsize
\end{cor}

%Remark 2.2.7      \ref{triv induced}
\begin{remark}\label{triv induced}\rm It is worth noting that when $u=$ id, the identity element, the centralizer $C_{G}(u)$ is all of $G$.  Thus if $\chi$ is an irreducible character of $D(G)$ induced from an irreducible character $\eta$ of $C_{G}(u) = G$, then the formula in Corollary \ref{old indicator formula} becomes
$$\ds \nu_m(\chi)\  = \ \frac{1}{|G|}\  \sum_{y \in G} \ z_m(\text{id},y) \ \eta(y)\  =  \ \frac{1}{|G|}\  \sum_{h \in G} \ \eta(h^m)\ =\ \nu_m(\eta).$$
where the last equivalence comes from Definition \ref{FSdef} and $\nu_m(\eta)$ is an indicator of $G$.  Thus any character of $D(G)$ trivially induced from a character $\eta$ of $C_{G}(\text{id}) = G$ has the same indicator values as $\eta$.
\end{remark}

\subsection{}
\bf Extending Group Automorphisms to $D(G)$ \rm   %-------------------2.3-------------------

We now describe how a group automorphism of $G$ can be extended to a Hopf automorphism on $D(G)$.

%Lemma 2.3.1  \ref{extendgroupalg}
\begin{lem}\label{extendgroupalg} \rm For a group $G$, an automorphism $\sigma: G \rightarrow G$ extends to an automorphism of the group algebra $kG$ via $$\sigma(\sum\alpha_g g) = \sum \alpha_g \sigma(g).$$  
\begin{proof}Multiplication is preserved since $\sigma(\alpha g \cdot \beta h) = \sigma(\alpha\beta gh) =$ $\alpha\beta\sigma(gh)$ $ = \alpha\beta\sigma(g)\sigma(h)$  $= \alpha\sigma(g)\beta\sigma(h) = \sigma(\alpha g) \cdot \sigma(\beta h)$.
\end{proof} \end{lem}

%Lemma 2.3.2
\begin{lem}\rm The $\sigma$ of Lemma \ref{extendgroupalg} further extends to the dual of $kG$, via for $f \in G^k$, (where $f: kG \rightarrow k$), $\sigma(f(g)) = f(\sigma(g))$.  When we consider the basis elements of $G^k$, namely $p_x$ where $x\in G$, we see that $$\sigma(p_x) = p_{\sigma^{-1}(x)}.$$\normalsize

\begin{proof}  $\sigma(p_x(g)) = p_x(\sigma(g)) = \left\{ \begin{array}{ll} 1 & x = \sigma(g) \\ 0 & x \neq \sigma(g) \end{array}\right. = \left\{ \begin{array}{ll} 1 & \sigma^{-1}(x) = g \\ 0 & \sigma^{-1}(x) \neq g \end{array}\right. = p_{\sigma^{-1}(x)}(g)$, $\forall \ g \in G$.

Thus since $p_x \cdot p_y =  \delta_{x,y}p_x$, we confirm that this extention does indeed perserve multiplication on $G^k$: \newline
$\sigma(p_x \cdot p_y) = \sigma(\delta_{x,y}p_x) = \delta_{x,y}p_{\sigma^{-1}(x)}$ \normalsize and \newline $\sigma(p_x) \cdot \sigma(p_y) = p_{\sigma^{-1}(x)} \cdot p_{\sigma^{-1}(y)} = \delta_{\sigma^{-1}(x), \sigma^{-1}(y)} p_{\sigma^{-1}(x)}$ \normalsize and since \newline $x=y \Leftrightarrow \sigma^{-1}(x) = \sigma^{-1}(y) $, that means $\delta_{x,y} = \delta_{\sigma^{-1}(x), \sigma^{-1}(y)}$, \normalsize thus multiplication is preserved. \end{proof}\end{lem}

We now further extend $\sigma$ to $D(G)$ in the following way:
%Definition 2.3.3
\begin{df} Let $\gamma_\sigma: D(G) \rightarrow D(G)$ be defined as: $$\gamma_\sigma(p_x \bt g) = \sigma(p_x) \bt \sigma^{-1}(g) = p_{\sigma^{-1}(x)} \bt \sigma^{-1}(g).$$\normalsize \end{df}

%Definition 2.3.4
\begin{df} An \it automorphism \rm of a Hopf algebra $H$ is a map $\gamma: H \rightarrow H$ that preserves the multiplication, unit, comultiplication, counit,  and antipode of $H$.  That is for all $g,h$ in $H$, $\gamma(g \cdot h) = \gamma(g) \cdot \gamma(h)$, $(\gamma \otimes \gamma)(\Delta(h)) = \Delta(\gamma(h))$, and $\gamma(S(h)) = S(\gamma(h))$. \end{df}

%Proposition 2.3.5
\begin{prop} $\gamma_\sigma$ is a Hopf automorphism of $D(G)$.
\begin{proof} 1) Using the definition of multiplicaiton in $D(G)$ given in Definition \ref{double}, we see that 
 \begin{align*}\gamma_\sigma\left( \ (p_x \bowtie g) \cdot ( p_y \bowtie h) \ \right) & = \gamma_\sigma\left( \delta_{x,gyg^{-1}}\ p_x \bowtie gh \right) \\ & =\delta_{x,gyg^{-1}}\ p_{\sigma^{-1}(x)} \bt \sigma^{-1}(gh). \end{align*}\normalsize  
Since $x = gyg^{-1}$ if and only if $\sigma^{-1}(x) = \sigma^{-1}(gyg^{-1}) = \sigma^{-1}(g)\sigma^{-1}(y)\sigma^{-1}(g^{-1})$ 
\newline$ = \sigma^{-1}(g)\sigma^{-1}(y)(\sigma^{-1}(g))^{-1}$, this means $\delta_{x,gyg^{-1}} = \delta_{\sigma^{-1}(x),\sigma^{-1}(g)\sigma^{-1}(y)(\sigma^{-1}(g))^{-1}}.$\normalsize \newline  Thus
\begin{align*}\gamma_\sigma\left( (p_x \bowtie g) \cdot ( p_y \bowtie h) \right) & =  \gamma_\sigma\left( \delta_{x,gyg^{-1}}\ p_x \bowtie gh \right) \\
& =\delta_{x,gyg^{-1}}\ p_{\sigma^{-1}(x)} \bt \sigma^{-1}(gh) \\
& = \delta_{\sigma^{-1}(x),\sigma^{-1}(g)\sigma^{-1}(y)(\sigma^{-1}(g))^{-1}} \ p_{\sigma^{-1}(x)} \bt \sigma^{-1}(g)\sigma^{-1}(h) \\
& = \left(p_{\sigma^{-1}(x)} \bt \sigma^{-1}(g) \right) \cdot \left( p_{\sigma^{-1}(y)} \bt \sigma^{-1}(h) \right)\\
& = \gamma_\sigma(p_x \bt g) \cdot \gamma_\sigma(p_y \bt h). \end{align*}\normalsize
Thus $\gamma_\sigma$ preserves multiplication.  Clearly, $\gamma_\sigma$ preserves the unit.

2) Using the definition of comultiplicaiton in $D(G)$ given in Definition \ref{double}, we see that
\begin{align*} \gamma_\sigma(\Delta(p_g \bt x)) & = (\gamma_\sigma \otimes \gamma_\sigma)(\sum_{h \in G} (p_h \bt x) \ot (p_{h^{-1}g} \bt x)) \\
& = \sum_{h \in G} \gamma_\sigma(p_h \bt x) \ot \gamma_\sigma(p_{h^{-1}g} \bt x) \\
& = \sum_{h \in G} (p_{\sigma^{-1}(h)} \bt \sigma^{-1}(x) ) \ot (p_{\sigma^{-1}(h^{-1}g)} \bt \sigma^{-1}(x)) \\
& = \sum_{h \in G} (p_{\sigma^{-1}(h)} \bt \sigma^{-1}(x) ) \ot (p_{(\sigma^{-1}(h))^{-1}\sigma^{-1}(g)} \bt \sigma^{-1}(x)) \\
& = \sum_{k \in G} (p_k \bt \sigma^{-1}(x)) \ot (p_{k^{-1}\sigma^{-1}(g)} \bt \sigma^{-1}(x) )\\
& = \Delta( p_{\sigma^{-1}(g)} \bt \sigma^{-1}(x) )\\
& = \Delta(\gamma_\sigma(p_g \bt x)),
\end{align*}\normalsize  where in the fifth equality $k = \sigma^{-1}(h)$.

Thus $\gamma_\sigma$ preserves comultiplication.  Clearly, $\gamma_\sigma$ preserves the counit.

3) Finally, using the definition of the antipode of $D(G)$ given in Definition \ref{double}, we see that 
 \begin{align*}\gamma_\sigma\left(S(p_x \bt g)\right) & = \gamma_\sigma(p_{g^{-1}x^{-1}g} \bt g^{-1}) \\
& = p_{\sigma^{-1}(g^{-1}x^{-1}g)} \bt \sigma^{-1}(g^{-1})\\
& = p_{\sigma^{-1}(g^{-1})\sigma^{-1}(x^{-1})\sigma^{-1}(g)} \bt \sigma^{-1}(g^{-1})\\
& = p_{(\sigma^{-1}(g))^{-1}(\sigma^{-1}(x))^{-1}\sigma^{-1}(g)} \bt (\sigma^{-1}(g))^{-1} \\
& = S(p_{\sigma^{-1}(x)} \bt \sigma^{-1}(g)) \\
& = S(\gamma_\sigma(p_x \bt g)). \end{align*}\normalsize  
Thus $\gamma_\sigma$ preserves the anitpode, and $\gamma_\sigma$ is a Hopf automorphism.
\end{proof}\end{prop}

%Lemma 2.3.6  \ref{joe}
\begin{lem}\label{joe}  Let $\chi$ be an irreducible character of $G$ and $\sigma \in Aut(G)$.  Define the map $\chi^\sigma$ to be $\chi^\sigma(g) = \chi(\sigma(g))$ for $g \in G$.  Then $\chi^\sigma$ is also an irreducible character of $G$.
%\begin{proof}
%From the definition of an automorphism, $G \cong \sigma(G)$ and thus they have the same (or isomorphic) irreducible characters.  By definition, %$\chi^\sigma$ is a character since it is still a map from $G$ to the field $k$.  If $\chi^\sigma$ were reducible, then it could be written as a sum of %irreducible characters of $G$, so say $\chi^\sigma = \sum_i \chi_i$, then for each $g \in G$, $\chi(\sigma(g)) = \chi^\sigma(g) = \sum_i \chi_i(g)$, thus for %every $h \in G$, $\chi(h) = \sum_i \chi_i(\sigma^{-1}(h))$ which contradicts the irreducibility of $\chi$.  Thus $\chi^\sigma$ is in fact equal to one of the %irreducible characters of $G$. 
%\end{proof}
\end{lem}

This fact is well known in group theory.  $\chi^\sigma$ is called quasi-equivalent to $\chi$.

Lemma \ref{joe} also extends to automorphism on Hopf algebras.

%Corollary 2.3.7
\begin{cor}  Let $\chi$ be an irreducible character of a Hopf algebra $H$ and $\gamma \in Aut(H)$.  Define the map $\chi^\gamma$ to be $\chi^\gamma(h) = \chi(\gamma(h))$ for $h \in H$.  Then $\chi^\gamma$ is also an irreducible character of $H$.
\end{cor}

%Lemma 2.3.8 \ref{why mix1}
\begin{lem}\label{why mix1}  Let $\chi$ be an irreducible character of $D(G)$ induced from an irreducible character $\eta$ of the centralizer $C_G(u)$, where $u \in G$.  Let $\sigma \in Aut(G)$ and $\gamma_\sigma \in Aut(D(G))$.  Then $\chi^{\gamma_\sigma}$ is an irreducible character of $D(G)$ which is quasi-equivalent to one induced from an irreducible character of the centralizer $C_G(\sigma(u))$ .
\begin{proof}
We see that $\sigma(C_G(g)) = C_G(\sigma(g))$ since if $g, x,\in G$, then $ x \in \sigma(C_G(g)) \Leftrightarrow \sigma^{-1}(x) \in C_G(g) \Leftrightarrow \sigma^{-1}(x)g = g\sigma^{-1}(x) \Leftrightarrow x\sigma(g) = \sigma(g)x \Leftrightarrow x \in C_G(\sigma(g))$.  Thus we have equality.

Since $\sigma$ is a group automorphism, $C_G(g) \cong \sigma(C_G(g)) = C_G(\sigma(g))$ and thus any irreducible character of $C_G(g)$ is quasi-equivalent to an irreducible character of $C_G(\sigma(g))$.  That is, if $\eta$ is an irreducible character of $C_G(g)$, then $\eta^\sigma$ is an irreducible character of $C_G(\sigma(g))$.   Thus, if $\chi$ is an irreducible character $D(G)$ induced from $\eta$, then $\chi^{\gamma_\sigma}$ is an irreducible character of $D(G)$ induced from $\eta^\sigma$.

\end{proof}
\end{lem}

%Example 2.3.9
\begin{ex} \rm Consider when $G=S_n$.  Let $\sigma \in Aut(S_n)$ and $\chi$ be an irreducible character of $D(S_n)$.  If $n \neq 6$, then $\sigma$ must be an inner automorphism, meaning is it defined by conjugation.  Since conjugation preserves the conjugacy classes, we see that $\chi^\sigma = \chi$ for $n \neq 6$.
\end{ex}

This however is not true for $G=S_6$.  Below we explicitly provide the outer automorphism of $S_6$ from [R Corollary 7.13].
%Definition 2.3.10
\begin{df}\label{outer auto}\rm Define $\sigma \in Aut(S_6)$ by 
$$\begin{array}{c} (1\ 2) \mapsto (1\ 5)(2\ 3)(4\ 6), \\ 
			(1\ 3) \mapsto (1\ 4)(2\ 6)(3\ 5), \\ 
			(1\ 4) \mapsto (1\ 3)(2\ 4)(5\ 6), \\ 
			(1\ 5) \mapsto (1\ 2)(3\ 6)(4\ 5), \\ 
			(1\ 6) \mapsto (1\ 6)(2\ 5)(3\ 4).
\end{array}$$ \end{df}
A routine but long calculation show that $\sigma^2=1$.  From this definition we can compute that
\begin{align*} (1\ 2)(3\ 4)(5\ 6) &\mapsto (2\ 3), \\ 
			(1\ 2)(3\ 4) & \mapsto (1\ 4)(5\ 6),\\
			(1\ 2\ 3) &\mapsto (1\ 3\ 6)(2\ 5\ 4), \\ 
			(1\ 2\ 3)(4\ 5\ 6) &\mapsto (1\ 6\ 3), \\ 
			(1\ 2\ 3)(4\ 5) &\mapsto (1\ 4\ 6\ 5\ 3\ 2), \\ 
			(1\ 2\ 3\ 4\ 5\ 6) &\mapsto (1\ 5\ 6)(2\ 4),\\
			(1\ 2\ 3\ 4) & \mapsto (2\ 6\ 3\ 5),\\
			(1\ 2\ 3\ 4)(5\ 6) &\mapsto (1\ 4)(2\ 5\ 3\ 6),\\
			(1\ 2\ 3\ 4\ 5) &\mapsto (1\ 2\ 3\ 4\ 5).\\
\end{align*}
So we see that the product of three transpositions maps to a single transposition and vice-a-versa, 3-cycles map to a product of two 3-cycles and vice-a-versa, 6-cycles map to a product of a transposition and a 3-cycle and vice-a-versa, and all other cycle types map to their own type.  We will use this fact later to prove Proposition \ref{homo prop}.
%Discuss how it ``changes" the character tables listed in Appendix A.
%\begin{align*} (3\ 4) &\mapsto (1\ 6)(2\ 3)(4\ 5), \\ 
%			(5\ 6) &\mapsto (1\ 4)(2\ 3)(5\ 6), \\ 
%			(4\ 6) &\mapsto (1\ 2)(3\ 5)(4\ 6), \\ 
%			(4\ 5\ 6) &\mapsto (1\ 3\ 6)(2\ 4\ 5), \\ 
%			(4\ 5) &\mapsto (1\ 5)(2\ 6)(3\ 4), \\
%			(1\ 2\ 3\ 4\ 5\ 6) &\mapsto (1\ 2\ 5)(3\ 4).
%\end{align*}

We end this section with a final Proposition, which proves the higher indicators of two quasi-equivalent characters are equal, and a question.
%Proposition 2.3.11  \ref{why mix}
\begin{prop}\label{why mix}  Let $\chi$ be an irreducible character of $D(G)$ and let $\sigma \in Aut(D(G))$.  Then $\nu_m (\chi^\sigma) = \nu_m (\chi)$ for all $m \in \BN$.
\begin{proof}
Recall that $\Lambda$ given in Equation \ref{lambda} is the unique integral of $D(G)$, and so $\Lambda^\sigma  = \sigma(\Lambda) = \Lambda$.  Since $\sigma$ also preserves the coalgebra structure of $D(S_n)$, $(\Lambda^{[n]})^\sigma = (\Lambda^\sigma)^{[n]} = \Lambda^{[n]}$.
Thus when we look at the original definition of the higher indicators as given in Definition \ref{def of FS ind} we see that 
$\nu_m (\chi^\sigma) = \chi^\sigma(\Lambda^{[m]}) = \chi (\sigma(\Lambda^{[m]})  = \chi (\Lambda^{[m]}) = \nu_m (\chi)$.
\end{proof}
\end{prop}

%Question 2.3.12
\begin{qu}\label{all autos?}  Does every Hopf automorphism of $D(S_n)$ for $n \geq 5$ come from an automorphism $\sigma \in Aut(G)$ as in Lemma \ref{why mix1}?
\end{qu}

%----------------------------------Section 3. Indicator Equivalences -------------------------------------------------

\section{Indicator Equivalences} % Chapter 3  

%3.0  
In this section we explore when higher Frobenius-Schur indicators are equivalent.  We first discuss what formulas we used to compute all the indicators of $D(S_n)$ in Section 3.1.  Then in Section 3.2, since all of the indicators we calculated were integers, we discuss results showing we need only consider a smaller subset of all possible higher indicators $\nu_m$ of a character to know them all.   In Section 3.3 we discuss when all higher indicators of different characters are correspondingly equivalent, and then focus on when an indicator is zero valued in Section 3.4.  We finish this section by trying to answer a question that naturally arises from a proposition in Section 3.4.

%3.1
\subsection{}
\bf A More Computable Formula for $\nu_m$ \rm   %-------------------3.1-------------------

We first computed indicators in GAP using the formula in Corollary \ref{old indicator formula}.  However, these computations were very time consuming, and we were only able to compute the indicators for $D(S_3)$ through $D(S_6)$ using this method.  Attempting to compute the indicators for $D(S_7)$ overloaded our computer.  We needed a more efficient formula or way of programming in order to compute the indicators for the double of larger symmetric groups.

Recently a different formula was given in [IMM] for calculating the higher indicators of other Hopf algebras.  We will use a variation of this formula which gives a more efficient computation. 

Recall from Definion \ref{Gmug} that $$G_m(g,y) = \ds \left \{ x\in G\ : \ \ds \prod_{j=0}^{m-1} x^{-j} g x^j = 1 \text{ and } x^m = y \right \}.$$\normalsize

We fix $g = u$ in a conjugacy class.  Then as noted in [IMM],
$$G_m(u,y) = \{h\in G \ | \ h^m = y, \ (uh)^m = h^m\}.$$\normalsize
  That is $ (uh)^m = h^m$ if and only if $\ds\prod_{j=0}^{m-1} h^{-j}uh^j = 1$.

Now let $\eta$ be a character for an irreducible representation of $C_G(u)$, where as before $C_G(u)$ is the centralizer of our element $u$.  When calculating $\nu_m(\chi)$ as in Corollary \ref{old indicator formula}, [IMM] further noted that the first condition of $G_m(u,y)$ is superfluous.  That is, $(uh)^m = h^m$ implies $h^m \in C_G(u)$.
Thus instead of summing over the centralizer, we can sum over a new set.
%Definition 3.1.1
\begin{df}\label{Gmu} \rm For a group $G$, define $\tilde{G}_m(u) := \{h \in G\ |\ (uh)^m = h^m\}$\normalsize.   \end{df}

%Corollary 3.1.2
\begin{cor}\label{ind formula}\rm [IMM] \it For a group $G$, a fixed element $u$ in a conjugacy class, and an irreducible character $\eta$ of $C_G(u)$, the formula for the $m^\text{th}$ Frobenius-Schur indicator of a character $\chi$ of $D(G)$ induced from $\eta$ as given in Corollary \ref{old indicator formula} is equivalent to 
$$\nu_m(\chi) = \frac{1}{|C_G(u)|} \ \sum_{h \in \tilde{G}_m(u)} \eta(h^m) .$$\normalsize
\end{cor}

In order to make the formula in Corollary \ref{ind formula} more efficient when using GAP, we consider the elements $h^m$.  There may be many different $h \in \tilde{G}_m(u)$ that have the same $h^m$, so rather than summing over $h$, we could sum over $h^m$.  But even distinct $h^m$ may be in the same conjugacy class of $C_G(u)$, which means the value of $\eta(h^m)$ will be the same, so we could sum over the conjugacy classes of $C_{G}(u)$.  Thus to avoid summing the same value more than once we give the following definitons and notation.

Let $Conj_{C_G(u)}$ be the set of conjugacy classes of $C_{G}(u)$, and $\mathfrak{R}_{C_G(u)}$ be a fixed set of conjugacy class representatives - that is, a set of elements in $C_{G}(u)$ such that each element is from a different conjugacy class.

For example, consider $C_{S_4}((1 2)(3 4) ) = \langle (1 2), (1 3)(2 4), (3 4) \rangle \cong D_8$.  Then \newline
$Conj_{C_{S_4}( (1 2)(3 4) )} = \{\ [ () ],\ [  (1 2), (3 4) ],\ [ (1 2)(3 4) ],\ [ (1 3)(2 4), (1 4)(2 3) ],$ 

\hfill $ [ (1 3 2 4), (1 4 2 3) ]\ \}$

\noindent and we may choose $\mathfrak{R}_{C_{S_4}( (1 2)(3 4) )} = \{\ () ,\ (1 2),\ (1 2)(3 4) ,\ (1 3)(2 4),\ (1 3 2 4)\ \}$.

%Definition 3.1.3     
\begin{df}\label{Gmm} \rm Let $\mathfrak{R}_{C_G(u)}$ be a fixed set of conjugacy class representatives of $C_G(u)$.  Define $\tilde{G}_m^m(u)$ to be the set of all $y \in \mathfrak{R}_{C_G(u)}$ for which there exists an $h \in \tilde{G}_m(u)$ such that $h^m$ is in the conjugacy class of $y$ in $C_G(u)$.

To construct the set $\tilde{G}_m^m(u)$ from the set $\tilde{G}_m(u)$ we would: \newline
(1) raise all elements $h \in \tilde{G}_m(u)$ to the $m^{\text{th}}$ power; \newline
(2) organize these new elements into their conjugacy classes; and \newline
(3) choose one representative from each of these conjugacy classes. 

The set $\tilde{G}_m^m(u)$ is the collection of all these conjugacy class representatives.  Another description of $\tilde{G}_m^m(u)$ will be given in Lemma \ref{3.1.5} below.
\end{df}

%Definition 3.1.4
\begin{df}\label{gamma}\rm Let $y \in C_{G}(u)$.  Define $\Gamma_m(u,y)$ to be the number of $h \in G$ such that $h\in \tilde{G}_m(u)$ and $h^m$ is in the conjugacy class of $y$ in $C_G(u)$, that is: 
$$\Gamma_m(u, y):=\ |\ \{ h \in G \ |\  (uh)^m = h^m \text{ and } h^m \in cl_{C_{G}(u)}(y)\}\ |,$$\normalsize  where $cl_{C_G(u)}(y)$ denotes the conjugacy class of $y$ in $C_G(u)$.
\end{df}

Note that for $y_1$ and $y_2$ in the same conjugacy class of $C_G(u)$, $\Gamma_m(u, y_1) = \Gamma_m(u, y_2)$.

%Lemma 3.1.5
\begin{lem}\label{3.1.5}  Using $\Gamma_m(u,y)$ we see:  $$\tilde{G}_m^m(u) = \{ y \in \mathfrak{R}_{C_G(u)} \ |\ \Gamma_m(u, y) \neq 0\}.$$\normalsize
\end{lem}

%Proposition 3.1.6
\begin{prop}\label{mine} Let $u$ be a fixed representative of a conjugacy class of $G$, and let $\eta$ be an irreducible character of $C_G(u)$.  Then the $m^\text{th}$ Frobenius-Schur indicator of a character $\chi$ of $D(G)$ induced from $\eta$ is given by
$$\nu_m(\chi) = \frac{1}{|C_{G}(u)|} \, \sum_{y\in \tilde{G}_m^m(u)} \Gamma_m(u, y)\eta(y).$$\normalsize
\end{prop}
\begin{proof}  Using the formula in Corollary \ref{ind formula} and Definitions \ref{Gmm} and \ref{gamma}, we have
\begin{align*} \nu_m(\chi) & = \frac{1}{|C_{G}(u)|} \, \sum_{h\in \tilde{G}_m(u)} \eta(h^m) \\
			 & =  \frac{1}{|C_{G}(u)|} \, \sum_{y\in \mathfrak{R}_{C_G(u)}} \Gamma_m(u, y)\ \eta(y) \\
			 & =  \frac{1}{|C_{G}(u)|} \, \sum_{y\in \tilde{G}_m^m(u)} \Gamma_m(u, y)\ \eta(y)
\end{align*}\normalsize
\end{proof}
The formula in Proposition \ref{mine} is the function we used to calculate all the higher indicators for $D(S_n)$ for $n \leq 10$.  Further discussion on how we translated this formula into GAP functions is found in Section 5.
%\newpage

%3.2
\subsection{}
\bf Equivalent Indicator Values \rm   %-------------------3.2-------------------

There are an infinite number of higher indicators $\nu_m$ for a single character $\chi$ of $D(G)$, but in order to find all of them, it turns out that we only need to calculate a small set of them.  We only need to find $\nu_d$ where $d$ is a divisor of the exponent of the group $G$.  Recall, the \it exponent \rm of a group $G$ is the smallest positive integer $e$ such that $g^e = 1$ for all $g \in G$.  We write $e = exp(G)$ for the exponent of $G$.  Thus all of our indicator calculations found in Section 4 only included such $\nu_d$.   

[IMM] showed that if $\nu_d(\chi)$ is an integer for every character $\chi$, then in fact each $\nu_m(\chi)$ is equal to some $\nu_d(\chi)$, and so all the indicators will be found by just finding the $\nu_d$.

%Theorem 3.2.1
\begin{thm}\label{divisors of exp}\rm [IMM] \it Let $\chi$ be an irreducible character of $D(G)$, induced from $\eta$ an irreducible character on $C_G(u)$, for $u \in G$ fixed.  Let $e$ = exp$(G)$, the exponent of the group $G$, and say $m \in \mathbb{N}$ and $d$ = gcd$(m,e)$.  Then
\newline
$1$.  If $\nu_d(\chi) \in \mathbb{Z}$, for all $\chi$ induced from characters on $C_G(u)$, then $\nu_m(\chi) \in \mathbb{Z}$, for all $\chi$.
\newline
$2$.  If $\nu_d(\chi) \in \mathbb{Z}$, $m=dk$ and $(k,e) = 1$, then $\nu_{dk}(\chi) = \nu_d(\chi)$.
\newline
$2'$.  If $\nu_d(\chi) \in \mathbb{Z}$, and $m=dk$, then $\nu_{dk}(\chi) = \nu_d(\chi)$. \rm 
\end{thm}

It should be noted that the equivalence of 2 and $2'$ in Theorem \ref{divisors of exp} is due to Richard Ng.

%3.3
\subsection{}
\bf I-equivalent Irreducible Characters \rm   %-------------------3.3-------------------

When first calculating all the indicators, we noticed that many distinct irreducible characters have the same set of indicator values.  We decided to create equivalence classes of characters to avoid repeating the same indicators over and over again in our tables.

%Definition 3.3.1
\begin{df}\rm Two characters $\chi$ and $\eta$ of $D(G)$ are indicator equivalent or \it I-equivalent \rm if $\nu_m(\chi) = \nu_m(\eta)$ for all positive integers $m$.  Conversely, two characters $\chi$ and $\eta$ of $D(G)$ are \it I-inequivalent \rm if there exists an $m$ for which $\nu_m(\chi) \neq \nu_m(\eta)$.
\end{df}

As we saw at the end of Section 2 in Corollary \ref{why mix}, given an automorphism $\sigma$ of $D(G)$, two quasi-equivalent characters $\chi$ and $\chi^\sigma$ are I-equivalent.

I-equivalent is an equivalence relationship, so all the irreducible characters of $D(S_n)$ can be collected by their \it I-equivalence class\rm.  Most irreducible character I-equivalence classes of $D(S_n)$ only contain characters induced from the same centralizer $C_{S_n}(u)$, however in a few instances this is not the case.  

%Definition 3.3.2
\begin{df}  \rm We say that an irreducible character I-equivalence class is

1) \it homogenous \rm if all the characters in that class are induced from the same centralizer $C_G(u)$, or 

2) \it mixed \rm if it contains characters induced from different centralizers of $G$.  
\end{df}

%Prop 3.3.3
\begin{prop}\label{homo prop} \

($1$)  The irreducible character I-equivalency classes of $D(S_5)$, $D(S_7)$, $D(S_8)$, $D(S_9)$, and $D(S_{10})$ are all homogenous.

($1'$)  If the irreducible characters $\chi_1$ and $\chi_2$ of $D(S_n)$ have the same indicator values for all $m$, then $\chi_1$ and $\chi_2$ are induced from the same centralizer $C_{S_n}(u)$ for $n = 5$, $7 \leq n \leq 10$.

(2)  For $D(S_6)$, the only irreducible character I-equivalence classes that are mixed may be as expected from the outer automorphism on $S_6$.

\begin{proof} The first statement is proven by observation of the I-equivalency classes as recorded in Section 4 and Appendix B of this paper.  The second statement considers the outer automorphism $\sigma$ of $S_6$ that has order 2 given in [R] and explicitly provided in Definition \ref{outer auto}.  This outer automorphism maps transpositions to the product of three transpositions, 3-cycles to the product of two 3-cycles, 6-cycles to the product of a 3-cycle and a transposition, and all other elements to their same cycle type.  When we look at the irreducible character I-equivalency classes we find the characters induced from $C_{S_6}( (1,2) )$ are mixed with or I-equivalent to characters induced from $C_{S_6}( (1,2)(3,4)(5,6) )$.  Characters induced from $C_{S_6}( (1,2,3) )$ are mixed with or I-equivalent to characters induced from $C_{S_6}( (1,2,3)(4,5,6) )$.  Characters induced from $C_{S_6}( (1,2,3)(4,5) )$ are mixed with or I-equivalent to characters induced from $C_{S_6}( (1,2,3,4,5,6) )$.  All other characters, induced from other centralizers are in homogeneous irreducible character I-equivalence classes.  This is consistent with the outer automorphism of $S_6$ and with the fact that an irreducible character induced from $C_{S_6}(u)$ is I-equivalent to another one induced from $C_{S_6}(\sigma(u))$ as was shown in Lemma \ref{why mix1} and Proposition \ref{why mix}. 
\end{proof}
\end{prop}

We also note that the mixed I-equivalence classes of $D(S_3)$ and $D(S_4)$ are not unexpected due to the small sizes of $S_3$ and $S_4$.  
Proposition \ref{homo prop} raises a natural question.
%Question 3.3.4
\begin{qu} Is it true for all $n \geq 5$, $n \neq 6$, that all irreducible character I-equivalency classes of $D(S_n)$ are homogenous? \end{qu}  We note that this question is related to Question \ref{all autos?}.

We tried to see if there was a connection between characters being I-equivalent and the centralizers from which they were induced.  This is why in Section 4 we give details about the centralizer groups, specifically noting if they are abelian or not.  Recall that all of the irreducible characters of an abelian group $G$ are linear.
 
%Lemma 3.3.5
\begin{lem}\label{S3&S4}  If $n = 3$ or $4$, then the irreducible characters of $D(S_n)$ induced from the same centralizer $C_{S_n}(u)$ are I-equivalent if and only if $C_{S_n}(u)$ is abelian.
\begin{proof}  By observation, as recorded in Section 4, Sections 1 and 2. \end{proof}
\end{lem}
Lemma \ref{S3&S4}, however, does not extend to the irreducible characters of larger doubles, such as $D(S_5)$, as we see in the following example.  

%Example 3.3.6
\begin{ex}\rm  \

(1) In $S_5$, let $u_4 = (1, 2, 3)$.  Then its centralizer is $C_{S_5}(u_4) \cong C_6$, the cyclic group of order 6, which is abelian.  However, the irreducible characters induced from $u_4$ are broken into two homogenous I-equivalency classes.  Thus an abelian centralizer does not imply all irreducible characters induced from that centralizer are in the same I-equivalency class. 

(2)  Now let $u_5 = (1,2,3) (4,5)$.  The centralizers of $u_4$ and $u_5$ are identical, since $C_{S_5}(u_5) \cong C_6$.  Yet, all the irreducible characters induced from $u_5$ are in the same single homogenous I-equivalency class.  Thus irreducible characters induced from isomorphic centralizers are not necessarily I-equivalent, nor will those isomorphic centralizers have the same number of irreducible character I-equivalency classes.
\end{ex}

We must also consider the set we sum over when computing the higher Frobenius-Schur indicators, that is $\tilde{G}_m^m(u)$.  From our computations, we noticed that if the sets $\tilde{G}_m^m(u)$ were not empty, then they contained the identity element.  We also noticed that a portion of the $\tilde{G}_m^m(u)$ contained only the identity.

%Lemma 3.3.7
\begin{lem}\label{I-equiv}  Let $u$ be a fixed conjugacy class representative of $G$.  If $C_{G}(u)$ is abelian, and $\tilde{G}_m^m(u) = \emptyset$ or $\tilde{G}_m^m(u)= \{ \text{id} \}$ for each $m$, then all $D(G)$ characters induced from $C_G(u)$ will be I-equivalent.

\begin{proof}  Let $\chi_1$ and $\chi_2$ be irreducible characters of $D(G)$ induced from irreducible characters of $C_{G}(u)$, $\eta_1$ and $\eta_2$ respecitively.  Since $C_{G}(u)$ is abelian, that means $\eta_1$ and $\eta_2$ are linear characters, so $\eta_1(\text{id}) = 1 = \eta_2(\text{id})$.  Now for a fixed $m$, if $\tilde{G}_m^m(u) = \emptyset$, then $\nu_m(\chi_1) = 0 = \nu_m(\chi_2)$.  If $\tilde{G}_m^m(u)= \{ \text{id} \}$, then \newline 
\begin{align*} \nu_m(\chi_1) & = \frac{1}{|C_{G}(u)|} \, \sum_{y\in \tilde{G}_m^m(u)} \Gamma_m(u, y)\ \eta_1(y) \\ 
					& = \frac{1}{|C_{G}(u)|} \ \Gamma_m(u, \text{id})\ \eta_1(\text{id}) \\
					& = \frac{1}{|C_{G}(u)|} \ \Gamma_m(u, \text{id})\ \eta_2(\text{id}) \\
					& =  \frac{1}{|C_{G}(u)|} \, \sum_{y\in \tilde{G}_m^m(u)} \Gamma_m(u, y)\ \eta_2(y) \\ 
					& = \nu_m(\chi_2)
\end{align*}\normalsize  Thus $\chi_1$ and $\chi_2$ are I-equivalent.
\end{proof}
\end{lem}

We make one last observation reguarding Lemma \ref{I-equiv}.  We ask if Lemma \ref{I-equiv} can be weakened by omitting the assumption that $C_G(u)$ is abelian, since for $D(S_n)$ where $n\leq 10$, all the irreducible characters induced from the same $C_{S_n}(u)$ were I-equivalent exactly when for each $m$, $\tilde{G}_m^m(u)$ was empty or only contained the identity.  %(The computations for $\tilde{G}_m^m(u)$ for $n = 9$ or $10$ were too large to store or record. - but on 4/21/12 I used GAP to compute only the centralizers and sets $\tilde{G}_m^m(u) for the u that had only ONE I-equivalence class - and all the centralizers were abelian, and all the sets were either empty or contained only the identity.  Work Saved in GAP DS9 DS10 I-equiv workspace.)  
We end this section with this question.
%Question 3.3.8
\begin{qu} Does the condition that for each $m$, $\tilde{G}_m^m(u) = \emptyset$ or $\tilde{G}_m^m(u) = \{ id \}$ imply that $C_G(u)$ is abelian? \end{qu}
%\newpage

% 3.4
\subsection{}
\bf Zero Valued Indicators \rm   %-------------------3.4-------------------

In this section we discuss when a higher indicator of a character of $D(S_n)$ may have a value of zero.  As mentioned in Remark \ref{triv induced}, some characters of $D(S_n)$ are trivially induced and have the same indicator values as those of $S_n$.  Since the higher indicators of $S_n$ have been known for some time, we would expect some of these trivially induced characters to have zero valued indicators - which some do.  So we exclude this situation (when $u = $ id) from our consideration of when $\nu_m(\chi) = 0$.  Another occurance of zero valued indicators that we ignore is when $m=1$.  When $m = 1$, all but the induced trivial character of $S_n$ have indicator values of zero, as we would expect.  Thus, when $\chi$ is trivially induced or $m=1$, we expect that $\nu_m(\chi)$ may be zero.

When first looking over our tables of indicators, we noticed that for odd $m$, there were many zero valued indicators.  More specifically, there were clusters of zeros coming from characters induced from the same centralizers, $C_G(u)$.  After inspecting the sets $\tilde{G}_m^m(u)$ that we sum over, we realized they were all empty, and that was why the indicator values were zero.  Note that we have already seen significance in $\tilde{G}_m^m(u)$ being empty in Lemma \ref{I-equiv}.  Now we consider how it affects the indicator values themselves.  $\tilde{G}_m^m(u)$ is empty if and only if $\tilde{G}_m(u)$ is  empty, thus the simplest case when $\nu_m(\chi) = 0$ is exactly when $\tilde{G}_m(u) = \{ h \in G \ | \ (uh)^m = h^m\} = \emptyset$. 

%3.4.1 Proposition:  u odd and m odd -> Gm(u) is empty. 
\begin{prop}\label{empty} If $u$ is an odd permutation of $G=S_n$ and $m$ is an odd positive integer, then $\tilde{G}_m(u) = \emptyset$.  Thus $\nu_m(\chi) = 0$ when $m$ is odd and $\chi$ is induced from an irreducible character of $C_{S_n}(u)$ for $u$ an odd permutation.
\begin{proof}
If $u$ is an odd permutation of $S_n$ then $\rm sgn \it (u) = -1$.  Recall that  for arbitrary permutations $a$ and $b$, $\rm sgn\it (ab) = \rm sgn\it (a)\rm sgn\it (b)$.  Thus we have for any $h \in S_n$, $\rm sgn \it (h^m) = \rm sgn \it (h)^m$, and $\rm sgn \it ((uh)^m) = \rm sgn \it (uh)^m$.  If $m$ is an odd positive integer, then $h^m$ has the same parity as $h$, and $(uh)^m$ has the same parity as $uh$.  

Thus for any $h \in S_n$
\begin{align*} 
\rm sgn \it ((uh)^m) = \rm sgn \it (uh) = \rm sgn \it (u) \rm sgn \it (h) = -\rm sgn \it (h) \neq \rm sgn \it (h) = \rm sgn \it (h^m)
\end{align*}
and so $(uh)^m$ can not equal $h^m$.  Thus $\tilde{G}_m(u) = \emptyset$, which means  $\nu_m(\chi) = 0$ when $m$ is odd and $\chi$ is induced from an irreducible character of $C_{S_n}(u)$ for $u$ an odd permutation.
\end{proof}
\end{prop} 
\vskip 1pc
  
%When is indicator = 0, but G_m not empty?  Chart
Certainly $\nu_m(\chi)$ can be $0$ when $\tilde{G}_m^m(u)$ is not empty, although this appears somewhat rare when we consider all indicators of $D(S_3)$ through $D(S_{10})$.  First, we formally define the notion of expected zeros and unexpected zeros using Proposition \ref{empty}.  

%Definition 3.4.2
\begin{df}\rm  We say $\nu_m(\chi) = 0$  is an \it expected zero \rm if $m=1$, or $\chi$ is trivially induced, or both $u$ is an odd permutation and $m$ is an odd positive integer.  Conversely, we say $\nu_m(\chi) = 0$  is an \it unexpected zero \rm if $m \neq 1$,  $\chi$ is not trivially induced, and $u$ and $m$ are not both odd.  We also may refer to an unexpected zero as an \it unexpected zero valued indicator.\rm
\end{df}
Next, recall that by Theorem \ref{divisors of exp}, to find all possible values of indicators $\nu_m(\chi)$, it suffices to find $\nu_d(\chi)$ where $d$ divides the exponent $e$, provided $\nu_d(\chi) \in \BZ$.  Thus we say  a \it distinct non-trivial higher indicator \rm of $D(G)$ is an indicator $\nu_m$ such that $m \geq 3$ and $m$ is a divisor of the exponent of the group $G$.  Here we consider $\nu_1$ and $\nu_2$ to be trivial since for any irreducible character of $D(G)$, $\nu_1$ is equal to 0 unless $\chi$ is the trivial character and $\nu_2$ is always equal to 1.

Finally, a \it non-trivial indicator value \rm is the value of $\nu_m(\chi)$ where $\chi$ is not a trivially induced character and $m \geq 3$.

%Proposition 3.4.3
\begin{prop}\label{unex 0s} \

1)  $D(S_3)$, $D(S_4)$, and $D(S_5)$ have no unexpected zero valued indicators.

2)  $D(S_6)$ has only two unexpected zero valued indicators.  It has thirteen non-trivially induced I-equivalent character classes and 10 distinct non-trivial higher indicators.  Thus $D(S_6)$ has 130 non-trivial indicator values.

3)  $D(S_7)$ has only two unexpected zero valued indicators.  It has 33 non-trivially induced I-equivalent character classes and 22 distinct non-trivial higher indicators.  Thus $D(S_7)$ has 726 non-trivial indicator values.

4)  $D(S_8)$ has thirteen unexpected zero valued indicators.   It has 66 non-trivially induced I-equivalent character classes and 30 distinct non-trivial higher indicators.  Thus $D(S_8)$ has 1980 non-trivial indicator values.

5)  $D(S_9)$ has 38 unexpected zero valued indicators.  It has 107 non-trivially induced I-equivalent character classes and 46 distinct non-trivial higher indicators.  Thus $D(S_9)$ has 4922 non-trivial indicator values.

6)  $D(S_{10})$ has 21 unexpected zero valued indicators.  It has 196 non-trivially induced I-equivalent character classes and 46 distinct non-trivial higher indicators.  Thus $D(S_{10})$ has 9016 non-trivial indicator values.
\begin{proof}  By observation of indicator tables found in Section 4 and Appendix B.
\end{proof}
\end{prop}

To summarize, 1.54\% of $D(S_6)$'s non-trivial indicator values are unexpected zeros, 0.28\% of $D(S_7)$'s non-trivial indicator values are unexpected zeros,  0.66\% of $D(S_8)$'s non-trivial indicator values are unexpected zeros, 0.77\% of $D(S_9)$'s non-trivial indicator values are unexpected zeros, and 0.23\% of $D(S_{10})$'s non-trivial indicator values are unexpected zeros.  Thus most zeros are expected.

Now in Table \ref{Table 3.1} below, we explicitly list the conditions in which the unexpected zeros from Proposition \ref{unex 0s} appear.  In Table \ref{Table 3.1}: 

$\bullet \ S_n$ tells us which double $D(S_n)$ has unexpected zero valued indicators,

$\bullet \ u$ is the conjugacy class representative of $S_n$ from whose centralizer the characters of $D(S_n)$ are induced, 

$\bullet\ m$ refers to which higher indicator $\nu_m$ had a zero value, and

$\bullet$ ``no. of chars with $ \nu_m = 0$ " is the number of I-equivalent character classes for which the indicator is zero when $u$ and $m$ are as specified.  In this column ``all" means that all the characters induced from $C_{S_n}(u)$ have a zero valued indicator for the specificed $m$.

\newpage

%``induced chars from $u$" is the total number of I-equivalent character classes that contain characters induced from characters of $C_{S_n}(u)$, 
%``total chars of $D(S_n)$ is the total number of I-equivalent character classes of $D(S_n)$, not including the I-equivalent character classes that contain characters induced from characters of $C_{S_n}(\text{id}) = S_n$.

%\setcounter{table}{2}

\begin{table}[ht] 
\caption{Unexpected Zero Valued Indicators}
$\begin{array}{c|c|cc}   
D(S_n) & u &\ m\ & \begin{array}{c} \text{no. of chars} \\ \text{with }\nu_m = 0 \end{array} 
 \rule[-7pt]{0pt}{20pt} \\ \hline
S_6 & \ (12)(34) \ & 3 & 1  \rule[0pt]{0pt}{13pt} \\ 
S_6 & (123) & 3 & 1 \\ \hline

S_7 & (12)(34) & 3 & 1\\ 
S_7 & (123) & 3 & 1 \\ \hline

S_8 & (12)(34) & 3, 5 & 2 \\ 
S_8 & (12)(34)(56)(78) & 5 & \rm all \  7 \\ 
S_8 & (123) & 3, 5 & 1 \\ \hline

S_9 & (12)(34) & 3, 5 & 3 \\ 
S_9 & (12)(34) &  5 \ \rm only & 3 \\ 
S_9 & (12)(34)(56)(78) & 5 & \rm all \  7 \\
S_9 & (123) & 3, 5 & 2 \\ 
S_9 & (123) & 5 \ \rm only & 2 \\
S_9 & (123)(456) & 5 & 3 \\
S_9 & (123)(456)(789) & 5 & \rm all \  6 \\
S_9 & (1234)(56) & 3, 5 & 1 \\
S_9 & (1234)(567)(89) & 5 & 1 \\
S_9 & (12345) & 3, 5 & 1 \\
S_9 & (12345)(67)(89) & 5 & \rm all \  2 \\ \hline

S_{10} & (12)(34) & 3, 5, 7, 15, 35 & 1 \\
S_{10} & (12)(34) & 3, 7 & 1 \\
S_{10} & (12)(34) & 3 \ \rm only & 1 \\
S_{10} & (12)(34)(56)(78) & 3 & 3 \\
S_{10} & (123) & 3, 5, 7, 15, 35 & 1 \\
S_{10} & (123) & 3 \ \rm only & 2 \\
S_{10} & (123)(456) & 3 & 1 \\
S_{10} & (1234)(56) & 3 & 1 \\
S_{10} & (12345) & 3 & 1\rule[-10pt]{0pt}{10pt}
\end{array} $\label{Table 3.1}
\end{table}

As we can see from the above table, it appears that the only time an unexpected zero valued indicator occurs is when $m$ is odd.  This raises the question:
%Question 3.4.4
\begin{qu}  Do unexpected zero valued indicators only occur when $m$ is odd? \end{qu}

We can also see that when $m=5$ and $D(S_n)$ is either $D(S_8)$ or $D(S_9)$, there are four occurances of all the induced characters of specific $C_{S_n}(u)$ having a zero valued indicator.  Upon closer inspection, we find that in those instances $\tilde{G}_5^5(u)$ is empty.

%D(S_n) & u &\ m\ & \begin{array}{c} \text{no. of chars} \\ \text{w/ }\nu_m = 0 \end{array} & \begin{array}{c} %\text{induced} \\ \text{chars from } u\end{array} &  \begin{array}{c} \text{total chars}\\ \text{of }D(S_n) \end{array}

%3.5
\subsection{}
\bf When $\tilde{G}_m(u)$ is not empty \rm   %-------------------3.5-------------------

In this section we consider the converse of Proposition \ref{empty}.  That is, we ask when it is true that if $u$ is an even permutation of $S_n$ or $m$ is an even positive integer, then $\tilde{G}_m(u) \neq \emptyset$.  

As we saw at the end of the last section, when $u$ is an even permutation of $S_n$ there are examples when $\tilde{G}_5(u) = \emptyset$, so only requiring $u$ to be an even permutation is not sufficient.  However, at least part of the converse of Proposition \ref{empty} is true.  Proposition \ref{notempty} below proves if $m$ is even, then $\tilde{G}_m(u)$ is not empty.

Interestingly enough, all our examples suggest that if $\tilde{G}_m(u)$ is not empty, then there exists an $h \in \tilde{G}_m(u) $ such that $h^m$ = id.  This conjecture is shown to be true for specific cases where $u$ is an even permutation and $m=3$; see Proposition \ref{m=3}.

%3.5.1 Lemma:   m even and u single cycle -> there is h  s.t. h^m = 1 = (uh)^m
\begin{lem}\label{m even u single} \rm
If $u$ is a cycle in $S_n$ of length $k$ and $m$ is a positive even integer, then there exists an $h \in S_n$ such that $ (uh)^m = h^m = $ id.  Thus $h \in \tilde{G}_m(u)$.

\begin{proof}  Let $u = (a_1 a_2 \ldots a_k)$, and
$$\begin{array}{rl}
h = & \left\{ 
		\begin{array}{ll} 
		(a_1 \ a_k) (a_2 \ a_{k-1}) \ldots (a_{\frac{k+1}{2}}) & \text{for } k \text{ odd} \\
		(a_1 \ a_k) (a_2 \ a_{k-1}) \ldots (a_{\frac{k}{2}} \ a_{\frac{k}{2}+1}) & \text{for } k \text{ even} 
     		\end{array} \right. \\
\end{array}$$
Now for $k$ odd, $h$ consists of $\frac{k-1}{2}$ transposition(s) and $n-k+1$ fixed point(s), and for $k$ even, $h$ consists of $\frac{k}{2}$ transposition(s) and $n-k$ fixed point(s).  

Thus $h^m = $ id, since $m$ is even.  
\newline Now consider
\[
\begin{array}{rl}
uh = & \left\{ 
		\begin{array}{ll} 
		(a_1) (a_2 \ a_k) (a_3 \ a_{k-1}) \ldots (a_{\frac{k+1}{2}} \ a_{\frac{k+3}{2}}) & \text{for } k \text{ odd} \\
		(a_1) (a_2 \ a_k) (a_3 \ a_{k-1}) \ldots (a_{\frac{k}{2}+1}) & \text{for } k \text{ even} 
     		\end{array} \right. \\
\end{array}
\]
for $k$ odd, $uh$ consists of $\frac{k-1}{2}$ transposition(s) and $n-k+1$ fixed point(s), and \newline for $k$ even, $uh$ consists of $\frac{k}{2}-1$ transposition(s) and $n-k+2$ fixed point(s).  
\newline \indent Thus $(uh)^m = $ id, since $m$ is even.  
Therefore $ (uh)^m = h^m = $ id, and $h \in \tilde{G}_m(u)$.

\end{proof}
\end{lem}

%3.5.2 Example of Lemma:  A) cycle is even length, B) cycle is odd length
\begin{ex} \rm In this example we explicitly construct the $h \in \tilde{G}_m(u)$ as described in the proof of Lemma \ref{m even u single} for both an even $k$ value (done in A) and an odd $k$ value (done in B).  Let $m=2s$ where $s$ is any positive integer and let $S_n=S_9$.
\newline  

\noindent{\bf A)}  If $u = (3 \ 5 \ 2 \ 1 \ 4 \ 7) = (a_1 \ a_2 \ a_3 \ a_4 \ a_5 \ a_6)$, so $k = 6$.  Then $h$ must be \newline
\indent $h = (a_1 \ a_6) (a_2 \ a_5) (a_3 \ a_4) = (3 \ 7) (5 \ 4) (2 \ 1)$, so \newline
\indent $uh =  (3 \ 5 \ 2 \ 1 \ 4 \ 7)(3 \ 7) (5 \ 4) (2 \ 1) = (3) (5 \ 7)(2 \ 4)(1) = (2 \ 4)(5 \ 7)$ \newline
\indent thus $h^m = (h^2)^s = \text{id} = ((uh)^2)^s = (uh)^m$. \newline
\noindent {\bf B)} If $u = (3 \ 5 \ 2 \ 1 \ 4 \ 7 \ 9) $ then $k = 7$, and $h$ must be \newline
\indent$h =  (3 \ 9) (5 \ 7) (2 \ 4)(1)$ or $h =  (3 \ 9) (5 \ 7) (2 \ 4)$, so

$uh =  (3 \ 5 \ 2 \ 1 \ 4 \ 7 \ 9)(3 \ 9) (5 \ 7) (2 \ 4) = (3)(5 \ 9)(2 \ 7)(1 \ 4) = (1 \ 4)(2 \ 7)(5 \ 9) $

thus $h^m = (h^2)^s = \text{id} = ((uh)^2)^s = (uh)^m$.
\end{ex}
%\newpage

%3.5.3 Proposition:   m even  -> H is not empty  (and h^m = 1)
\begin{prop}\label{notempty} \rm
If $m$ is an even positive integer, then $\tilde{G}_m(u) \neq \emptyset$. \newline 
\indent Specifically, there exists an $h \in \tilde{G}_m(u)$ such that $h^m = $ id.

\begin{proof}\rm
Let $u \in S_n$ such that $u=\sigma_1 \sigma_2 \ldots \sigma_l$ is its cycle decomposition (so all the $\sigma_i$ move disjoint subsets of ${1,2,\ldots,n}$).  Apply Lemma \ref{m even u single} to each $\sigma_i$ to find a corresponding $h_i$ such that $h_i^m = (\sigma_i h_i)^m = $ id.  Now $\sigma_i$ and $h_i$ move the same subset of ${1,2,\ldots,n}$ (or $h_i$ moves one less element than $\sigma_i$ moves).  Thus, not only are all the $h_i$ disjoint from each other, but $h_i$ is disjoint from $\sigma_j$ provided $i \neq j$.  Thus letting $h = h_1 h_2 \ldots h_l$, we get \newline
\indent $h^m = \ (h_1 \ldots h_l)^m\ =\ h_1^m  \ldots h_l^m\ =\ \text{id}^l \ =\ \text{id}$, and
$$\begin{array}{rcl}
(uh)^m & = \ ( \sigma_1 \sigma_2 \ldots \sigma_l h_1 h_2 \ldots h_l )^m  & = \ ( \sigma_1 h_1 \sigma_2 h_2 \ldots \sigma_l h_l )^m \\
 & & = \ (\sigma_1 h_1)^m (\sigma_2 h_2)^m \ldots (\sigma_l h_l)^m \ =\ \text{id}\ \\
\end{array} $$
Thus $(uh)^m = h^m$ implies $h$ is in $\tilde{G}_m(u)$ and $\tilde{G}_m(u) \neq \emptyset$.
\end{proof}
\end{prop}

%3.5.4 Example of Prop:  m even and u not just one cycle

\begin{ex} \rm  In this example we explicitly construct the $h \in \tilde{G}_m(u)$ as described in the proof of Proposition \ref{notempty}. 
 Let $m=2s$ where $s \in \BN$, let $S_n=S_{15}$, and let  \newline 
\indent $u = (3 \ 5 \ 2 \ 1 \ 4 \ 7)(6 \ 9)( 11 \ 14 \ 13) (12 \ 10 \ 8 \ 15) = \sigma_1 \sigma_2 \sigma_3 \sigma_4$.  Then we construct \newline
\indent $h_1 =  (3 \ 7) (5 \ 4) (2 \ 1)$, $h_2 = (6 \ 9)$, $h_3 = (11 \ 13)(14)$, and $h_4 = (12 \ 15)(10 \ 8)$, so \newline
\indent $h = (3 \ 7) (5 \ 4) (2 \ 1)(6 \ 9)(11 \ 13)(12 \ 15)(10 \ 8)$.  Then
\[\begin{array}{rl}
uh & =  (3)(5 \ 7)(2 \ 4)(1) \ (6)(9) \ (11)(14 \ 13) \ (12)(10 \ 15)(8)\\
 & = (2 \ 4)(5 \ 7)(10 \ 15)(13 \ 14)\\
\end{array}\]
Thus $h^m = (h^2)^s = \text{id}^s = \text{id}$ and similarily $(uh)^m = $ id, so $h^m = (uh)^m$.
\end{ex}

%3.5.5 Prop:  m = 3 and even u is single cycle, so odd length
\begin{prop}\label{m=3} \rm If $u$ is an even permutation of $S_n$, then $\tilde{G}_3(u)$ is not empty in the following cases:\newline
(1)  $u$ is a single cycle of odd length $k = 2c+1$, where $c \in \BN$ \newline
(2)  $u$ is the product of a transposition and a disjoint cycle of even length \newline
(3)  $u$ is the product of a 4-cycle and a disjoint cycle of even length at least 4 \newline
Specifically in each case, there exists an $h$ in $\tilde{G}_m(u)$ such that $h^3 = $ id.

\begin{proof}  Consider Case (1), where $u = (a_1 \ldots a_k)$ is an even permutation.  Then let $b = \left\lfloor \frac{k}{4} \right\rfloor$.  Now construct $h$ to be 
$$\begin{array}{rl}
h = & \left\{ \begin{array}{r}
	(a_k \ a_{k-1} \ a_{k-2})(a_1)(a_2 \ a_{k-3} \ a_3) (a_4)(a_5 \ a_{k-4} \ a_6) (a_7)\ldots(a_{3b-1}) \ \ \ \ \ \ \ \ \ \ \ \ \\
	\text{for } k=1\text{ mod}4 \\
	(a_k \ a_{k-1} \ a_{k-2})(a_1)(a_2 \ a_{k-3} \ a_3) (a_4)(a_5 \ a_{k-4} \ a_6) (a_7)\ldots(a_{3b-1} \ a_{3b+1} \ a_{3b}) \\
	\text{for } k=3\text{ mod}4 \\
	\end{array}
	\right. \\
\end{array}$$
\indent Then $h^3 = $ id and
$$\begin{array}{rl}
uh = & \left\{ \begin{array}{r}
	(a_k)(a_{k-1})(a_{k-2} \ a_1 \ a_2)(a_3)(a_{k-3} \ a_4 \ a_5)(a_6)\ldots(a_{3b} \ a_{3b-2} \ a_{3b-1}) \\
		 \text{for } k=1\text{ mod}4 \\
	(a_k)(a_{k-1})(a_{k-2} \ a_1 \ a_2)(a_3)(a_{k-3} \ a_4 \ a_5)(a_6)\ldots(a_{3b+1})  \ \ \ \ \ \ \ \ \ \ \ \ \\
		 \text{for } k=3\text{ mod}4 \\
	\end{array}
	\right. \\
\end{array}$$
\indent so $(uh)^3 = $ id.  Thus $h^3=(uh)^3$ and $h \in \tilde{G}_3(u)$. \newline

\noindent Now consider Case (2), where $u=\sigma_1\sigma_2 = (a_1 \ a_2)(a_3 \ldots a_{2r})$, and $r>1$ is an integer.  Notice $u$ is an even permutation since $\sigma_1$ is odd and $\sigma_2$ is odd.
Construct $h$ to be 
$$\begin{array}{rl}
h = & \left\{ \begin{array}{ll}
	(a_1 a_2 a_3)(a_4 a_{2r} a_5) (a_6)(a_7 a_{2r-1} a_8)(a_9)\ldots(a_{\frac{3r}{2}+1}) 
		& \text{for } r \text{ even} \\
	(a_1 a_2 a_3)(a_4 a_{2r} a_5) (a_6)(a_7 a_{2r-1} a_8)(a_9)\ldots(a_{\frac{3r-1}{2}} \ a_{\frac{3r+3}{2}} \ a_{\frac{3r+1}{2}}) 
		& \text{for } r \text{ odd} \\
	\end{array}
	\right. \\
\end{array}$$
\indent Then $h^3 = $ id and
$$\begin{array}{rl}
uh = & \left\{ \begin{array}{r}
	(a_1)(a_2 a_4 a_3)(a_5)(a_6 a_7 a_{2r})(a_9 a_{10} a_{2r-1})(a_{11})\ldots(a_{\frac{3r}{2}} \ a_{\frac{3r}{2}+1} \ a_{\frac{3r}{2}+2}) \\
		 \text{for } r \text{ even} \\
	(a_1)(a_2 a_4 a_3)(a_5)(a_6 a_7 a_{2r})(a_9 a_{10} a_{2r-1})(a_{11})\ldots(a_{\frac{3r+3}{2}}) \ \ \ \ \ \ \ \ \ \ \ \ \ \\
		\text{for } r \text{ odd} \\
	\end{array}
	\right. \\
\end{array}$$
\indent so $(uh)^3 = $ id.  Thus $h^3=(uh)^3$ and $h \in \tilde{G}_3(u)$. \newline

\noindent Now consider Case (3), where $u=\sigma_1\sigma_2 = (a_1 a_2 a_3 a_4)(a_5 \ldots a_{2r})$, and $r>3$ is an integer.  Notice $u$ is an even permutation since $\sigma_1$ is odd and $\sigma_2$ is odd.
Construct $h$ to be 
\[
\begin{array}{rl}
h = & \left\{ \begin{array}{r}
	(a_2 a_1 a_5) (a_6)(a_7 a_4  a_3)(a_8 a_{2r} a_9)(a_{2r-1})(a_{10} a_{2r-2} a_{11})(a_{2r-3})\ldots(a_{r+4}) \\
		 \text{for } r \text{ even} \\
	(a_2 a_1 a_5) (a_6)(a_7 a_4  a_3)(a_8 a_{2r} a_9)(a_{2r-1})(a_{10} a_{2r-2} a_{11})(a_{2r-3})\ldots \ \ \ \ \ \ \ \ \\
\ldots(a_{r+3} a_{r+5} a_{r+4})\ \ \ \text{for } r \text{ odd} \\
	\end{array}
	\right. \\
\end{array}
\]
\indent Then $h^3 = $ id and
\[
\begin{array}{rl}
uh = & \left\{ \begin{array}{r}
	(a_2)(a_1 a_6 a_7)(a_4)(a_5 a_3 a_8)(a_9)(a_{10} a_{2r-1} a_{2r}) (a_{11})(a_{12} a_{2r-3} a_{2r-2}) \ldots\\
\ldots (a_{r+4} a_{r+5} a_{r+6}) \ \ \  \text{for } r \text{ even} \\
	(a_2)(a_1 a_6 a_7)(a_4)(a_5 a_3 a_8)(a_9)(a_{10} a_{2r-1} a_{2r}) (a_{11})(a_{12} a_{2r-3} a_{2r-2}) \ldots\\
\ldots (a_{r+5}) \ \ \  \text{for } r \text{ odd} \\
	\end{array}
	\right. \\
\end{array}
\]
\indent so $(uh)^3 = $ id.  Thus $h^3=(uh)^3$ and $h \in \tilde{G}_3(u)$.

\end{proof}
\end{prop}

%3.5.6 Example of Prop. u even perm and m = 3.

\begin{ex} \rm  In these examples we explicitly construct the $h \in \tilde{G}_3(u)$ as described in the proof of Proposition \ref{m=3} for each of the cases.  Here we let $S_n = S_{14}$. \newline
{\bf Case (1)}  Let $u = (1 \ 3 \ 7 \ 6 \ 2 \ 4 \ 5) = (a_1 \ a_2 \ a_3 \ a_4 \ a_5 \ a_6 \ a_7)$, so $k = 7 = 3\text{ mod}4$ and $b = 1$, then $h=(a_7 \ a_6 \ a_5)(a_1)(a_2 \ a_4 \ a_3) = (5 \ 4 \ 2)(1)(3 \ 6 \ 7)$,  and \newline
$uh = (1 \ 3 \ 7 \ 6 \ 2 \ 4 \ 5) \ (5 \ 4 \ 2)(1)(3 \ 6 \ 7) = (5)(4)(2 \ 1 \ 3)(6)(7) = (2 \ 1 \ 3)$. \newline
\indent Thus $(uh)^3 = h^3 $ = id. \newline
{\bf Case (2)}  Let $u = (1 \ 2)(3 \ 4 \ 5 \ 6 \ 7 \ 8 \ 9 \ 10)$, so $r=5$.  Then we construct $h$ as \newline
\indent $h \ = \ (1 \ 2 \ 3)(4 \ 10 \ 5)(6)(7 \ 9 \ 8) $ which implies that $ h^3 = $ id, and \newline
\indent  $uh = (1 \ 2)(3 \ 4 \ 5 \ 6 \ 7 \ 8 \ 9 \ 10) \ (1 \ 2 \ 3)(4 \ 10 \ 5)(7 \ 9 \ 8)$ \newline
\indent $\ \ \ \ = (1)(2 \ 4 \ 3)(5)(6 \ 7 \ 10)(8)(9)$, \newline
so $(uh)^3 = $ id as well. \newline
%\newpage
\noindent{\bf Case (3) A)}  Let $u = (1 \ 2 \ 3 \ 4)(5 \ 6 \ 7 \ 8 \ 9 \ 10 \ 11 \ 12)$, so $r=6$.  Then \newline
\indent $h \ = \ (2 \ 1 \ 5)(6)(7 \ 4 \ 3)(8 \ 12 \ 9)(11)(10)$ which implies that $h^3 = $ id, and \newline
\indent $uh = (1 \ 2 \ 3 \ 4)(5 \ 6 \ 7 \ 8 \ 9 \ 10 \ 11 \ 12) \ (2 \ 1 \ 5)(6)(7 \ 4 \ 3)(8 \ 12 \ 9)(11)(10)$ \newline
\indent $\ \ \ \  = (2)(1 \ 6 \ 7)(4)(5 \ 3 \ 8)(9)(10 \ 11 \ 12) $, \newline
so $(uh)^3 = $ id as well. \newline
{\bf Case (3) B)}Let $u = (1 \ 2 \ 3 \ 4)(5 \ 6 \ 7 \ 8 \ 9 \ 10 \ 11 \ 12 \ 13 \ 14)$, so $r=7$.  Then \newline
\indent $h \ = \ (2 \ 1 \ 5)(6)(7 \ 4 \ 3)(8 \ 14 \ 9)(13)(10 \ 12 \ 11) $ which implies that $h^3 = $ id, and \newline
\indent $uh  = (1 \ 2 \ 3 \ 4)(5 \ 6 \ 7 \ 8 \ 9 \ 10 \ 11 \ 12 \ 13 \ 14) \ (2 \ 1 \ 5)(6)(7 \ 4 \ 3)(8 \ 12 \ 9)(13)(10 \ 12 \ 11)$ \newline \indent $\ \ \ \ = (2)(1 \ 6 \ 7)(4)(5 \ 3 \ 8)(9)(10 \ 13 \ 14)(11)(12) $ \newline
so $(uh)^3 = $ id as well.
\end{ex}

%***My hope is to construct some kind of inductive proof to be able to conclude that if $m=3$ and $u$ is of even parity, then $H_m(u) \neq \emptyset$. 

We end this section with some open questions.
%Questions 3.5.7
\begin{qu}\label{3.5.7}  If $\tilde{G}_m^m(u)$ is not empty, does that imply that $\tilde{G}_m^m(u)$ contains the identity element? \end{qu}
%Question 3.5.8
\begin{qu} If Questions \ref{3.5.7} is not true in general, is it at least true when $m$ is an odd positive integer? \end{qu}

%===========================================================================

%-------------------------------------Section 4: Indicators of $D(S_n)$ for some $n$--------------------------

\section{Indicators of $D(S_n)$ for some $n$} 
%Chapter 4

The $m^{\text{th}}$ Frobenius-Schur indicators of a character from each irreducible I-equivalent character class of $D(S_n)$ are presented in this section for $n \leq 6$.  Our main result is stated in Theorem 4.0 %\ref{my thm}
 below.  It indicates that Scharf's theorem, proving all higher indicators of $S_n$ are non-negative integers, may be plausible to extend to $D(S_n)$. \newline

\noindent{\bf Theorem 4.0} \it \ %\begin{thm}\label{my thm} 
All higher Frobenius-Schur indicators for $D(S_n)$ for $3 \leq n \leq 10$ are non-negative integers.  \rm
\begin{proof}  See the indicator value Tables 4.1, 4.2, 4.3, 4.4, for $n \leq 6$ and all tables in Apendix B for $7 \leq n \leq 10$.
\end{proof}
%\end{thm}
%\subsection*{}
\setcounter{thm}{1}
All indicators of $D(S_n)$ were calculated in GAP by using a series of functions equivalent to the formula from Proposition \ref{mine}
 \begin{equation} \label{calc}\nu_m(\chi_{i.j}) = \frac{1}{|C_{S_n}(u_i)|} \, \sum_{y\in \tilde{G}_m^m(u_i)} \Gamma_m(u_i, y)\eta_j(y) \end{equation}\normalsize
 where $u_i$ is a representative of a conjugacy class in $S_n$, $\eta_j$ is an irreducible character of $C_{S_n}(u_i)$ the centralizer of $u_i$ in $S_n$, and $\chi_{i.j}$ is the irreducible character of $D(S_n)$ induced up from $\eta_j$ as described in Lemma \ref{repsofDG}. Recall also that $\tilde{G}^m_m(u)$ and $\Gamma_m(u_i, y)$ were given in Definitions \ref{Gmm} and \ref{gamma}.  Character tables containing all the characters $\eta_j$ of a given centralizer are provided in Appendix A.  

For the remainder of this paper, we use $()$ to notate the identity element of $S_n$.

%4.1 
\subsection{}
\bf Indicators of $D(S_3)$ \rm   %-------------------4.1-------------------

In this section, we give details leading up to the use of Equation \ref{calc} to find all the irreducible I-equivalent character classes of $D(S_3)$ and their indicator values.  To use Equation \ref{calc}, we first choose conjugacy class representatives of $S_3$ so we can look at the character tables of their centralizers.

The conjugacy class representatives $u_i$ used in these calculations and their centralizers $C_{S_3}(u_i)$ are:
$$
\begin{array} {ll}
\begin{array}{r| c }
i & u_i \\ \hline 
1 &  ()  \\
2 &  (1,2) \\
3 & (1,2,3) \\ \hline 
\end{array} &
	\begin{array}{l} \\
	C_{S_3}(u_{1}) = S_3 \text{ is not abelian.}\\
	C_{S_3}(u_{2}) = \langle (1,2)\rangle \cong C_2 \text{ is a cyclic abelian group.} \\
	C_{S_3}(u_{3}) = \langle (1,2,3) \rangle \cong C_3 \text{ is a cyclic abelian group.} \\
	\end{array}
\end{array}
$$

Next we need to determine which of the higher indicators we are interested in computing.  The exponent of $S_3$ is 6, so according to Theorem \ref{divisors of exp}, as long as the indicator $\nu_m(\chi)$ is an integer when $m$ is a divisior of 6, that is when $m = 1,2,3,\text{ and }6$, then all higher indicators will be integer valued and we will know exactly what they all are.

So we need to look at the sets $\tilde{G}_m^m$ that we sum over, and the corresponding ``coefficients" $\Gamma_m(u,y)$ for each element $y$ in $\tilde{G}_m^m$.  Below we list the sets $\tilde{G}_m(u_i)$ which give a better understanding of how each element in $\tilde{G}_m^m(u_i)$ is found, and how the value for $\Gamma_m(u_i, y)$ is found for each $y$ in $\tilde{G}_m^m(u_i)$.  Next to each set $\tilde{G}_m(u_i)$, we list the set $\tilde{G}_m^m(u_i)$ with the corresponding value of $\Gamma_m(u,y)$ paired with and preceding each element in $\tilde{G}_m^m(u_i)$.  So if $\tilde{G}_m^m(u_i) = \{y,...,z\}$ then we will display this set by writing ``$\tilde{G}_m^m(u_i)$ will consist of $\{ [\Gamma_m(u,y),y],...,[\Gamma_m(u,z),z]\}$."  Thus we have: 

{\bf(1)} For $u_1 = ()$:  \newline
$\tilde{G}_m(u_1) = S_3 \text{ for all }m$, so \hfill $\tilde{G}_1^1(u_1) \text{ consists of } \{ [ 1, () ], [3, (1,2)], [2,(1,2,3)]\}$.

\hspace{2.35in} $\tilde{G}_2^2(u_1) \text{ consists of } \{ [ 4, () ],  [2,(1,2,3)]\}$. 

\hspace{2.35in} $\tilde{G}_3^3(u_1) \text{ consists of } \{ [ 3, () ], [3, (1,2)]\}$, and

\hspace{2.35in} $\tilde{G}_6^6(u_1) \text{ consists of } \{ [ 6, () ]\}$. 

{\bf(2)} For $u_2=(1,2)$: \newline
$\tilde{G}_1(u_2) = \emptyset $, \hspace{1.63in} $ \tilde{G}_1^1(u_2) = \emptyset$. \newline
$\tilde{G}_2(u_2) = \{ (), (1,2)\} $, \hspace{0.99in} $\tilde{G}_2^2(u_2) \text{ consists of }\{ [ 2, () ]\}$. \newline
$\tilde{G}_3(u_2) = \emptyset $,  \hspace{1.63in} $\tilde{G}_3^3(u_2)=\emptyset $. \newline
$\tilde{G}_6(u_2) = S_3 $,   \hspace{1.55in} $\tilde{G}_6^6(u_2) \text{ consists of } \{ [ 6, () ]\}$. 

{\bf(3)} For $u_3 = (1,2,3)$: \newline
$\tilde{G}_1(u_3) = \emptyset $, \hspace{1.63in} $\tilde{G}_1^1(u_3)=\emptyset$.\newline
$\tilde{G}_2(u_3) = \{ (2,3), (1,2), (1,3) \} $, \hspace{.32in} $\tilde{G}_2^2(u_3) \text{ consists of } \{ [ 3, () ] \}$.\newline
$\tilde{G}_3(u_3) = \{ (), (1,2,3), (1,3,2) \} $,\hspace{.3in} $\tilde{G}_3^3(u_3) \text{ consists of } \{ [ 3, () ] \}.$\newline
$\tilde{G}_6(u_3) = S_3 $, \hspace{1.56in} $\tilde{G}_6^6(u_3) \text{ consists of } \{ [ 6, () ] \}$.

As we can see, all the $\tilde{G}_m^m(u_2)$ are either empty or contain only the identity element, so since $C_{S_3}(u_2)$ is abelian,  Lemma \ref{I-equiv} tells us that all characters induced from $C_{S_3}(u_2)$ will be I-equivalent.  Similarily all characters induced from $C_{S_3}(u_3)$ will be I-equivalent.

Since we know $\eta_{2.1}$ is I-equivalent to $\eta_{2.2}$ by Lemma \ref{I-equiv}, and we know $\eta_{3.1}$, $\eta_{3.2}$, and $\eta_{3.3}$ are all three I-equivalent by Lemma \ref{I-equiv}, we do not really need the full character tables of $C_{S_3}(u_2)$ or $C_{S_3}(u_3)$.  We only need to know how many irreducible characters there are for each of these centralizers.

Thus using the coefficients $\Gamma$, the elements in $\tilde{G}_m^m(u_i)$, and the character tables of $C_{S_3}(u_i)$ with Equations \ref{calc}, we find:
\begin{prop} The irreducible characters of $D(S_3)$ have the following 4 distinct I-equivalent character classes:
$$[\ \chi_{1.1}\ ],\ \ [\ \chi_{1.2},\chi_{3.1},\chi_{3.2},\chi_{3.3}\ ],\ \ [\ \chi_{1.3}\ ],\ \ [\ \chi_{2.1},\chi_{2.2}\ ].$$ \end{prop}

The indicator values of these irreducible I-equivalent characters are displayed in Table 4.1 below.
\begin{table}[ht]
\caption{$D(S_3)$ indicators: (exponent 6)}
$ \begin{array}{r|cccc} \hline 
m =  & 1 
 & 2 
 & 3 
 & 6 
\rule[-7pt]{0pt}{20pt} \\ \hline 
\nu_m(\chi_{1.1}) & 0 & 1 & 0 & 1 \rule[0pt]{0pt}{13pt} \\ 
\nu_m(\chi_{1.2}) & 0 & 1 & 1 & 2 \\ 
\nu_m(\chi_{1.3}) & 1 & 1 & 1 & 1 \\ 
\nu_m(\chi_{2.1}) & 0 & 1 & 0 & 3 \rule[-7pt]{0pt}{5pt} \\ 
\hline 
\end{array} $
\label{table:S3} 
\end{table}

%4.2
\subsection{}
\bf Indicators of $D(S_4)$ \rm   %-------------------4.2-------------------

In this section, we give details leading up to the use of Equation \ref{calc} to find all the irreducible I-equivalent character classes of $D(S_4)$ and their indicator values.  Recall Equation \ref{calc}:
 $$\nu_m(\chi_{i.j}) = \frac{1}{|C_{S_n}(u_i)|} \, \sum_{y\in \tilde{G}_m^m(u_i)} \Gamma_m(u_i, y)\eta_j(y) $$\normalsize
 where $u_i$ is a representative of a conjugacy class in $S_n$, $\eta_j$ is an irreducible character of $C_{S_n}(u_i)$ the centralizer of $u_i$ in $S_n$, and $\chi_{i.j}$ is the irreducible character of $D(S_n)$ induced up from $\eta_j$ as described in Lemma \ref{repsofDG}. $\tilde{G}^m_m(u)$ and $\Gamma_m(u_i, y)$ were given in Definitions \ref{Gmm} and \ref{gamma}.

To begin, we first choose conjugacy class representatives of $S_4$ so we can look at the character tables of their centralizers. 

The conjugacy class representatives $u_i$ used in these calculations and their centralizers $C_{S_4}(u_i)$ are:
$$
\begin{array} {ll}
\begin{array}{r| c }
i & u_i \\ \hline 
1 &  ()  \\
2 &  (1,2) \\
3 & (1,2)(3,4) \\
4 & (1,2,3) \\
5 & (1,2,3,4) \\ \hline 
\end{array} &
	\begin{array}{l} \\
	C_{S_4}(u_{1}) = S_4 \text{ is not abelian.}\\
	C_{S_4}(u_{2}) = \langle (1,2), (3,4)\rangle \cong C_2 \times C_2 \text{ is abelian.} \\
	C_{S_4}(u_{3}) = \langle (1,2), (1,3)(2,4), (3,4)\rangle \cong D_8 \text{ is not abelian.} \\
	C_{S_4}(u_{4}) = \langle (1,2,3) \rangle \cong C_3 \text{ is cyclic abelian.} \\
	C_{S_4}(u_{5}) = \langle (1,2,3,4)\rangle \cong C_4 \text{ is abelian.} \\
	\end{array}
\end{array}
$$

\noindent{\bf(1)} First we consider when $u_1 = ()$.

The sets $\tilde{G}_m^m(u_1)$ that we sum over, and the corresponding coefficients $\Gamma_m(u_1,y)$ for each element $y$ in $\tilde{G}_m^m$ are listed below, with the corresponding value of $\Gamma_m(u_1,y)$ preceding each element in $\tilde{G}_m^m(u_1)$. 

 Since $\tilde{G}_m(u_1) = S_4$ for all $m$, we have: \newline
$\tilde{G}_1^1(u_1)$ consists of $\{ [ 1, () ], [ 6, (1,2) ], [ 3, (1,2)(3,4) ], [ 8, (1,2,3) ], [ 6, (1,2,3,4) ] \}$; \newline
$\tilde{G}_2^2(u_1)$ consists of $\{ [ 10, () ], [ 6, (1,2)(3,4) ], [ 8, (1,2,3) ]\}$; \newline
$\tilde{G}_3^3(u_1)$ consists of $\{ [ 9, () ], [ 6, (1,2) ], [ 3, (1,2)(3,4) ], [ 6, (1,2,3,4) ] \}$; \newline
$\tilde{G}_4^4(u_1)$ consists of $\{ [ 16, () ],  [ 8, (1,2,3) ]\}$;\newline
$\tilde{G}_6^6(u_1)$ consists of $\{ [ 18, () ], [ 6, (1,2)(3,4) ] \}$;\newline
$\tilde{G}_{12}^{12}(u_1)$ consists of $\{ [ 24, () ] \}$. \newline

For the rest of the $u=u_i$, we will list the sets $\tilde{G}_m(u)$ which give a better understanding of how each element in $\tilde{G}_m^m(u)$ and values for $\Gamma_m$ were found.  Then we list the sets $\tilde{G}_m^m(u)$ with the corresponding value of $\Gamma_m(u,y)$ preceding each element in $\tilde{G}_m^m(u)$ in the same fashion as in Section 4.1.

In the pages that follow we will see that for $i=2,4,$ and 5 all $\tilde{G}_m^m(u_i)$ are either empty or only contain the identity element.  So Lemma \ref{I-equiv} applies to $C_{S_4}(u_i)$, for $i=2,4,$ and 5 since each of these centralizer groups is abelian.  Thus we do not really need the character tables of these centralizers, we only need to know that there are 4 irreducible characters of $C_{S_4}(u_2)$, 3 irreducible characters of $C_{S_4}(u_4)$, and 4 irreducible characters of $C_{S_4}(u_5)$.\newline

\noindent {\bf(2)} Now consider when $u_2 = (1,2)$. 
$$\begin{array}{ll}
\ \tilde{G}_1(u_2) = \emptyset, & \ \tilde{G}_1^1(u_2) = \emptyset.  \\
\ \tilde{G}_2(u_2) = \{ (), (3,4), (1,2), (1,2)(3,4)\}, & \ 
	\tilde{G}_2^2(u_2) \text{ consists of }\{ [ 4, () ] \}. \\
\ \tilde{G}_3(u_2) = \emptyset, & \ \tilde{G}_3^3(u_2) = \emptyset.\\
\begin{array}{l}
\tilde{G}_4(u_2) = \{ (), (3,4), (1,2), (1,2)(3,4), \\ \ \ \ (1,3)(2,4), (1,3,2,4), (1,4,2,3),\\ \ \ \  (1,4)(2,3) \},
\end{array} & \begin{array}{l}
\tilde{G}_4^4(u_2) \text{ consists of }\{ [ 8, () ] \}. \\ \\ \\ \end{array} \\
\begin{array}{l}
\tilde{G}_6(u_2) = \{ (), (3,4), (2,3), (2,4), (1,2),  \\ \ \ \ (1,2)(3,4), (1,2,3), (1,2,4), 
  (1,3,2),  \\ \ \ \ (1,3), (1,4,2), (1,4) \}, 
\end{array} & \begin{array}{l}
\tilde{G}_6^6(u_2) \text{ consists of }\{ [ 12, () ] \}. \\ \\ \\ \end{array} \\
\tilde{G}_{12}(u_2) = S_4, & \ \tilde{G}_{12}^{12}(u_2) \text{ consists of }\{ [ 24, () ] \}.
\end{array}$$
\newpage

\noindent {\bf(3)} Now consider when $u_3 = (1,2)(3,4)$.

$$\begin{array}{ll}
\ \tilde{G}_1(u_3) = \emptyset, & \ \tilde{G}_1^1((u_{3}) = \emptyset  \\
\begin{array}{l}
 \tilde{G}_2(u_3) = \{ (), (3,4), (1,2), (1,2)(3,4), \\ \ \ \  (1,3)(2,4), (1,4)(2,3), 
  (1,3,2,4), \\ \ \ \ (1,4,2,3)\}, \end{array} & \begin{array}{l} 
	\tilde{G}_2^2(u_3) \text{ consists of }\{  [ 6, () ], \\ \ \ \ \ \ \ \ \ \ [ 2, (1,2)(3,4) ] \}. \\ \end{array} \\
\begin{array}{l}
 \tilde{G}_3(u_3) = \{ (2,3,4), (2,4,3), (1,2,3), \\ \ \ \ (1,2,4), (1,3,2), (1,3,4), \\ \ \ \
  (1,4,2), (1,4,3) \}, \end{array} & \begin{array}{l} \tilde{G}_3^3(u_3) \text{ consists of } \{ [ 8, () ] \}.\\ \\ \end{array} \\ 
\begin{array}{l}
\tilde{G}_4(u_3) = \{ (), (3,4), (2,3), (2,4), \\ \ \ \ (1,2), (1,2)(3,4), (1,2,3,4), \\ \ \ \
  (1,2,4,3), (1,3,4,2), (1,3), \\ \ \ \ (1,3)(2,4), (1,3,2,4), (1,4,3,2), \\ \ \ \ (1,4),   (1,4,2,3), (1,4)(2,3) \},
\end{array} & \begin{array}{l}
\tilde{G}_4^4(u_3) \text{ consists of }\{  [ 16, () ] \}. \\ \\ \\ \\ \\ \end{array} \\
\begin{array}{l}
\tilde{G}_6(u_3) = \{(), (3,4), (2,3,4), (2,4,3), \\ \ \ \ (1,2), (1,2)(3,4), (1,2,3), \\ \ \ \
  (1,2,4), (1,3,2), (1,3,4), \\ \ \ \ (1,3)(2,4), (1,4,2), (1,4,3), \\ \ \ \ (1,4)(2,3), 
  (1,3,2,4), (1,4,2,3)\}, 
\end{array} & \begin{array}{l}
\tilde{G}_6^6(u_3) \text{ consists of }\{ [ 14, () ], \\ \ \ \ \ \ \ \ \ \ [ 2, (1,2)(3,4) ] \}. \\ \\ \\ \\ \end{array} \\
\tilde{G}_{12}(u_3) = S_4, & \ \tilde{G}_{12}^{12}(u_3) \text{ consists of }\{ [ 24, () ] \}.
\end{array}$$

%\newpage

\noindent {\bf(4)} Now consider when $u_4 = (1,2,3)$.

$$\begin{array}{ll}
\ \tilde{G}_1(u_4) = \emptyset, & \ \tilde{G}_1^1((u_{4}) = \emptyset  \\

\begin{array}{l}
 \tilde{G}_2(u_4) = \{ (2,3), (1,2), (1,3) \}, \end{array} & \begin{array}{l} 
	\tilde{G}_2^2(u_4) \text{ consists of }\{  [ 3, () ] \}.  \end{array} \\

\begin{array}{l}
 \tilde{G}_3(u_4) = \{ (), (2,4,3), (1,2,3), \\ \ \ \ (1,3,2), (1,3,4), (1,4,2) \}, \end{array} & \begin{array}{l} \tilde{G}_3^3(u_4) \text{ consists of } \{ [ 6, () ] \}.\\ \\ \end{array} \\ 

\begin{array}{l}
\tilde{G}_4(u_4) = \{ (3,4), (2,3), (2,4), (1,2), \\ \ \ \ (1,2,3,4), (1,2,4,3), (1,3,4,2), \\ \ \ \
  (1,3), (1,3,2,4), (1,4,3,2), \\ \ \ \ (1,4), (1,4,2,3) \},
\end{array} & \begin{array}{l}
\tilde{G}_4^4(u_4) \text{ consists of }\{  [ 12, () ] \}. \\ \\ \\ \\ \end{array} \\

\begin{array}{l}
\tilde{G}_6(u_4) = \{(), (2,3), (2,3,4), (2,4,3), \\ \ \ \ (1,2), (1,2)(3,4), (1,2,3), 
  (1,2,4), \\ \ \ \ (1,3,2), (1,3), (1,3,4), (1,3)(2,4),\\ \ \ \ (1,4,2), (1,4,3), (1,4)(2,3) \}, 
\end{array} & \begin{array}{l}
\tilde{G}_6^6(u_4) \text{ consists of }\{ [ 15, () ] \}. \\ \\ \\ \\ \end{array} \\

\tilde{G}_{12}(u_4) = S_4, & \ \tilde{G}_{12}^{12}(u_4) \text{ consists of }\{ [ 24, () ] \}.
\end{array}$$
 \newline

\newpage

\noindent {\bf(5)} Finally, consider when $u_5 = (1,2,3,4)$.

$$\begin{array}{ll}
\ \tilde{G}_1(u_5) = \emptyset, & \ \tilde{G}_1^1((u_{5}) = \emptyset  \\

\begin{array}{l}
 \tilde{G}_2(u_5) = \{ (2,4), (1,2)(3,4), (1,3), \\ \ \ \ (1,4)(2,3) \}, \end{array} & \begin{array}{l} 
	\tilde{G}_2^2(u_5) \text{ consists of }\{  [ 4, () ] \}.  \\ \\ \end{array} \\

\ \tilde{G}_3(u_5) = \emptyset, & \ \tilde{G}_3^3(u_5) = \emptyset. \\ 

\begin{array}{l}
\tilde{G}_4(u_5) = \{ (), (2,4), (1,2)(3,4), (1,2,3,4), \\ \ \ \ (1,3), (1,3)(2,4), 
  (1,4,3,2), (1,4)(2,3) \},
\end{array} & \begin{array}{l}
\tilde{G}_4^4(u_5) \text{ consists of }\{  [ 8, () ] \}. \\ \\  \end{array} \\

\begin{array}{l}
\tilde{G}_6(u_5) = \{(3,4), (2,3), (2,4,3), (2,4), \\ \ \ \ (1,2), (1,2)(3,4), (1,3,2), 
  (1,3), \\ \ \ \ (1,4,2), (1,4,3), (1,4), (1,4)(2,3) \}, 
\end{array} & \begin{array}{l}
\tilde{G}_6^6(u_5) \text{ consists of }\{ [ 12, () ] \}. \\ \\ \\  \end{array} \\

\tilde{G}_{12}(u_5) = S_4, & \ \tilde{G}_{12}^{12}(u_5) \text{ consists of }\{ [ 24, () ] \}.
\end{array}$$

%\newpage

Thus using the coeffieicents $\Gamma$, the elements in $\tilde{G}_m^m(u_i)$, and the character tables of $C_{S_4}(u_i)$ with Equations \ref{calc}, we find:
\begin{prop} The irreducible characters of $D(S_4)$ have the following 7 distinct I-equivalent classes:\newline
$ [\ \chi_{1.1}\ ], \ [\ \chi_{1.2}, \chi_{1.4}, \chi_{3.1}, \chi_{3.2}, \chi_{3.3}, \chi_{3.4}\ ], \ [\ \chi_{1.3}\ ], \ [\ \chi_{1.5}\ ],$ \newline 
$[\ \chi_{2.1}, \chi_{2.2},\chi_{2.3}, \chi_{2.4}, \chi_{5.1}, \chi_{5.2}, \chi_{5.3}, \chi_{5.4}\ ], \ [\ \chi_{3.5}\ ],\
[\ \chi_{4.1}, \chi_{4.2}, \chi_{4.3}\ ]$.
\end{prop}

The indicator values of these irreducible I-equivalent characters are displayed in Table 4.2 below.
\begin{table}[ht]
\caption{$D(S_4)$ indicators: (exponent 12)}
$ 
\begin{array}{r|cccccc} \hline 
m =  & 1 
 & 2 
 & 3 
 & 4 
 & 6 
 & 12 
\rule[-7pt]{0pt}{20pt} \\ \hline 
\nu_m(\chi_{1.1}) & 0 & 1 & 0 & 1 & 1 & 1 \rule[0pt]{0pt}{13pt} \\ 
\nu_m(\chi_{1.2}) & 0 & 1 & 1 & 2 & 2 & 3 \\ 
\nu_m(\chi_{1.3}) & 0 & 1 & 1 & 1 & 2 & 2 \\ 
\nu_m(\chi_{1.5}) & 1 & 1 & 1 & 1 & 1 & 1 \\ 
\nu_m(\chi_{2.1}) & 0 & 1 & 0 & 2 & 3 & 6 \\ 
\nu_m(\chi_{3.5}) & 0 & 1 & 2 & 4 & 3 & 6 \\ 
\nu_m(\chi_{4.1}) & 0 & 1 & 2 & 4 & 5 & 8 \rule[-7pt]{0pt}{5pt} \\ 
\hline 
\end{array} 
$
\label{table:S4} 
\end{table}
%\clearpage

%4.3
\subsection{}
\bf Indicators of $D(S_5)$ \rm   %-------------------4.3-------------------

In this section, we give details leading up to the use of Equation \ref{calc} to find all the irreducible I-equivalent character classes of $D(S_5)$ and their indicator values.  Recall Equation \ref{calc}: 
 $$\nu_m(\chi_{i.j}) = \frac{1}{|C_{S_n}(u_i)|} \, \sum_{y\in \tilde{G}_m^m(u_i)} \Gamma_m(u_i, y)\eta_j(y) $$\normalsize
 where $u_i$ is a representative of a conjugacy class in $S_n$, $\eta_j$ is an irreducible character of $C_{S_n}(u_i)$ the centralizer of $u_i$ in $S_n$, and $\chi_{i.j}$ is the irreducible character of $D(S_n)$ induced up from $\eta_j$ as described in Lemma \ref{repsofDG}. $\tilde{G}^m_m(u)$ and $\Gamma_m(u_i, y)$ were given in Definitions \ref{Gmm} and \ref{gamma}.

To begin, we first choose conjugacy class representatives of $S_5$ so we can look at the character tables of their centralizers.  

The conjugacy class representatives $u_i$ used in these calculations and their centralizers $C_{S_5}(u_i)$ are:
$$
\begin{array} {ll}
\begin{array}{r| c }
i & u_i \\ \hline 
1 &  ()  \\
2 &  (1,2) \\
3 & (1,2)(3,4) \\
4 & (1,2,3) \\
5 & (1,2,3)(4,5) \\ 
6 & (1,2,3,4) \\
7 & (1,2,3,4,5) \\ \hline 
\end{array} &
	\begin{array}{l} \\
	C_{S_5}(u_{1}) = S_5 \text{ is not abelian.}\\
	C_{S_5}(u_{2}) = \langle (1,2), (3,5), (4,5)\rangle \cong  D_{12} \text{ is not abelian.} \\
	C_{S_5}(u_{3}) = \langle (1,2), (1,3)(2,4), (3,4) \rangle \cong D_8 \text{ is not abelian.} \\
	C_{S_5}(u_{4}) = \langle (1,2,3), (4,5) \rangle \cong C_6 \text{ is cyclic abelian.} \\
	C_{S_5}(u_{5}) = \langle (1,2,3), (4,5)\rangle \cong C_6 \text{ is cyclic abelian.} \\
	C_{S_5}(u_{6}) = \langle (1,2,3,4)\rangle \cong C_4 \text{ is cyclic abelian.} \\
	C_{S_5}(u_{7}) = \langle (1,2,3,4,5)\rangle \cong C_5 \text{ is cyclic abelian.} \\
	\end{array}
\end{array}
$$

{\bf(1)} When $u_1 = ()$:

The sets $\tilde{G}_m^m(u_1)$ that we sum over, and the corresponding coefficients $\Gamma_m(u_1,y)$ for each element $y$ in $\tilde{G}_m^m$ are listed below, with the corresponding value of $\Gamma_m(u_1,y)$ preceding each element in $\tilde{G}_m^m(u_1)$, in the same fashion as in Section 4.1.

 Since $\tilde{G}_m(u_1) = S_5$ for all $m$, we have: \newline
$\tilde{G}_1^1(u_{1})$ consists of $\{ [ 1, () ], [ 10, (1,2) ], [ 15, (1,2)(3,4) ], [ 20, (1,2,3) ],$\newline
\indent $ [ 20, (1,2,3)(4,5) ], [ 30, (1,2,3,4) ], [ 24, (1,2,3,4,5) ] \}$; \newline 
$\tilde{G}_2^2(u_{1})$ consists of $\{ [ 26, () ], [ 40, (1,2,3) ], [ 30, (1,2)(3,4) ], [ 24, (1,2,3,4,5) ] \}$; \newline 
$\tilde{G}_3^3(u_{1})$ consists of $\{ [ 21, () ], [ 30, (1,2) ], [ 15, (1,2)(3,4) ], [ 30, (1,2,3,4) ],$\newline
\indent $ [ 24, (1,2,3,4,5) ] \}$; \newline 
$\tilde{G}_4^4(u_{1})$ consists of $\{ [ 56, () ], [ 40, (1,2,3) ], [ 24, (1,2,3,4,5) ] \}$; \newline 
$\tilde{G}_5^5(u_{1})$ consists of $\{ [ 25, () ], [ 10, (1,2) ], [ 15, (1,2)(3,4) ], [ 20, (1,2,3) ], $\newline
\indent $[ 20, (1,2,3)(4,5) ], [ 30, (1,2,3,4) ] \}$; \newline 
$\tilde{G}_6^6(u_{1})$ consists of $\{ [ 66, () ], [ 30, (1,2)(3,4) ],  [ 24, (1,2,3,4,5) ] \}$; \newline 
$\tilde{G}_{10}^{10}(u_{1})$ consists of $\{ [ 50, () ], [ 40, (1,2,3) ], [ 30, (1,2)(3,4) ] \}$; \newline 
$\tilde{G}_{12}^{12}(u_{1})$ consists of $\{ [ 96, () ], [ 24, (1,2,3,4,5) ] \}$; \newline 
$\tilde{G}_{15}^{15}(u_{1})$ consists of $\{ [ 45, () ], [ 30, (1,2) ], [ 15, (1,2)(3,4) ], [ 30, (1,2,3,4) ] \}$; \newline 
$\tilde{G}_{20}^{20}(u_{1})$ consists of $\{ [ 80, () ], [ 40, (1,2,3) ]  \}$; \newline 
$\tilde{G}_{30}^{30}(u_{1})$ consists of $\{ [ 90, () ], [ 30, (1,2)(3,4) ] \}$; \newline 
$\tilde{G}_{60}^{60}(u_{1})$ consists of $\{ [ 120, () ] \}$.  \newline 

For the rest of the $u$'s, we will list only the sets $\tilde{G}_m^m(u)$  that we sum over, and the corresponding coefficients $\Gamma_m(u,y)$ for each element $y$ in $\tilde{G}_m^m(u)$, with the corresponding value of $\Gamma_m(u,y)$ preceding each element in $\tilde{G}_m^m(u)$ in the same fashion as above. We do not list the sets $\tilde{G}_m(u)$ as we did in Sections 4.1 and 4.2 since these sets can be very large and too long to include in this paper.

In the pages that follow we will see that for $i=5,6,$ and 7 all $\tilde{G}_m^m(u_i)$ are either empty or only contain the identity element.  So Lemma \ref{I-equiv} applies to $C_{S_5}(u_i)$, for $i=5,6,$ and 7 since each of these centralizer groups is abelian.  Thus we do not really need the character tables of these centralizers, we only need to know that there are 6 irreducible characters of $C_{S_5}(u_5)$, 4 irreducible characters of $C_{S_5}(u_6)$, and 5 irreducible characters of $C_{S_5}(u_7)$.\newline

{\bf(2)} When $u_2 = (1,2)$:

\noindent$\tilde{G}_1^1(u_{2}) = \emptyset$;\newline 
$\tilde{G}_2^2(u_{2})$ consists of $\{ [ 8, () ], [ 4, (3,4,5) ] \}$; \newline 
$\tilde{G}_3^3(u_{2}) = \emptyset$;\newline 
$\tilde{G}_4^4(u_{2})$ consists of $\{ [ 20, () ], [ 4, (3,4,5) ] \}$; \newline 
$\tilde{G}_5^5(u_{2}) = \emptyset$;\newline 
$\tilde{G}_6^6(u_{2})$ consists of $\{ [ 36, () ] \}$; \newline 
$\tilde{G}_{10}^{10}(u_{2})$ consists of $\{ [ 8, () ], [ 4, (3,4,5) ] \}$; \newline 
$\tilde{G}_{12}^{12}(u_{2})$ consists of $\{ [ 72, () ] \}$; \newline 
$\tilde{G}_{15}^{15}(u_{2}) = \emptyset$;\newline 
$\tilde{G}_{20}^{20}(u_{2})$ consists of $\{ [ 44, () ], [ 4, (3,4,5) ] \}$; \newline 
$\tilde{G}_{30}^{30}(u_{2})$ consists of $\{ [ 60, () ] \}$; \newline 
$\tilde{G}_{60}^{60}(u_{2})$ consists of $\{ [ 120, () ] \}$. \newline

 {\bf(3)} When $u_3 = (1,2)(3,4)$:

\noindent $\tilde{G}_1^1(u_{3}) = \emptyset$;\newline 
$\tilde{G}_2^2(u_{3})$ consists of $\{ [ 6, () ], [ 2, (1,2)(3,4) ] \}$; \newline 
$\tilde{G}_3^3(u_{3})$ consists of $\{ [ 8, () ] \}$; \newline 
$\tilde{G}_4^4(u_{3})$ consists of $\{ [ 32, () ] \}$; \newline 
$\tilde{G}_5^5(u_{3})$ consists of $\{ [ 8, () ] \}$; \newline 
$\tilde{G}_6^6(u_{3})$ consists of $\{ [ 38, () ], [ 2, (1,2)(3,4) ] \}$; \newline 
$\tilde{G}_{10}^{10}(u_{3})$ consists of $\{ [ 30, () ], [ 2, (1,2)(3,4) ] \}$; \newline 
$\tilde{G}_{12}^{12}(u_{3})$ consists of $\{ [ 80, () ] \}$; \newline 
$\tilde{G}_{15}^{15}(u_{3})$ consists of $\{ [ 32, () ] \}$; \newline 
$\tilde{G}_{20}^{20}(u_{3})$ consists of $\{ [ 56, () ] \}$; \newline 
$\tilde{G}_{30}^{30}(u_{3})$ consists of $\{ [ 78, () ], [ 2, (1,2)(3,4) ] \}$; \newline 
$\tilde{G}_{60}^{60}(u_{3})$ consists of $\{ [ 120, () ] \}$. \newline 

%\newpage
 {\bf(4)} When $u_4 = (1,2,3)$:

\noindent $\tilde{G}_1^1(u_{4}) = \emptyset$;\newline 
$\tilde{G}_2^2(u_{4})$ consists of $\{ [ 6, () ] \}$; \newline 
$\tilde{G}_3^3(u_{4})$ consists of $\{ [ 9, () ], [ 3, (4,5) ] \}$; \newline 
$\tilde{G}_4^4(u_{4})$ consists of $\{ [ 30, () ] \}$; \newline 
$\tilde{G}_5^5(u_{4})$ consists of $\{ [ 12, () ] \}$; \newline 
$\tilde{G}_6^6(u_{4})$ consists of $\{ [ 36, () ] \}$; \newline 
$\tilde{G}_{10}^{10}(u_{4})$ consists of $\{ [ 30, () ] \}$; \newline 
$\tilde{G}_{12}^{12}(u_{4})$ consists of $\{ [ 84, () ] \}$; \newline 
$\tilde{G}_{15}^{15}(u_{4})$ consists of $\{ [ 33, () ], [ 3, (4,5) ] \}$; \newline 
$\tilde{G}_{20}^{20}(u_{4})$ consists of $\{ [ 54, () ] \}$; \newline 
$\tilde{G}_{30}^{30}(u_{4})$ consists of $\{ [ 72, () ] \}$; \newline 
$\tilde{G}_{60}^{60}(u_{4})$ consists of $\{ [ 120, () ] \}$. \newline

 {\bf(5)} When $u_5 = (1,2,3)(4,5)$:

\noindent $\tilde{G}_1^1(u_{5}) = \emptyset$;\newline 
$\tilde{G}_2^2(u_{5})$ consists of $\{  [ 6, () ] \}$; \newline 
$\tilde{G}_3^3(u_{5}) = \emptyset$;\newline 
$\tilde{G}_4^4(u_{5})$ consists of $\{ [ 18, () ] \}$; \newline 
$\tilde{G}_5^5(u_{5}) = \emptyset$;\newline 
$\tilde{G}_6^6(u_{5})$ consists of $\{ [ 36, () ] \}$; \newline 
$\tilde{G}_{10}^{10}(u_{5})$ consists of $\{ [ 18, () ] \}$; \newline 
$\tilde{G}_{12}^{12}(u_{5})$ consists of $\{ [ 72, () ] \}$; \newline 
$\tilde{G}_{15}^{15}(u_{5}) = \emptyset$;\newline 
$\tilde{G}_{20}^{20}(u_{5})$ consists of $\{ [ 54, () ] \}$; \newline 
$\tilde{G}_{30}^{30}(u_{5})$ consists of $\{ [ 60, () ] \}$; \newline 
$\tilde{G}_{60}^{60}(u_{5})$ consists of $\{ [ 120, () ] \}$. \newline

 {\bf(6)} When $u_6 = (1,2,3,4)$:

\noindent $\tilde{G}_1^1(u_{6}) = \emptyset$;\newline 
$\tilde{G}_2^2(u_{6})$ consists of $\{ [ 4, () ] \}$; \newline 
$\tilde{G}_3^3(u_{6}) = \emptyset$;\newline 
$\tilde{G}_4^4(u_{6})$ consists of $\{ [ 24, () ] \}$; \newline 
$\tilde{G}_5^5(u_{6}) = \emptyset$;\newline 
$\tilde{G}_6^6(u_{6})$ consists of $\{ [ 36, () ] \}$; \newline 
$\tilde{G}_{10}^{10}(u_{6})$ consists of $\{ [ 12, () ] \}$; \newline 
$\tilde{G}_{12}^{12}(u_{6})$ consists of $\{ [ 72, () ] \}$; \newline 
$\tilde{G}_{15}^{15}(u_{6}) = \emptyset$;\newline 
$\tilde{G}_{20}^{20}(u_{6})$ consists of $\{ [ 56, () ] \}$; \newline 
$\tilde{G}_{30}^{30}(u_{6})$ consists of $\{ [ 60, () ] \}$; \newline 
$\tilde{G}_{60}^{60}(u_{6})$ consists of $\{ [ 120, () ] \}$.  \newline

%\newpage
 {\bf(7)} When $u_7 = (1,2,3,4,5)$:

\noindent $\tilde{G}_1^1(u_{7}) = \emptyset$;\newline 
$\tilde{G}_2^2(u_{7})$ consists of $\{ [ 5, () ] \}$; \newline 
$\tilde{G}_3^3(u_{7})$ consists of $\{ [ 5, () ] \}$; \newline 
$\tilde{G}_4^4(u_{7})$ consists of $\{ [ 30, () ] \}$; \newline 
$\tilde{G}_5^5(u_{7})$ consists of $\{ [ 10, () ] \}$; \newline 
$\tilde{G}_6^6(u_{7})$ consists of $\{ [ 35, () ] \}$; \newline 
$\tilde{G}_{10}^{10}(u_{7})$ consists of $\{ [ 25, () ] \}$; \newline 
$\tilde{G}_{12}^{12}(u_{7})$ consists of $\{ [ 80, () ] \}$; \newline 
$\tilde{G}_{15}^{15}(u_{7})$ consists of $\{ [ 35, () ] \}$; \newline 
$\tilde{G}_{20}^{20}(u_{7})$ consists of $\{ [ 50, () ] \}$; \newline 
$\tilde{G}_{30}^{30}(u_{7})$ consists of $\{ [ 75, () ] \}$; \newline 
$\tilde{G}_{60}^{60}(u_{7})$ consists of $\{ [ 120, () ] \}$. \newline

%\newpage

Thus using the coeffieicents $\Gamma$, the elements in $\tilde{G}_m^m(u_i)$, and the character tables of $C_{S_5}(u_i)$ with Equations \ref{calc}, we find:
\begin{prop} The irreducible characters of $D(S_5)$ have the following 15 distinct I-equivalent classes:\newline
$ [\ \chi_{1.1}\ ],  [\ \chi_{1.2}\ ],  [\ \chi_{1.3}, \chi_{1.5}\ ],  [\ \chi_{1.4}\ ],  [\ \chi_{1.6}\ ],  [\ \chi_{1.7}\ ],$ \newline
$ [\ \chi_{2.1}, \chi_{2.2}, \chi_{2.3}, \chi_{2.4}\ ],  [\ \chi_{2.5}, \chi_{2.6}\ ],$ %\newline
$ [\ \chi_{3.1}, \chi_{3.2}, \chi_{3.3}, \chi_{3.4}\ ],  [\ \chi_{3.5}\ ],$ \newline
$ [\ \chi_{4.1}, \chi_{4.5}, \chi_{4.6}\ ],  [\ \chi_{4.2}, \chi_{4.3}, \chi_{4.4}\ ],$ %\newline
$ [\ \chi_{5.1}, \chi_{5.2}, \chi_{5.3}, \chi_{5.4}, \chi_{5.5}, \chi_{5.6}\ ],$ \newline
$ [\ \chi_{6.1}, \chi_{6.2}, \chi_{6.3}, \chi_{6.4}\ ],$ %\newline
$ [\ \chi_{7.1}, \chi_{7.2}, \chi_{7.3}, \chi_{7.4}, \chi_{7.5}\ ]$. \newline
\end{prop}

The indicator values of these irreducible I-equivalent characters are displayed in Table 4.3 below.
\begin{table}[ht]
\caption{$D(S_5)$ indicators: (exponent 60)}
$ 
\begin{array}{r|cccccccccccc} \hline 
m =  & 1 
 & 2 
 & 3 
 & 4 
 & 5 
 & 6 
 & 10 
 & 12 
 & 15 
 & 20 
 & 30 
 & 60 
\rule[-7pt]{0pt}{20pt} \\ \hline 
\nu_m(\chi_{1.1}) & 0 & 1 & 0 & 1 & 0 & 1 & 1 & 1 & 0 & 1 & 1 & 
1 \rule[0pt]{0pt}{13pt} \\ 
\nu_m(\chi_{1.2}) & 0 & 1 & 0 & 2 & 1 & 2 & 2 & 3 & 1 & 3 & 3 & 4 \\ 
\nu_m(\chi_{1.3}) & 0 & 1 & 1 & 2 & 1 & 3 & 2 & 4 & 2 & 3 & 4 & 5 \\ 
\nu_m(\chi_{1.4}) & 0 & 1 & 1 & 3 & 1 & 3 & 2 & 5 & 2 & 4 & 4 & 6 \\ 
\nu_m(\chi_{1.6}) & 0 & 1 & 1 & 2 & 1 & 2 & 2 & 3 & 2 & 3 & 3 & 4 \\ 
\nu_m(\chi_{1.7}) & 1 & 1 & 1 & 1 & 1 & 1 & 1 & 1 & 1 & 1 & 1 & 1 \\ 
\nu_m(\chi_{2.1}) & 0 & 1 & 0 & 2 & 0 & 3 & 1 & 6 & 0 & 4 & 5 & 10 \\ 
\nu_m(\chi_{2.5}) & 0 & 1 & 0 & 3 & 0 & 6 & 1 & 12 & 0 & 7 & 10 & 20 \\ 
\nu_m(\chi_{3.1}) & 0 & 1 & 1 & 4 & 1 & 5 & 4 & 10 & 4 & 7 & 10 & 15 \\ 
\nu_m(\chi_{3.5}) & 0 & 1 & 2 & 8 & 2 & 9 & 7 & 20 & 8 & 14 & 19 & 30 \\ 
\nu_m(\chi_{4.1}) & 0 & 1 & 2 & 5 & 2 & 6 & 5 & 14 & 6 & 9 & 12 & 20 \\ 
\nu_m(\chi_{4.2}) & 0 & 1 & 1 & 5 & 2 & 6 & 5 & 14 & 5 & 9 & 12 & 20 \\ 
\nu_m(\chi_{5.1}) & 0 & 1 & 0 & 3 & 0 & 6 & 3 & 12 & 0 & 9 & 10 & 20 \\ 
\nu_m(\chi_{6.1}) & 0 & 1 & 0 & 6 & 0 & 9 & 3 & 18 & 0 & 14 & 15 & 30 \\ 
\nu_m(\chi_{7.1}) & 0 & 1 & 1 & 6 & 2 & 7 & 5 & 16 & 7 & 10 & 15 & 
24 \rule[-7pt]{0pt}{5pt} \\ 
\hline 
\end{array} 
$
\label{table:S5} 
\end{table}
%\clearpage

%4.4
\subsection{}
\bf Indicators of $D(S_6)$ \rm   %-------------------4.4-------------------

In this section, we give details leading up to the use of Equation \ref{calc} to find all the irreducible I-equivalent character classes of $D(S_6)$ and their indicator values.  Recall Equation \ref{calc}: 
 $$\nu_m(\chi_{i.j}) = \frac{1}{|C_{S_n}(u_i)|} \, \sum_{y\in \tilde{G}_m^m(u_i)} \Gamma_m(u_i, y)\eta_j(y) $$\normalsize
 where $u_i$ is a representative of a conjugacy class in $S_n$, $\eta_j$ is an irreducible character of $C_{S_n}(u_i)$ the centralizer of $u_i$ in $S_n$, and $\chi_{i.j}$ is the irreducible character of $D(S_n)$ induced up from $\eta_j$ as described in Lemma \ref{repsofDG}. $\tilde{G}^m_m(u)$ and $\Gamma_m(u_i, y)$ were given in Definitions \ref{Gmm} and \ref{gamma}.

To begin, we first choose conjugacy class representatives of $S_6$ so we can look at the character tables of their centralizers. 

\clearpage

The conjugacy class representatives $u_i$ used in these calculations and their centralizers $C_{S_6}(u_i)$ are:
$$
\begin{array} {ll}
\begin{array}{r| c }
i & u_i \\ \hline 
1 &  ()  \\
2 &  (1,2) \\
3 & (1,2)(3,4) \\
4 & (1,2)(3,4)(5,6) \\
5 & (1,2,3) \\
6 & (1,2,3)(4,5) \\ 
7 & (1,2,3)(4,5,6) \\
8 & (1,2,3,4) \\
9 & (1,2,3,4)(5,6) \\
10 & (1,2,3,4,5) \\ 
11 & (1,2,3,4,5,6) \\ \hline 
\end{array} &
	\begin{array}{l} \\
	C_{S_6}(u_{1}) = S_6.\\
	C_{S_6}(u_{2}) = \langle (1,2), (3,6), (4,6), (5,6)\rangle. \\
	C_{S_6}(u_{3}) = \langle (1,2), (1,3)(2,4), (3,4), (5,6) \rangle \cong D_8  \times C_2. \\
	C_{S_6}(u_{4}) = \langle (1,2), (1,5)(2,6), (3,4), (3,5)(4,6), (5,6) \rangle. \\
	C_{S_6}(u_{5}) = \langle (1,2,3), (4,6), (5,6) \rangle. \\
	C_{S_6}(u_{6}) = \langle (1,2,3), (4,5) \rangle \cong C_6 \text{ is abelian.} \\
	C_{S_6}(u_{7}) = \langle (1,2,3), (1,4)(2,5)(3,6), (4,5,6) \rangle. \\
	C_{S_6}(u_{8}) = \langle (1,2,3,4), (5,6) \rangle \cong C_4 \times C_2 \text{ is abelian.} \\
	C_{S_6}(u_{9}) = \langle (1,2,3,4), (5,6) \rangle \cong C_4 \times C_2 \text{ is abelian.} \\
	C_{S_6}(u_{10}) = \langle (1,2,3,4,5)\rangle \cong C_5 \text{ is abelian.} \\
	C_{S_6}(u_{11}) = \langle (1,2,3,4,5,6)\rangle \cong C_6 \text{ is abelian.} \\
	\end{array}
\end{array}
$$
It is also worth noting that although $C_{S_6}(u_{2})  \neq C_{S_6}(u_{4})$, they are isomorphic as groups.  That is $C_{S_6}(u_{2})  \cong C_{S_6}(u_{4})$.  The same holds true for $C_{S_6}(u_{5})  \neq C_{S_6}(u_{7})$, but $C_{S_6}(u_{5})  \cong C_{S_6}(u_{7})$.  And clearly $C_{S_6}(u_{8})  = C_{S_6}(u_{9})$.\newline

 {\bf(1)} When $u_1 = ()$:\newline
The sets $\tilde{G}_m^m(u_1)$ that we sum over, and the corresponding coefficients $\Gamma_m(u_1,y)$ for each element $y$ in $\tilde{G}_m^m$ are listed below, with the corresponding value of $\Gamma_m(u_1,y)$ preceding each element in $\tilde{G}_m^m(u_1)$, in the same fashion as in Section 4.1.

 Since $\tilde{G}_m(u_1) = S_6$ for all $m$, we have: \newline
$\tilde{G}_{1}^{1}(u_{1})$ consists of $\{  [ 1, () ], [ 15, (1,2) ], [ 45, (1,2)(3,4) ], [ 15, (1,2)(3,4)(5,6) ],$

$  [ 40, (1,2,3) ], [ 120, (1,2,3)(4,5) ], [ 40, (1,2,3)(4,5,6) ], [ 90, (1,2,3,4) ],$ 

$[ 90, (1,2,3,4)(5,6) ], [ 144, (1,2,3,4,5) ], [ 120, (1,2,3,4,5,6) ] \}$; \newline 
$\tilde{G}_{2}^{2}(u_{1})$ consists of $\{ [ 76, () ], [ 160, (1,2,3) ], 
  [ 160, (1,2,3)(4,5,6) ],$

$ [ 180, (1,2)(3,4) ], [ 144, (1,2,3,4,5) ]  \}$; \newline 
$\tilde{G}_{3}^{3}(u_{1})$ consists of $\{ [ 81, () ], [ 135, (1,2) ], [ 45, (1,2)(3,4) ],$

$ [ 135, (1,2)(3,4)(5,6) ], [ 90, (1,2,3,4) ], [ 90, (1,2,3,4)(5,6) ],$

$ [ 144, (1,2,3,4,5) ]  \}$; \newline 
$\tilde{G}_{4}^{4}(u_{1})$ consists of $\{ [ 256, () ], [ 160, (1,2,3) ], 
  [ 160, (1,2,3)(4,5,6) ],$

$ [ 144, (1,2,3,4,5) ] \}$; \newline 
$\tilde{G}_{5}^{5}(u_{1})$ consists of $\{  [ 145, () ], [ 15, (1,2) ], [ 45, (1,2)(3,4) ], [ 15, (1,2)(3,4)(5,6) ],$ 

$  [ 40, (1,2,3) ], [ 120, (1,2,3)(4,5) ], [ 40, (1,2,3)(4,5,6) ],  [ 90, (1,2,3,4) ],$

$ [ 90, (1,2,3,4)(5,6) ], [ 120, (1,2,3,4,5,6) ]  \}$; \newline 
$\tilde{G}_{6}^{6}(u_{1})$ consists of $\{ [ 396, () ], [ 180, (1,2)(3,4) ], [ 144, (1,2,3,4,5) ] \}$; \newline 
$\tilde{G}_{10}^{10}(u_{1})$ consists of $\{ [ 220, () ], [ 160, (1,2,3) ], 
  [ 160, (1,2,3)(4,5,6) ],$

$ [ 180, (1,2)(3,4) ] \}$; \newline 
$\tilde{G}_{12}^{12}(u_{1})$ consists of $\{  [ 576, () ], [ 144, (1,2,3,4,5) ] \}$; \newline 
$\tilde{G}_{15}^{15}(u_{1})$ consists of $\{  [ 225, () ], [ 135, (1,2) ], [ 45, (1,2)(3,4) ],$

$ [ 135, (1,2)(3,4)(5,6) ],   [ 90, (1,2,3,4) ], [ 90, (1,2,3,4)(5,6) ] \}$; \newline 
$\tilde{G}_{20}^{20}(u_{1})$ consists of $\{ [ 400, () ], [ 160, (1,2,3) ], 
  [ 160, (1,2,3)(4,5,6) ] \}$; \newline 
$\tilde{G}_{30}^{30}(u_{1})$ consists of $\{  [ 540, () ], [ 180, (1,2)(3,4) ] \}$; \newline 
$\tilde{G}_{60}^{60}(u_{1})$ consists of $\{ [ 720, () ] \}$. \newline

For the rest of the $u$'s, we list only the sets $\tilde{G}_m^m(u)$  that we sum over, and the corresponding coefficients $\Gamma_m(u,y)$ for each element $y$ in $\tilde{G}_m^m(u)$, with the corresponding value of $\Gamma_m(u,y)$ preceding each element in $\tilde{G}_m^m(u)$ in the same fashion as above. We do not list the sets $\tilde{G}_m(u)$ as we did in Sections 4.1 and 4.2 since these sets can be very large and too long to include in this paper.

In the pages that follow we will see that for $i=6,8,9,10,$ and 11 all $\tilde{G}_m^m(u_i)$ are either empty or only contain the identity element.  So Lemma \ref{I-equiv} applies to $C_{S_6}(u_i)$, for $i=6,8,9,10,$ and 11 since each of these centralizer groups is abelian.  Thus we do not really need the character tables of these centralizers, we only need to know that there are 6 irreducible characters of $C_{S_6}(u_6)$, 8 irreducible characters of $C_{S_6}(u_8)$, 8 irreducible characters of $C_{S_6}(u_9)$, 5 irreducible characters of $C_{S_6}(u_{10})$, and 6 irreducible characters of $C_{S_6}(u_{11})$.\newline

 {\bf(2)} When $u_2 = (1,2)$:

\noindent $\tilde{G}_1^1(u_{2}) = \emptyset$; \newline
$\tilde{G}_{2}^{2}(u_{2})$ consists of $\{  [ 20, () ], [ 16, (4,5,6) ], [ 12, (3,4)(5,6) ] \}$; \newline 
$\tilde{G}_3^3(u_{2}) = \emptyset$; \newline
$\tilde{G}_{4}^{4}(u_{2})$ consists of $\{ [ 80, () ], [ 16, (4,5,6) ]  \}$; \newline 
$\tilde{G}_5^5(u_{2}) = \emptyset$; \newline
$\tilde{G}_{6}^{6}(u_{2})$ consists of $\{ [ 180, () ], [ 12, (3,4)(5,6) ]  \}$; \newline 
$\tilde{G}_{10}^{10}(u_{2})$ consists of $\{  [ 20, () ], [ 16, (4,5,6) ], [ 12, (3,4)(5,6) ] \}$; \newline 
$\tilde{G}_{12}^{12}(u_{2})$ consists of $\{ [ 432, () ] \}$; \newline 
$\tilde{G}_{15}^{15}(u_{2}) = \emptyset$; \newline
$\tilde{G}_{20}^{20}(u_{2})$ consists of $\{ [ 176, () ], [ 16, (4,5,6) ]  \}$; \newline 
$\tilde{G}_{30}^{30}(u_{2})$ consists of $\{  [ 372, () ], [ 12, (3,4)(5,6) ] \}$; \newline 
$\tilde{G}_{60}^{60}(u_{2})$ consists of $\{ [ 720, () ] \}$. \newline 

{\bf(3)} When $u_3 = (1,2)(3,4)$:

\noindent $\tilde{G}_1^1(u_{3}) = \emptyset$; \newline
$\tilde{G}_{2}^{2}(u_{3})$ consists of $\{ [ 12, () ], [ 4, (1,2)(3,4) ]  \}$; \newline 
$\tilde{G}_{3}^{3}(u_{3})$ consists of $\{  [ 16, () ], [ 8, (5,6) ], [ 8, (1,2)(3,4)(5,6) ] \}$; \newline 
$\tilde{G}_{4}^{4}(u_{3})$ consists of $\{ [ 112, () ] \}$; \newline 
$\tilde{G}_{5}^{5}(u_{3})$ consists of $\{ [ 64, () ] \}$; \newline 
$\tilde{G}_{6}^{6}(u_{3})$ consists of $\{ [ 268, () ], [ 4, (1,2)(3,4) ]  \}$; \newline 
$\tilde{G}_{10}^{10}(u_{3})$ consists of $\{ [ 108, () ], [ 4, (1,2)(3,4) ]  \}$; \newline 
$\tilde{G}_{12}^{12}(u_{3})$ consists of $\{ [ 496, () ] \}$; \newline 
$\tilde{G}_{15}^{15}(u_{3})$ consists of $\{  [ 144, () ], [ 8, (5,6) ], [ 8, (1,2)(3,4)(5,6) ] \}$; \newline 
$\tilde{G}_{20}^{20}(u_{3})$ consists of $\{ [ 272, () ] \}$; \newline 
$\tilde{G}_{30}^{30}(u_{3})$ consists of $\{ [ 428, () ], [ 4, (1,2)(3,4) ]  \}$; \newline 
$\tilde{G}_{60}^{60}(u_{3})$ consists of $\{ [ 720, () ] \}$. \newline 

 {\bf(4)} When $u_4 = (1,2)(3,4)(5,6)$:

\noindent $\tilde{G}_1^1(u_{4}) = \emptyset$; \newline
$\tilde{G}_{2}^{2}(u_{4})$ consists of $\{ [ 20, () ], [ 12, (3,4)(5,6) ],   [ 16, (1,3,5)(2,4,6) ] \}$; \newline 
$\tilde{G}_3^3(u_{4}) = \emptyset$; \newline
$\tilde{G}_{4}^{4}(u_{4})$ consists of $\{ [ 80, () ], [ 16, (1,3,5)(2,4,6) ] \}$; \newline 
$\tilde{G}_5^5(u_{4}) = \emptyset$; \newline
$\tilde{G}_{6}^{6}(u_{4})$ consists of $\{ [ 180, () ], [ 12, (3,4)(5,6) ]  \}$; \newline 
$\tilde{G}_{10}^{10}(u_{4})$ consists of $\{ [ 20, () ], [ 12, (3,4)(5,6) ], [ 16, (1,3,5)(2,4,6) ] \}$; \newline 
$\tilde{G}_{12}^{12}(u_{4})$ consists of $\{ [ 432, () ] \}$; \newline 
$\tilde{G}_{15}^{15}(u_{4}) = \emptyset$; \newline
$\tilde{G}_{20}^{20}(u_{4})$ consists of $\{  [ 176, () ], [ 16, (1,3,5)(2,4,6) ] \}$; \newline 
$\tilde{G}_{30}^{30}(u_{4})$ consists of $\{  [ 372, () ], [ 12, (3,4)(5,6) ] \}$; \newline 
$\tilde{G}_{60}^{60}(u_{4})$ consists of $\{ [ 720, () ] \}$. \newline

 {\bf(5)} When $u_5 = (1,2,3)$:

\noindent $\tilde{G}_1^1(u_{5}) = \emptyset$; \newline
$\tilde{G}_{2}^{2}(u_{5})$ consists of $\{ [ 12, () ], [ 6, (4,5,6) ]  \}$; \newline 
$\tilde{G}_{3}^{3}(u_{5})$ consists of $\{ [ 18, () ], [ 18, (5,6) ]  \}$; \newline 
$\tilde{G}_{4}^{4}(u_{5})$ consists of $\{ [ 102, () ], [ 6, (4,5,6) ] \}$; \newline 
$\tilde{G}_{5}^{5}(u_{5})$ consists of $\{ [ 54, () ] \}$; \newline 
$\tilde{G}_{6}^{6}(u_{5})$ consists of $\{ [ 252, () ] \}$; \newline 
$\tilde{G}_{10}^{10}(u_{5})$ consists of $\{ [ 102, () ], [ 6, (4,5,6) ]  \}$; \newline 
$\tilde{G}_{12}^{12}(u_{5})$ consists of $\{ [ 486, () ] \}$; \newline 
$\tilde{G}_{15}^{15}(u_{5})$ consists of $\{ [ 144, () ], [ 18, (5,6) ]  \}$; \newline 
$\tilde{G}_{20}^{20}(u_{5})$ consists of $\{ [ 264, () ], [ 6, (4,5,6) ]  \}$; \newline 
$\tilde{G}_{30}^{30}(u_{5})$ consists of $\{ [ 414, () ] \}$; \newline 
$\tilde{G}_{60}^{60}(u_{5})$ consists of $\{ [ 720, () ] \}$. \newline 

%\newpage
 {\bf(6)} When $u_6 = (1,2,3)(4,5)$:

\noindent $\tilde{G}_1^1(u_{6}) = \emptyset$; \newline
$\tilde{G}_{2}^{2}(u_{6})$ consists of $\{ [ 6, () ] \}$; \newline 
$\tilde{G}_3^3(u_{6}) = \emptyset$; \newline
$\tilde{G}_{4}^{4}(u_{6})$ consists of $\{ [ 96, () ] \}$; \newline 
$\tilde{G}_5^5(u_{6}) = \emptyset$; \newline
$\tilde{G}_{6}^{6}(u_{6})$ consists of $\{ [ 198, () ] \}$; \newline 
$\tilde{G}_{10}^{10}(u_{6})$ consists of $\{ [ 30, () ] \}$; \newline 
$\tilde{G}_{12}^{12}(u_{6})$ consists of $\{ [ 432, () ] \}$; \newline 
$\tilde{G}_{15}^{15}(u_{6}) = \emptyset$; \newline
$\tilde{G}_{20}^{20}(u_{6})$ consists of $\{ [ 192, () ] \}$; \newline 
$\tilde{G}_{30}^{30}(u_{6})$ consists of $\{ [ 414, () ] \}$; \newline 
$\tilde{G}_{60}^{60}(u_{6})$ consists of $\{ [ 720, () ] \}$.  \newline

 {\bf(7)} When $u_7 = (1,2,3)(4,5,6)$:

\noindent $\tilde{G}_1^1(u_{7}) = \emptyset$; 
$\tilde{G}_{2}^{2}(u_{7})$ consists of $\{  [ 12, () ], [ 6, (1,2,3)(4,6,5) ] \}$; \newline 
$\tilde{G}_{3}^{3}(u_{7})$ consists of $\{  [ 18, () ], [ 18, (1,4)(2,5)(3,6) ] \}$; \newline 
$\tilde{G}_{4}^{4}(u_{7})$ consists of $\{  [ 102, () ], [ 6, (1,2,3)(4,6,5) ] \}$; \newline 
$\tilde{G}_{5}^{5}(u_{7})$ consists of $\{ [ 54, () ] ] \}$; \newline 
$\tilde{G}_{6}^{6}(u_{7})$ consists of $\{ [ 252, () ] ] \}$; \newline 
$\tilde{G}_{10}^{10}(u_{7})$ consists of $\{  [ 102, () ], [ 6, (1,2,3)(4,6,5) ] \}$; \newline 
$\tilde{G}_{12}^{12}(u_{7})$ consists of $\{ [ 486, () ] ] \}$; \newline 
$\tilde{G}_{15}^{15}(u_{7})$ consists of $\{  [ 144, () ], [ 18, (1,4)(2,5)(3,6) ] \}$; \newline 
$\tilde{G}_{20}^{20}(u_{7})$ consists of $\{  [ 264, () ], [ 6, (1,2,3)(4,6,5) ] \}$; \newline 
$\tilde{G}_{30}^{30}(u_{7})$ consists of $\{ [ 414, () ] \}$; \newline 
$\tilde{G}_{60}^{60}(u_{7})$ consists of $\{ [ 720, () ] \}$. \newline 

%\newpage
 {\bf(8)} When $u_8 = (1,2,3,4)$:

\noindent $\tilde{G}_1^1(u_{8}) = \emptyset$; \newline
$\tilde{G}_{2}^{2}(u_{8})$ consists of $\{ [ 8, () ] \}$; \newline 
$\tilde{G}_3^3(u_{8}) = \emptyset$; \newline
$\tilde{G}_{4}^{4}(u_{8})$ consists of $\{ [ 80, () ] \}$; \newline 
$\tilde{G}_5^5(u_{8}) = \emptyset$; \newline
$\tilde{G}_{6}^{6}(u_{8})$ consists of $\{ [ 168, () ] \}$; \newline 
$\tilde{G}_{10}^{10}(u_{8})$ consists of $\{ [ 40, () ] \}$; \newline 
$\tilde{G}_{12}^{12}(u_{8})$ consists of $\{ [ 432, () ] \}$; \newline 
$\tilde{G}_{15}^{15}(u_{8}) = \emptyset$; \newline
$\tilde{G}_{20}^{20}(u_{8})$ consists of $\{ [ 176, () ] \}$; \newline 
$\tilde{G}_{30}^{30}(u_{8})$ consists of $\{ [ 392, () ] \}$; \newline 
$\tilde{G}_{60}^{60}(u_{8})$ consists of $\{ [ 720, () ] \}$. \newline

%\newpage
 {\bf(9)} When $u_9 = (1,2,3,4)(5,6)$:

\noindent $\tilde{G}_1^1(u_{9}) = \emptyset$; \newline
$\tilde{G}_{2}^{2}(u_{9})$ consists of $\{ [ 8, () ] \}$; \newline 
$\tilde{G}_{3}^{3}(u_{9})$ consists of $\{ [ 16, () ] \}$; \newline 
$\tilde{G}_{4}^{4}(u_{9})$ consists of $\{ [ 80, () ] \}$; \newline 
$\tilde{G}_{5}^{5}(u_{9})$ consists of $\{ [ 64, () ] \}$; \newline 
$\tilde{G}_{6}^{6}(u_{9})$ consists of $\{ [ 232, () ] \}$; \newline 
$\tilde{G}_{10}^{10}(u_{9})$ consists of $\{ [ 104, () ] \}$; \newline 
$\tilde{G}_{12}^{12}(u_{9})$ consists of $\{ [ 496, () ] \}$; \newline 
$\tilde{G}_{15}^{15}(u_{9})$ consists of $\{ [ 144, () ] \}$; \newline 
$\tilde{G}_{20}^{20}(u_{9})$ consists of $\{ [ 240, () ] \}$; \newline 
$\tilde{G}_{30}^{30}(u_{9})$ consists of $\{ [ 392, () ] \}$; \newline 
$\tilde{G}_{60}^{60}(u_{9})$ consists of $\{ [ 720, () ] \}$. \newline

 {\bf(10)} When $u_{10} = (1,2,3,4,5)$:

\noindent $\tilde{G}_1^1(u_{10}) = \emptyset$; \newline
$\tilde{G}_{2}^{2}(u_{10})$ consists of $\{ [ 5, () ] \}$; \newline 
$\tilde{G}_{3}^{3}(u_{10})$ consists of $\{ [ 20, () ] \}$; \newline 
$\tilde{G}_{4}^{4}(u_{10})$ consists of $\{ [ 85, () ] \}$; \newline 
$\tilde{G}_{5}^{5}(u_{10})$ consists of $\{ [ 55, () ] \}$; \newline 
$\tilde{G}_{6}^{6}(u_{10})$ consists of $\{ [ 245, () ] \}$; \newline 
$\tilde{G}_{10}^{10}(u_{10})$ consists of $\{ [ 100, () ] \}$; \newline 
$\tilde{G}_{12}^{12}(u_{10})$ consists of $\{ [ 485, () ] \}$; \newline 
$\tilde{G}_{15}^{15}(u_{10})$ consists of $\{ [ 135, () ] \}$; \newline 
$\tilde{G}_{20}^{20}(u_{10})$ consists of $\{ [ 260, () ] \}$; \newline 
$\tilde{G}_{30}^{30}(u_{10})$ consists of $\{ [ 400, () ] \}$; \newline 
$\tilde{G}_{60}^{60}(u_{10})$ consists of $\{ [ 720, () ] \}$.  \newline

 {\bf(11)} When $u_{11} = (1,2,3,4,5,6)$:

\noindent $\tilde{G}_1^1(u_{11}) = \emptyset$; \newline
$\tilde{G}_{2}^{2}(u_{11})$ consists of $\{ [ 6, () ] \}$; \newline 
$\tilde{G}_3^3(u_{11}) = \emptyset$; \newline
$\tilde{G}_{4}^{4}(u_{11})$ consists of $\{ [ 96, () ] \}$; \newline 
$\tilde{G}_5^5(u_{11}) = \emptyset$; \newline
$\tilde{G}_{6}^{6}(u_{11})$ consists of $\{ [ 198, () ] \}$; \newline 
$\tilde{G}_{10}^{10}(u_{11})$ consists of $\{ [ 30, () ] \}$; \newline 
$\tilde{G}_{12}^{12}(u_{11})$ consists of $\{ [ 432, () ] \}$; \newline 
$\tilde{G}_{15}^{15}(u_{11}) = \emptyset$; \newline
$\tilde{G}_{20}^{20}(u_{11})$ consists of $\{ [ 192, () ] \}$; \newline 
$\tilde{G}_{30}^{30}(u_{11})$ consists of $\{ [ 414, () ] \}$; \newline 
$\tilde{G}_{60}^{60}(u_{11})$ consists of $\{ [ 720, () ] \}$.  \newline

Thus using the coeffieicents $\Gamma$, the elements in $\tilde{G}_m^m(u_i)$, and the character tables of $C_{S_6}(u_i)$ with Equations \ref{calc}, we find:
\begin{prop}  The irreducible characters of $D(S_6)$ have the following 21 distinct I-equivalent classes:\newline
$[\ \chi_{1.1} \ ], [\ \chi_{1.2}, \chi_{1.7} \ ], [\ \chi_{1.3} \ ],  [\ \chi_{1.4}, \chi_{1.10} \ ],
[\ \chi_{1.5}, \chi_{1.8} \ ], [\ \chi_{1.6} \ ], [\ \chi_{1.9} \ ], [\ \chi_{1.11} \ ]$, \newline
$[\ \chi_{2.1}, \chi_{2.2}, \chi_{2.3}, \chi_{2.4}, \chi_{4.1}, \chi_{4.2}, \chi_{4.3}, \chi_{4.4} \ ], [\ \chi_{2.5}, \chi_{2.6}, \chi_{4.5}, \chi_{4.6} \ ]$, \newline 
\indent $[\ \chi_{2.7}, \chi_{2.8}, \chi_{2.9}, \chi_{2.10}, \chi_{4.7}, \chi_{4.8}, \chi_{4.9}, \chi_{4.10} \ ]$, \newline 
$[\ \chi_{3.1}, \chi_{3.6}, \chi_{3.7}, \chi_{3.8} \ ], [\ \chi_{3.2}, \chi_{3.3}, \chi_{3.4}, \chi_{3.5} \ ], [\ \chi_{3.9}, \chi_{3.10} \ ]$, \newline 
$[\ \chi_{5.1}, \chi_{5.5}, \chi_{5.6}, \chi_{7.1}, \chi_{7.5}, \chi_{7.6} \ ], [\ \chi_{5.2}, \chi_{5.3}, \chi_{5.4}, \chi_{7.2}, \chi_{7.3}, \chi_{7.4} \ ]$, \newline 
\indent $[\ \chi_{5.7}, \chi_{5.8}, \chi_{5.9}, \chi_{7.7}, \chi_{7.8}, \chi_{7.9} \ ]$, \newline 
$[\ \chi_{6.1}, \chi_{6.2}, \chi_{6.3}, \chi_{6.4}, \chi_{6.5}, \chi_{6.6}, \chi_{11.1}, \chi_{11.2}, \chi_{11.3}, 
\chi_{11.4}, \chi_{11.5}, \chi_{11.6} \ ]$, \newline 
$[\ \chi_{8.1}, \chi_{8.2}, \chi_{8.3}, \chi_{8.4}, \chi_{8.5}, \chi_{8.6}, \chi_{8.7}, \chi_{8.8} \ ]$, \newline 
$[\ \chi_{9.1}, \chi_{9.2}, \chi_{9.3}, \chi_{9.4}, \chi_{9.5}, \chi_{9.6}, \chi_{9.7}, \chi_{9.8} \ ]$, \newline 
$[\ \chi_{10.1}, \chi_{10.2}, \chi_{10.3}, \chi_{10.4}, \chi_{10.5} \ ]$.
\end{prop}

The indicator values of these irreducible I-equivalent characters are displayed in the Table 4.4 below.
\newpage
\begin{table}[ht]
\caption{$D(S_6)$ indicators: (exponent 60)}
$
\begin{array}{r|cccccccccccc} \hline 
m =  & 1 
 & 2 
 & 3 
 & 4 
 & 5 
 & 6 
 & 10 
 & 12 
 & 15 
 & 20 
 & 30 
 & 60 
\rule[-7pt]{0pt}{20pt} \\ \hline 
\nu_m(\chi_{1.1}) & 0 & 1 & 0 & 1 & 0 & 1 & 1 & 1 & 0 & 1 & 1 & 
1 \rule[0pt]{0pt}{13pt} \\ 
\nu_m(\chi_{1.2}) & 0 & 1 & 0 & 2 & 1 & 3 & 2 & 4 & 1 & 3 & 4 & 5 \\ 
\nu_m(\chi_{1.3}) & 0 & 1 & 0 & 3 & 2 & 5 & 3 & 7 & 2 & 5 & 7 & 9 \\ 
\nu_m(\chi_{1.4}) & 0 & 1 & 1 & 2 & 1 & 3 & 2 & 4 & 2 & 3 & 4 & 5 \\ 
\nu_m(\chi_{1.5}) & 0 & 1 & 1 & 4 & 2 & 5 & 3 & 8 & 3 & 6 & 7 & 10 \\ 
\nu_m(\chi_{1.6}) & 0 & 1 & 2 & 5 & 3 & 9 & 4 & 13 & 5 & 8 & 12 & 16 \\ 
\nu_m(\chi_{1.9}) & 0 & 1 & 2 & 3 & 2 & 5 & 3 & 7 & 4 & 5 & 7 & 9 \\ 
\nu_m(\chi_{1.11}) & 1 & 1 & 1 & 1 & 1 & 1 & 1 & 1 & 1 & 1 & 1 & 1 \\ 
\nu_m(\chi_{2.1}) & 0 & 1 & 0 & 2 & 0 & 4 & 1 & 9 & 0 & 4 & 8 & 15 \\ 
\nu_m(\chi_{2.5}) & 0 & 1 & 0 & 3 & 0 & 8 & 1 & 18 & 0 & 7 & 16 & 30 \\ 
\nu_m(\chi_{2.7}) & 0 & 1 & 0 & 5 & 0 & 11 & 1 & 27 & 0 & 11 & 23 & 45 \\ 
\nu_m(\chi_{3.1}) & 0 & 1 & 2 & 7 & 4 & 17 & 7 & 31 & 10 & 17 & 27 & 45 \\ 
\nu_m(\chi_{3.2}) & 0 & 1 & 0 & 7 & 4 & 17 & 7 & 31 & 8 & 17 & 27 & 45 \\ 
\nu_m(\chi_{3.9}) & 0 & 1 & 2 & 14 & 8 & 33 & 13 & 62 & 18 & 34 & 53 & 90 \\ 
\nu_m(\chi_{5.1}) & 0 & 1 & 2 & 6 & 3 & 14 & 6 & 27 & 9 & 15 & 23 & 40 \\ 
\nu_m(\chi_{5.2}) & 0 & 1 & 0 & 6 & 3 & 14 & 6 & 27 & 7 & 15 & 23 & 40 \\ 
\nu_m(\chi_{5.7}) & 0 & 1 & 2 & 11 & 6 & 28 & 11 & 54 & 16 & 29 & 46 & 80 \\ 
\nu_m(\chi_{6.1}) & 0 & 1 & 0 & 16 & 0 & 33 & 5 & 72 & 0 & 32 & 69 & 120 \\ 
\nu_m(\chi_{8.1}) & 0 & 1 & 0 & 10 & 0 & 21 & 5 & 54 & 0 & 22 & 49 & 90 \\ 
\nu_m(\chi_{9.1}) & 0 & 1 & 2 & 10 & 8 & 29 & 13 & 62 & 18 & 30 & 49 & 90 \\ 
\nu_m(\chi_{10.1}) & 0 & 1 & 4 & 17 & 11 & 49 & 20 & 97 & 27 & 52 & 80 & 
144 \rule[-7pt]{0pt}{5pt} \\ 
\hline 
\end{array} 
$
\label{table:S6} 
\end{table}
%\clearpage

%----------------------------------------5. GAP Functions and Code -------------------------------

\section{GAP Functions and Code} 
%Chapter 5

In this section we provide the code to all of the most efficient functions we used to calculate the indicators in the Section 4, as well as the functions used to print them in a format that could be directly copied into a LaTex file.  The code for our first collection of very time consuming functions is not included.  
First, recall Equation \ref{calc}: 
 $$\nu_m(\chi_{i.j}) = \frac{1}{|C_{S_n}(u_i)|} \, \sum_{y\in \tilde{G}_m^m(u_i)} \Gamma_m(u_i, y)\eta_j(y) $$\normalsize
 where $u_i$ is a representative of a conjugacy class in $S_n$, $\eta_j$ is an irreducible character of $C_{S_n}(u_i)$ the centralizer of $u_i$ in $S_n$, and $\chi_{i.j}$ is the irreducible character of $D(S_n)$ induced up from $\eta_j$ as described in Lemma \ref{repsofDG}. $\tilde{G}^m_m(u)$ and $\Gamma_m(u_i, y)$ were defined in Definitions \ref{Gmm} and \ref{gamma}.  

This equation was coded into the multiple functions that follow in Sections 1 and 2.  To compute and display all the indicator tables of $D(S_n)$ for $n \leq 9 $ we used the last function FSLaTexTabs (Definition \ref{FSLaTexTabs}) given in Section 3 which is built from all the functions preceeding it.  In Section 4 we give a function needed for providing details seen in Sections 4.1 and 4.2, as well as a function needed to modify the size of the indicator tables.  And Section 5 gives the details of how we computed the indicators of $D(S_{10})$.

The time needed to compute the indicator tables of $D(S_3)$ was 2 seconds, $D(S_4)$ was 2 seconds, $D(S_5)$ was 3 seconds, $D(S_6)$ was 5 seconds, $D(S_7)$ was 20 seconds, $D(S_8)$ was 3 minutes, and $D(S_9)$ was a little over 1 hour.  Attempting to compute the indicator tables of $D(S_{10})$ using FSLaTexTabs resulted in GAP error messages indicating there was not enough space to compute and store all the values at once.  As an algebra, $D(S_{10})$ has dimension $10!^2$ or 13,168,189,440,000 which is 100 times larger than the dimension of $D(S_9)$.  We were able to overcome these errors by breaking the code down into smaller pieces, but as a result it took about a four to five days to compute all the indicators of $D(S_{10})$.

Note that in GAP, all text following a \# is comment text and not part of the coding.  [GAP]

%5.1
\subsection{}
\bf Programming the set $\tilde{G}_m^m(u)$ \rm   %-------------------5.1-------------------

There are functions in GAP that allow us to pick representatives from conjugacy classes, compute centralizer groups, find irreducible characters, and even display character tables.   Thus, the first function we needed to write was how to compute the set $\tilde{G}_m^m(u)$ and the corresponding ``coefficient" $\Gamma_m(u,y)$ for each element $y \in \tilde{G}_m^m(u)$.

\begin{df}\label{Hmu} \rm  The function \it Hmu(G, m, u) \rm returns a lists of pairs $[ a, y ]$, where $a = \Gamma_m(u,y)$ and $y$ is an element in the set $\tilde{G}_m^m(u)$ for the specified group $G$.  The code for programming this function in GAP is provided below, as well as a running example of how the computation works at various steps.  The code is all left justified while the running example is right justified.
\end{df}
\begin{verbatim}
Hmu:= function(G,m,u)\end{verbatim}
\begin{flushright}{Example: Calling  Hmu( SymmetricGroup(5), 2 , (1,2) );}\\
{in GAP means $G = S_5,\ m = 2,$ and $u = (1,2)$.}\end{flushright} \begin{verbatim}
local GG, Gm, H, i, Hm, j, Cent, CC, cc, sum;
if u = () then
 Cent:= G;
 CC:=ConjugacyClasses(G);
 Hm:=List(CC, x -> [Size(x),(Representative(x))^m]);
     # if u = (), then (uh)^m = h^m for all h, 
     # so we need not check this.
else
 GG:=EnumeratorSorted(G);
 Gm:=List(GG, x->Position(GG, x^m));
     # All elements in G are ordered and given a position number.
     # Gm stores the position number of h^m in the place of h.\end{verbatim} 
\begin{flushright} {Example: GG $= [ (), (4,5), (3,4), (3,4,5), (3,5,4), (3,5),$ etc ],}\\
 {(GG)$^2 = [ (), (), (), (3,5,4), (3,4,5), (),$ etc ], so} \\
 {Gm $=$ [  1,  1,  1,  5,  4,  1,  etc ]}\end{flushright} \begin{verbatim}               
 H:=List(GG, x->[]);
 for i in [1..Size(GG)] do
  if Gm[Position(GG, u*GG[i])] = Gm[i] then
   Add(H[Gm[i]], i);;
  fi;
 od;
     # (uh)^m = h^m iff the position of element (uh)^m is the   
     # same as the position of the element h^m = Gm[i], so that 
     # is what's checked above.  All elements h that have the 
     # same h^m (and who satisfy(uh)^m = h^m) are collected and 
     # stored (as thier position in G numbers) in the position 
     # of h^m in G.\end{verbatim} 
\begin{flushright}{Example: H $=[[1, 2, 3, 6, 25, 26, 27, 30], [ ], [ ], [5, 29], [4, 28], [ ],$ etc ]}\\
 {no $h^2 = (4,5)$ so nothing is stored in the second place value.} \end{flushright}\begin{verbatim}
 Hm:=[];
 for j in [1..Size(H)] do
  if Size(H[j]) <> 0 then
   Add( Hm, [ Size(H[j]), GG[j] ] );
  fi;
 od;
     # This counts how many h in G have the same h^m (that also 
     # satisfy (uh)^m = h^m) and pairs this count with the
     # element h^m.\end{verbatim} 
\begin{flushright} {Example: Hm $= [ [ 8, () ], [ 2, (3,4,5) ], [ 2, (3,5,4) ] ]$,}\\
 { so there are 8 $h$ s.t. $h^2 = ()$, 2 $h$ s.t. $h^2 = (3,4,5)$,}\\
 { and 2 $h$ s.t. $h^2 = (3,5,4)$. But $(3,4,5)$ and}\\
 { $(3,5,4)$ are in the same conjugacy class,}\\
 { so these should be conbined.}\end{flushright}\begin{verbatim}
 Cent:=Centralizer(G, u);
 CC:=ConjugacyClasses(Cent);
fi;  # if u = (), then Cent and CC have been defined above
 H:=List(Hm, x->[x[1], Position(CC, ConjugacyClass(Cent,x[2]))]);
 cc:=List(H, x->x[2]);
     # This takes Hm and converts the element h^m into the number 
     # of the position of the conjugacy class of Cent(u)
     # that h^m falls into.\end{verbatim}
\begin{flushright}Example: H $= [ [ 8, 1 ], [ 2, 3 ], [ 2, 3 ]$], and cc $= [ 1, 3, 3$]\\
 since $(3,4,5)$ \& $(3,5,4)$ are in the third\\
 conjugacy class of Cent($(1,2)$).\end{flushright}\begin{verbatim}
 Hm:=[]; 
 for i in DuplicateFreeList(cc) do
  sum:=0;;
  for j in Positions(cc,i) do
   sum:= sum + H[j][1];
  od;
  Add( Hm, [sum, Representative(CC[i])]);
 od;\end{verbatim}
\begin{flushright} Example: DuplicateFreeList(cc) = [ 1, 3 ], thus $i = 1$ or $3$\\
when $i = 3$, Positions(cc,3) $= [ 2, 3 ]$, giving $j = 2$ or $3$\end{flushright}\begin{verbatim}
return Hm;
     # This returns a list of how many h in G satisfy 
     # (uh)^m = h^m for specific h^m, as indicated by the
     # conjugacy class representative of h^m.
\end{verbatim}
\begin{flushright}Example: Hm $= [ [8, () ], [ 4, (3,4,5) ] ]$\end{flushright}\begin{verbatim}
end; 
\end{verbatim}

\subsection{}
\bf Calculating the Indicators \rm   %-------------------5.2-------------------

The next two functions actually compute the higher indicators of $D(G)$ by evalutating Equation \ref{calc} for specific values of $u$ and $m$.

\begin{df} \label{FSGmu} FSGmu(G, m, u) \rm is a function which returns the $m^\text{th}$ Frobenius Schur indicator of all the irreducible characters of $D(G)$ induced from the irreducible characters of the centralizer of $u$ in $G$.  The code for programming this function in GAP is provided below.
\end{df}
\begin{verbatim}
FSGmu:= function(G,m,u)
local Centu, CharTab, Chars, CC, H, FS, sum, j, k, c, Xs, st;
if u = () then FS:=Indicator(CharacterTable(G),m);
else
 H:=Hmu(G,m,u);
# This is a list of lists.  Each sublist has two elements.  The
# first element is the number of h in G such that h^m = (uh)^m 
# for a specific h^m, and the second is the conjugacy class 
# representative in the centralizer of the element h^m.
 if H = [] then 
  FS:=List(ConjugacyClasses(Centralizer(G,u)), x -> 0);
# Since the number of irreducible characters of a finite group
# is equal to the number of conjugacy classes of that finite 
# group, FS will have the number of conjugacy classess 
# (= number of irreducibe characters) many 0's.  If H is empty,
# the indicator will be 0 for all irreducible characters.
 else 
  Centu:= Centralizer(G, u);
     # This is C_G(u) the centralizer of u in G.
  CharTab:= CharacterTable(Centu);
     # This is the character table for the centralizer group.
  Chars:= Irr(CharTab);
     # This is the list of all irreducible characters and the
     # character values for each conjugacy class.
  CC:=ConjugacyClasses(CharTab);
     # This lists the conjugacy classes of Centu in the order
     # that they appear in the table.
  FS:=[];
  for j in [1..Size(Chars)] do
     # for each irreducible character find the mth
     # Frobenius-Schur indicator.
    sum:=0;
      for k in [1..Size(H)] do
               # sum over the elements h^m that come from the H.
         for c in [1..Size(CC)] do
               # find the correct column of the CharTab (aka
               # the correct ConjClass) that the h^m is in.
           if H[k][2] in CC[c] then
             sum:= sum + (H[k][1])*(Chars[j][c]);
           fi;
         od;
      od;
    FS[j]:=sum/Size(Centu);
  od;
 fi;
fi;
return FS;
end;
\end{verbatim}

\begin{df} \label{FSGu} FSGu(G, mrange, u, i) \rm is a function that calculates all the $m^\text{th}$ Frobenius Schur indicators in $mrange$ of all the irreducible characters of $D(G)$ induced from the irreducible characters of the centralizer of $u$ in $G$.  $mrange$ is a list of values - usually the list of divisors of the exponent of $G$.  

 The results of this function are returned in a matrix where each row corresponds to an irreducible character and each column corresponds to the $m^\text{th}$ indicator values for a specified $m$ of the corresponding irreducible character.  (The number $i$ indicates the number of the conjugacy class that $u$ comes from).  The code for programming this function in GAP is provided below.
\end{df}

\begin{verbatim}
FSGu:=function(G,mrange,u,i)
local FSM, MAT, r, m, Xs, j, st;
FSM:=[]; MAT:=[]; r:=1; 
 for m in mrange do
   FSM:= FSGmu(G,m,u);  
     # Here we've found the mth Frobenius-Schur indicators 
     # for all irreducible characters induced from Cent(u).
   MAT[r]:=FSM;
     # This is a row in the matrix corresponding to m.
   r:=r+1;
 od;
 Xs:=[];
  for j in [1..Size(MAT[1])] do 
     # number of columns in MAT[1] is number of irreducible
     # characters of D(G) induced from Cent(u).
   st:=String(j);                 
   st:=Concatenation("\\chi_{",String(i),".",st,"}");   
   Xs[j]:=st;                 
  od;
 Add(MAT,Xs,1);       
return TransposedMat(MAT);
end;
\end{verbatim}

\subsection{}
\bf Condensing and Displaying Indicator Tables \rm   %-------------------5.3-------------------

Once we sucessfully computed all the higher indicators of $D(G)$ we still needed to determine which irreducible characters were I-equivalent and then remove those rows from our tables.  We also wanted an easy way to copy our tables of data into LaTex so we could print them.  The following functions allowed us to do just that.

\begin{df} \label{FSMatRed} FSMatRed(M) \rm is a function that takes a Frobenius Schur indicator matrix - as given from the functions FSGu -  (without the top row of $m$ values) and splits the first column (of all the irreducible characters) from the rest of the matrix (all the indicator values) and then reduces the matrix so there are no duplicate rows while at the same time keeping track of which rows were duplicates and thus which irreducible characters give equivalent values (that is, it finds the I-equivalent characters).  The list of I-equivalent irreducible characters and reduced matrix are returned.  The code for programming this function in GAP is provided below.
\end{df}
\begin{verbatim}
FSMatRed:=function(M)
local INDS, XS,samexs, sameinds, count, temp, i, j;
INDS:= M{[1..Size(M)]}{[2..Size(M[1])]};
     # INDS is the matrix M with the first column (the
     # characters) deleted, so INDS is just the indicator values.
XS:= M{[1..Size(M)]}{[1]};
     # XS is the column of irreducible characters taken 
     # from the matrix M.
samexs:=[]; sameinds:=[]; 
count:=1; temp:=[];
for i in [1..Size(XS)] do
 if INDS[i] in sameinds then    # do nothing
 else       # if INDS[i] is not in the list sameinds yet, add it
  sameinds[count]:=INDS[i];
     # sameinds is the matrix of non duplicated indicator values 
     # (being built row by row).
  temp:=Positions(INDS,INDS[i]);
     # temp contains the positions of the rows with identical 
     # indicator values
  for j in [1..Size(temp)] do
   temp[j]:=XS[temp[j]][1];  
  od;
     # temp now contains the irreducible characters 
     # corresponding to each row with identical indicators.
  samexs[count]:= temp;
     # samexs is the list of I-equivalent characters 
     # corresponding to the rows of sameinds.
  count:=count+1;
 fi;
od;
return [samexs, sameinds];
end;
\end{verbatim}

\begin{df} \label{FSIndicators} FSIndicators(num,exp) \rm  This function returns the a list of two parts.  The first part is a list of all the I-equivalent irreducible characters of $D(S_{num})$ and the second part is the matrix of all calculated $m^{\text{th}}$ Frobenius Schur indicators of $D(S_{num})$ for $m$ a divisor of $exp$.  This matrix has no duplicate rows (or columns) and each row $i$ corresponds to the $i^{\text{th}}$ I-equivalency class listed in the first part.  The code for programming this function in GAP is provided below.
\end{df}
\begin{verbatim}
FSIndicators:=function(num,exp)
local G, CCrep, factors, Mat, XS, k, FSMat, MAT, xs, samexs,
 sameinds, count, temp, i, pos, j;
factors:=DivisorsInt(exp);
G:=SymmetricGroup(num);
CCrep:= List(ConjugacyClasses(G), x -> Representative(x)); 
Mat:=[]; XS:=[];
for k in [1..Size(CCrep)] do        #for each CCrep do 
 FSMat:= FSMatRed(FSGu(G,factors,CCrep[k],k));  
 MAT:= FSMat[2];
 xs:= FSMat[1];
# MAT is the matrix of all the irreducible characters induced by
# CCrep[i] (in rows) with their corresponding mth indicators in
# columns (with the corresponding text or character numbers
# listed in xs). This separates the indicator values from the
# character numbers.  This matrix has already been reduced
# once eliminating duplications.
 Mat:=Concatenation(Mat,MAT); 
     # This stacks all the matrices on top of each other.
 XS:=Concatenation(XS, xs);
od;
samexs:=[]; sameinds:=[]; count:=1; temp:=[];
for i in [1..Size(Mat)] do
 if Mat[i] in sameinds then       # do nothing
 else       # if Mat[i] is not in the list sameinds yet, add it
  sameinds[count]:=Mat[i];
  pos:=Positions(Mat,Mat[i]); 
  temp:=[];
  for j in [1..Size(pos)] do
   temp:= Concatenation(temp,XS[pos[j]]);  
  od;
  samexs[count]:= temp;
  count:=count+1;
 fi;
od;
# Ultimately sameinds is the matrix with all the Frobenius-Schur 
# indicators, each row for a irreducible character and each 
# column for a value of m; and samexs is a list of all
# I-equivalent irreducible characters.
return [samexs, sameinds];
end;
\end{verbatim}

\begin{df} \label{FSLaTexTabs} FSLaTexTabs(num,exp) \rm is a function that gives a text that can be directly copied and pasted into LaTex containing the higher Frobenius-Schur  indicator table for $D(S_{num})$.  The exponent or \it exp \rm will provide what higher indicator values are distinct.  This output specifies that there are 13 columns per table.  The code for programming this function in GAP is provided below.
\end{df}
\begin{verbatim}
FSLaTexTabs:=function(num,exp)
local M, samexs, sameinds, factors, columns, j, numtables, k,
 range, z, l, i;
M:=FSIndicators(num, exp); 
samexs:=M[1]; sameinds:=M[2];
factors:=DivisorsInt(exp); 
     # Now for the display:
columns:="";
for j in [1..Size(factors)] do
 columns:=Concatenation(columns,"c");
od;
if IsInt(Size(factors)/12) then
 numtables:=Size(factors)/12;
else
 numtables:=Int(Size(factors)/12)+1;
fi;
     # The 12 means each table will have 12 values of m as 
     # columns (so 13 coulmns total)
Print("The $m^{\\textrm{th}}$ Frobenius-Schur Indicators
 of the irreducible \n characters of $D(S_",num,")$
 are:\n");
for k in [1..numtables] do
  if k*12 <= Size(factors) then
   range:=[(k-1)*12+1..k*12];
  else
   range:= [(k-1)*12+1..Size(factors)];
  fi;
  if k = 1 then 
   Print("\\begin{table}[ht] \n \\caption{$D(S_{",num,"})$:
    (exponent: $",exp,"$) Set 1} \n");
   Print("\\[ \n"); 
  else
   Print("\\begin{table}[ht] \n \\caption{$D(S_{",num,"})$:
    $m = ",factors[range[1]]," \\ldots ",
    factors[range[Size(range)]], "$ Set 1} \n \\[ \n");
  fi;
  Print("\\begin{array}{r|",columns{range},"} \\hline \n");
  for j in range do  
   if j = (k-1)*12+1 then Print("m = "); fi;
   Print(" & ",factors[j]," \n");
  od;
  Print("\\rule[-7pt]{0pt}{20pt} \\\\ \\hline \n");
  for z in [1..Size(samexs)] do
    Print("\\nu_m(",samexs[z][1],")");
    for l in range do
     Print(" & ",sameinds[z][l]);  
    od;
    if z = 1 or IsInt((z-1)/38) 
     then Print(" \\rule[0pt]{0pt}{13pt} \\\\ \n");
    elif IsInt(z/38) then 
     # This means each 38 rows of irreducible characters,
     # start a new table.
     Print(" \\rule[-7pt]{0pt}{5pt} \\\\ \n");
     Print("\\hline \n");
     Print("\\end{array} \n");
     Print("\\]\n\\label{table:S",num,".",Int(z/38),k,"}
       \n\\end{table} \n\n");
       # end table and start a new table.
     Print("\\begin{table}[ht] \n \\caption{$D(S_{",num,"})$:
      $m = ", factors[range[1]]," \\ldots ",
      factors[range[Size(range)]],"$ Set ", 1+Int(z/38),"}
      \n\\[ \n");
     Print("\\begin{array}{r|",columns{range},"} \\hline \n");
     for j in range do  
      if j = (k-1)*12+1 then Print("m = "); fi;
      Print(" & ",factors[j]," \n");
     od;
     Print("\\rule[-7pt]{0pt}{20pt} \\\\ \\hline \n");
    elif z = Size(samexs)
     then Print(" \\rule[-7pt]{0pt}{5pt} \\\\ \n");
    else Print(" \\\\ \n");
    fi;
  od;
  Print("\\hline \n");
  Print("\\end{array} \n");
  Print("\\]");
  Print("\n\\label{table:S",num,".",Int(z/38)+1,k,"}
    \n\\end{table} \n\n");
od;

Print("The ",Size(sameinds)," I-equivalency irreducible 
  character classes of $D(S_",num,")$ are:\n");
Print("\\newline \n");
for i in [1..Size(sameinds)] do
 if Size(samexs[i]) > 1 then
  for j in [1..Size(samexs[i])-1] do
   if j=1 then Print("[ $",samexs[i][j],","); 
   elif IsInt(j/13) then
    Print(samexs[i][j],"$, \n\n\\hspace{.4in}$");
   else Print(samexs[i][j],",");
   fi;
   if IsInt(j/4) then Print("\n"); fi;
  od;
   Print(samexs[i][Size(samexs[i])],"$ ],\n");
 else
  Print("[ $",samexs[i][1],"$ ],\n");
 fi;
 Print("\n");
od;
return;
end;
\end{verbatim}

\subsection{}
\bf Additional Functions \rm   %-------------------5.4-------------------

In this section we provide the code for two more functions we found useful in when writing this dissertation.  The first function steps through how $\tilde{G}_m^m(u)$ is computed from $\tilde{G}_m(u)$.  We used this function when writing Sections 4.1 and 4.2 when we gave the lists of elements in both of these sets.  The second function became necessary when we started copying our tables of indicators and saw they were too large to fit on the page.  The second function allows us to specify how many columns should be in each indicator table. 

\begin{df} \label{HmuFullList} HmuFullList(G, m, u) \rm This function is almost identical to Hmu, but it prints out the set 
of $\tilde{G}_m(u)$, and then $\tilde{G}_m^m(u)$ in a number of different ways.  The code is provided below.
\end{df}
\begin{verbatim}
HmuFullList:= function(G,m,u)
local GG, Gm, H, i, Hm, HM, Elem, j, Cent, CC, cc, sum;
if u = () then
 Cent:= G;
 CC:=ConjugacyClasses(G);
 Hm:=List(CC, x -> [Size(x),(Representative(x))^m]);
 Print("Since u = (), we will not print the full list of 
  $\\tilde{G}_",m,"(",u,")$, instead here is \n a condensed
  list of pairs, where the \"first\" is the number of elements 
  that \n when raised to the ",m," are equal to the
  \"second\".\n");
else
 GG:=EnumeratorSorted(G);
 Gm:=List(GG, x->Position(GG, x^m));
 H:=List(GG, x->[]);
 for i in [1..Size(GG)] do
  if Gm[Position(GG, u*GG[i])] = Gm[i] then
   Add(H[Gm[i]], i);;
  fi;
 od;
 Hm:=[];HM:=[];
 for j in [1..Size(H)] do
  if Size(H[j]) <> 0 then
   Add( Hm, [ Size(H[j]), GG[j] ] );
   Elem:=List(H[j], x->GG[x]);
   Add( HM, Elem);
  fi;
 od;
 Print("$\\tilde{G}_",m,"(",u,")$ = ",Concatenation(HM)," \n");
 Print("Or if we group the elements of $\\tilde{G}_",m,"(",u,")$
  into subsets we get ",HM," \n");
 Print("Condensing this, we get a list of pairs, where the
  \"first\" is the number of \n elements that when raised to the
  ",m," are equal to the \"second\". \n",Hm,"\n Condensing 
  further to combine elements in the same conjugacy class we get
  \n");
 Cent:=Centralizer(G, u);
 CC:=ConjugacyClasses(Cent);
fi;
H:=List(Hm, x-> [x[1], Position(CC, ConjugacyClass(Cent,x[2]))]);
 cc:=List(H, x-> x[2]);
 Hm:=[]; 
 for i in DuplicateFreeList(cc) do
  sum:=0;;
  for j in Positions(cc,i) do
   sum:= sum + H[j][1];
  od;
  Add( Hm, [sum, Representative(CC[i])]);
 od;
return Hm;
end;
\end{verbatim}

\begin{df} \label{FSLaTexTabsSpecific} FSLaTexTabsSpecific(M, num, exp, tablesizes) \rm This function is very similar to FSLaTexTabs in that it gives a text that can be directly copied and pasted into LaTex containing the higher Frobenius Schur indicator table for $D(S_num)$. However, there are three important differences.  

First instead of calling other functions to compute the indicators, you are required to input a Matrix that contains the reduced matrix of indicators (we do this by either inputing FSindicators(exp,num) as M or in the case of $S_{10}$ we can input the matrix that took multiple days to compute - see Section 5).  Second the output is in LaTex tables and not arrays, and third this function alows you to specify how many columns you want per table. 

Tablesizes is a list of the number of m values you want in each table.  The code for programming this function in GAP is provided below.
\end{df}
\begin{verbatim}
FSLaTexTabsSpecific:=function(M, num, exp,tablesizes)
local factors, samexs, sameinds, numtables, columns, j, k, range,
 lowr, highr, z, l, i;
factors:=DivisorsInt(exp);
if Sum(tablesizes) <> Size(factors) then
 Print("There are ",Size(factors)," many different m values, but 
  the column sizes of the tables you provided sum up to ",
  Sum(tablesizes),".  Please try again.\n");
 return;
else
samexs:=M[1];  sameinds:=M[2]; 
# Now for the display:
numtables:=Size(tablesizes);
columns:="";
for j in [1..Size(factors)] do
 columns:=Concatenation(columns,"c");
od;
for k in [1..numtables] do
if k = 1 then 
 range:=[1..tablesizes[k]];
 lowr:=1;
else
 lowr:=1; highr:=0;
 for i in [1..k-1] do
  lowr:= lowr+tablesizes[i];
  highr:= highr+tablesizes[i];
 od;
 range:=[lowr..highr+tablesizes[k]]; 
fi;
if k = 1 then 
 Print("\\begin{table}[ht] \n \\caption{$D(S_{",num,"})$:
  (exponent: $",exp,"$) Set 1} \n");
 Print("\\[ \n"); 
else
 Print("\\begin{table}[ht] \n \\caption{$D(S_{",num,"})$:
  $m = ",factors[range[1]]," \\ldots ",
  factors[range[Size(range)]], "$ Set 1} \n \\[ \n");
fi;
Print("\\begin{array}{r|",columns{range},"} \\hline \n");
for j in range do  
 if j = lowr then Print("m = "); fi;
 Print(" & ",factors[j]," \n");
od;
Print("\\rule[-7pt]{0pt}{20pt} \\\\ \\hline \n");
for z in [1..Size(samexs)] do
Print("\\nu_m(",samexs[z][1],")");
 for l in range do
  Print(" & ",sameinds[z][l]);  
 od;
 if z = 1 or IsInt((z-1)/38) then 
  Print(" \\rule[0pt]{0pt}{13pt} \\\\ \n");
 elif IsInt(z/38) then 
     # This means each 38 rows of irreducible characters, 
     # start a new table.
  Print(" \\rule[-7pt]{0pt}{5pt} \\\\ \n");
  Print("\\hline \n");
  Print("\\end{array} \n");
  Print("\\]\n\\label{table:S",num,".",Int(z/38),k,"}\n
   \\end{table} \n\n");
# end table and start a new table.
  Print("\\begin{table}[ht] \n \\caption{$D(S_{",num,"})$:
   $m = ", factors[range[1]]," \\ldots ", 
   factors[range[Size(range)]],"$ Set ", 1+Int(z/38),"}
   \n\\[ \n"); 
  Print("\\begin{array}{r|",columns{range},"} \\hline \n");
  for j in range do  
   if j = lowr then Print("m = "); fi;
   Print(" & ",factors[j]," \n");
  od;
  Print("\\rule[-7pt]{0pt}{20pt} \\\\ \\hline \n");
 elif z = Size(samexs) then
  Print(" \\rule[-7pt]{0pt}{5pt} \\\\ \n");
 else Print(" \\\\ \n");
 fi;
od;
Print("\\hline \n");
Print("\\end{array} \n");
Print("\\]");
Print("\n\\label{table:S",num,".",Int(z/38)+1,k,"}\n\\end{table}
 \n\n");
od;
Print("The ",Size(sameinds)," I-equivalency irreducible character
 classes of $D(S_",num,")$ are:\n");
Print("\\newline \n");
for i in [1..Size(sameinds)] do
 if Size(samexs[i]) > 1 then
  for j in [1..Size(samexs[i])-1] do
   if j=1 then Print("[ $",samexs[i][j],","); 
   elif IsInt(j/13) then
    Print(samexs[i][j],"$, \n\n\\hspace{.4in}$");
   else Print(samexs[i][j],",");
   fi;
   if IsInt(j/4) then Print("\n"); fi;
  od;
   Print(samexs[i][Size(samexs[i])],"$ ],\n");
 else
  Print("[ $",samexs[i][1],"$ ],\n");
 fi;
 Print("\n");
od;
return;
fi;
end;\end{verbatim}

\subsection{}
\bf Computing the Indicators of $D(S_{10})$ \rm   %-------------------5.5-------------------

To compute all the indicators of $D(S_n)$ for $n \leq 9 $ we used either FSLaTexTabs (Definition \ref{FSLaTexTabs}) or FSLaTexTabsSpecific (Definition \ref{FSLaTexTabsSpecific}), but when we attempted to use either one for computing the indicators of $D(S_{10})$, GAP produced an error message that said, ``gap: cannot extend the workspace any more ."  In short, $D(S_{10})$ was too large to compute and store all the indicators at once.  

Of all the previously given functions the only ones that did not give error messages when working with $S_{10}$ were Hmu (Definition \ref{Hmu}) and FSGmu (Definition \ref{FSGmu}).  Even attempting to use FSGu (Definition \ref{FSGu}) where we used FSGmu repeatedly to compute and store all indicators for a specific conjugacy class representative $u$ resulted in error messages indicating there was not enough space to compute and store these all at once.

\indent Since FSGu worked, we used the following command propmts and functions to calculate all the indicators of $D(S_{10})$, saved them in a single matrix M, and then used FSLaTexTabsSpecific(M, $n$, $exp$, [list of table sizes]) to print our tables. 

{\bf (1)}  First we used the command prompt code displayed below to compute all the indicators for a specific conjugacy class representative $u \in S_{10}$, one $m$ value at a time (where $m$ was a divisor of $exp = 2520$ the exponent of $S_{10}$).  Each computation for a specific $m$ took between 2 and 4 minutes.  The advantage of having each line print out one value of $m$ at a time, was we were able to save the progress in case our computer lost power, over heated or had issues before the full computation was complete.  When $u=(1\ 2)$, the code we used was:

\begin{verbatim}
>for w in DivisorsInt(9*8*7*5) do 
    Print(FSGmu(SymmetricGroup(10),w,(1,2)),",\n");
  od;
\end{verbatim}
The first four lines of its output looked like this:
\begin{verbatim}
[ 0, 0, 0, 0, 0, 0, 0, 0, 0, 0, 0, 0, 0, 0, 0, 0, 0, 0, 0, 0, 0,
  0, 0, 0, 0, 0, 0, 0, 0, 0, 0, 0, 0, 0, 0, 0, 0, 0, 0, 0, 0, 0,
  0, 0 ],
[ 1, 1, 1, 1, 1, 1, 1, 1, 1, 1, 1, 1, 1, 1, 1, 1, 1, 1, 1, 1, 1,
  1, 1, 1, 1, 1, 1, 1, 1, 1, 1, 1, 1, 1, 1, 1, 1, 1, 1, 1, 1, 1,
  1, 1 ],
[ 0, 0, 0, 0, 0, 0, 0, 0, 0, 0, 0, 0, 0, 0, 0, 0, 0, 0, 0, 0, 0,
  0, 0, 0, 0, 0, 0, 0, 0, 0, 0, 0, 0, 0, 0, 0, 0, 0, 0, 0, 0, 0,
  0, 0 ],
[ 2, 2, 2, 2, 6, 6, 6, 6, 8, 8, 8, 8, 12, 12, 12, 12, 13, 13,
  13, 13, 14, 14, 14, 14, 21, 21, 21, 21, 19, 19, 28, 28, 
  28, 28, 33, 33, 33, 33, 34, 34, 34, 34, 45, 45 ],
\end{verbatim}

{\bf (2)}  Next we saved each of these sets of data as a matrix.  Each of these matrices of indicators of irreducible characters induced from one centralizer took between 2.5 and 4.5 hours.  Since there are 42 conjugacy classes of $S_{10}$, it took about a week to compute all 42 matrices of indicators.  Once we had all the matrices, we then put each matrix through a funciton that added the corresponding characters to the matrix, found the I-equivalency classes, and condensed the matrix so there were no identical rows of indicators.

\begin{df}FSuMatRed \rm   This function takes the matrix of all calculated $ m^{\text{th}}$ Frobenius-Schur indicators of all the irreducible characters of $D(G)$ induced from the irreducbile characters of the centralizer of an element $u$ in $G$, and it returns a list of the I-equivalency classes and the corresponding matrix of their indicator values. The number $i$ indicates the number of the conjugacy class that $u$ comes from.  The first half of this function simply adds the list of characters to the matrix as they correspond to the indiciator values.  The second half of this function is identiacal to FSMatRed (Definition \ref{FSMatRed}).  The code for programming this function in GAP is provided below.  
\end{df} 
\begin{verbatim}
FSuMatRed:=function(MAT,i)
local Xs, j, st, M, INDS, XS, samexs, sameinds, count, temp,
 pos, k;
Xs:=[];
# first we add the list of characters to this matrix
for j in [1..Size(MAT[1])] do 
# number of columns in MAT[1] is number of irreducible characters
# of D(G) induced from u.
 st:=String(j);                 
 st:=Concatenation("\\chi_{",String(i),".",st,"}");   
 Xs[j]:=st;                 
od;
Add(MAT,Xs,1);       
M:=TransposedMat(MAT);
INDS:= M{[1..Size(M)]}{[2..Size(M[1])]};
XS:= M{[1..Size(M)]}{[1]};
samexs:=[]; sameinds:=[]; count:=1; temp:=[];
for k in [1..Size(INDS)] do
 if INDS[k] in sameinds then       #do nothing
 else       #if INDS[k] is not in the list sameinds yet, add it
  sameinds[count]:=INDS[k];
  temp:=Positions(INDS,INDS[k]);
  for j in [1..Size(temp)] do
   temp[j]:=XS[temp[j]][1];  
  od;
  samexs[count]:= temp;
  count:=count+1;
 fi;
od;
return [samexs, sameinds];
end;
\end{verbatim}

{\bf (3)}  After using FSuMatRed on all 42 of our matrices, we used the ``Concatenation" command in GAP to make one long list of matrix pairs - each pair consisting of the list of I-equivalency classes and the matrix of corresponding indicator values.  We then put that list through the following function to check for any mixed irreducible character I-equivalence classes. 

\begin{df} FSMat \rm This function takes a list of pairs and combines them while removing and accounting for duplications.  Each pair contains the list of irreducible character I-equivalency classes of $D(S_{10})$ as computed from one centralizer and the matrix of Frobenius Schur indicators of these classes.  This function is somewhat similar to FSIndicators (Definition  \ref{FSIndicators} The code for programming this function in GAP is provided below.  
\end{df}
\begin{verbatim}
FSMat:=function(Mats)
local size, i, INDS, XS, samexs, sameinds, count, temp, k, pos, j;
size:=Size(Mats)/2;
if IsInt(size) then 
 INDS:=[]; XS:=[];
 for i in [1.. size] do
  INDS:= Concatenation(INDS,Mats[2*i]);
  XS:= Concatenation(XS,Mats[2*i-1]);
 od;
 samexs:=[]; sameinds:=[]; count:=1; temp:=[];
 for k in [1..Size(INDS)] do
  if INDS[k] in sameinds then       #do nothing
  else       #if INDS[k] is not in the list sameinds yet, add it
   sameinds[count]:=INDS[k];
   pos:=Positions(INDS,INDS[k]); 
   temp:=[];
   for j in [1..Size(pos)] do
    temp:= Concatenation(temp,XS[pos[j]]);  
   od;
   samexs[count]:= temp;
   count:=count+1;
  fi;
 od;
 return [samexs, sameinds];
else return "bad input";
fi;
end;
\end{verbatim}

{\bf (4)} The result from FSMat gave us the matrix M, consisting of first the list of all the I-equivalency classes of $D(S_{10})$ and second the matrix of corresponding indicator values, that we needed to run FSLaTexTabsSpecific (Definition \ref{FSLaTexTabsSpecific}).

Attempting to run any of our functions using $S_{11}$ resulted in errors from GAP. \newline

%-------------Appendix A-------------------------------------
\appendix
\section{Character Tables}
This appendix provides all character tables used in calculating all the higher indicators of $D(S_3)$ to $D(S_6)$.  Although the characters of these centrailzers are well known, the order in which they are displayed in tables varies from source to source.  Therefore, we include them here to make our notation of the characters of $D(S_n)$ clear.  Note that in the tables that follow, for n = 3,4 and 5, E$(n)$ denotes a primitive $n^\text{th}$ root of 1.  

\subsection{}
\bf Character Tables of the Centralizers of $S_3$ \rm

The character tabels of the centralizers in $S_3$ of conjugacy class representatives of $S_3$ are given below.  These tables were used to calculate the indicators of $D(S_3)$ recorded in Section 4.1.

\begin{table}[ht]  \caption{Character tables of Centralizers of $S_3$} 
$\begin{array}{ccc}
C_{S_3}(()) \cong S_3: & C_{S_3}((1,2))\cong C_2:  & C_{S_3}((1,2,3))\cong C_3: \\
\begin{array}{r|ccc} \hline 
 &  ()
 &  (1,2)
 & (1,2,3) 
\rule[-7pt]{0pt}{20pt} \\ \hline 
\eta_{1.1} & 1 & -1 & 1 \rule[0pt]{0pt}{13pt} \\ 
\eta_{1.2} & 2 & 0 & -1 \\ 
\eta_{1.3} & 1 & 1 & 1 \rule[-7pt]{0pt}{5pt} \\ 
\hline 
\end{array} 
&
\begin{array}{c}
\begin{array}{r|cc} \hline 
 & ()
 & (1,2) 
\rule[-7pt]{0pt}{20pt} \\ \hline 
\eta_{2.1} & 1 & -1 \rule[0pt]{0pt}{13pt} \\ 
\eta_{2.2} & 1 & 1 \rule[-7pt]{0pt}{5pt} \\ 
\hline 
\end{array} \\ \\
\end{array}
&
\begin{array}{r|ccc} \hline 
 &  ()
 & (1,2,3) 
 &  (1,3,2) 
\rule[-7pt]{0pt}{20pt} \\ \hline 
\eta_{3.1} & 1 & 1 & 1 \rule[0pt]{0pt}{13pt} \\ 
\eta_{3.2} & 1 & E(3) & E(3)^2 \\ 
\eta_{3.3} & 1 & E(3)^2 & E(3) \rule[-7pt]{0pt}{5pt} \\ 
\hline 
\end{array} 
\end{array}
$ %\newline
\end{table}

\subsection{}
\bf Character Tables of the Centralizers of $S_4$ \rm

The character tabels of the centralizers in $S_4$ of conjugacy class representatives of $S_4$ are given below.  These tables were used to calculate the indicators of $D(S_4)$ recorded in Section 4.2.
%\newpage

\begin{table}[ht]  \caption{Character table of $C_{S_4}(())\cong S_4$}
$\begin{array}{r|ccccc} \hline 
 &   ()
 &  (1,2)
 &  (1,2)(3,4)
 &  (1,2,3)
 &  (1,2,3,4)
\rule[-7pt]{0pt}{20pt} \\ \hline 
\eta_{1.1} & 1 & -1 & 1 & 1 & -1 \rule[0pt]{0pt}{13pt} \\ 
\eta_{1.2} & 3 & -1 & -1 & 0 & 1 \\ 
\eta_{1.3} & 2 & 0 & 2 & -1 & 0 \\ 
\eta_{1.4} & 3 & 1 & -1 & 0 & -1 \\ 
\eta_{1.5} & 1 & 1 & 1 & 1 & 1 \rule[-7pt]{0pt}{5pt} \\ 
\hline 
\end{array} $ \end{table}

%\vs

\begin{table}[ht]  \caption{Character table of $C_{S_4}((1,2))\cong \langle (1,2), (3,4)\rangle \cong C_2 \times C_2$}
$\begin{array}{r|cccc} \hline 
 & () 
 & (3,4) 
 &  (1,2) 
 & (1,2)(3,4) 
\rule[-7pt]{0pt}{20pt} \\ \hline 
\eta_{2.1} & 1 & 1 & 1 & 1 \rule[0pt]{0pt}{13pt} \\ 
\eta_{2.2} & 1 & -1 & -1 & 1 \\ 
\eta_{2.3} & 1 & -1 & 1 & -1 \\ 
\eta_{2.4} & 1 & 1 & -1 & -1 \rule[-7pt]{0pt}{5pt} \\ 
\hline 
\end{array}$ \end{table}
%\clearpage

\begin{table}[ht]  \caption{Character table of $C_{S_4}((1,2)(3,4)) \cong  \langle (1,2), (1,3)(2,4), (3,4)\rangle \cong D_8$}
$\begin{array}{r|ccccc} \hline 
 & () & (3,4) & (1,2)(3,4) & (1,3)(2,4) & (1,3,2,4)
\rule[-7pt]{0pt}{20pt} \\ \hline 
\eta_{3.1} & 1 & 1 & 1 & 1 & 1 \rule[0pt]{0pt}{13pt} \\ 
\eta_{3.2} & 1 & -1 & 1 & -1 & 1 \\ 
\eta_{3.3} & 1 & -1 & 1 & 1 & -1 \\ 
\eta_{3.4} & 1 & 1 & 1 & -1 & -1 \\ 
\eta_{3.5} & 2 & 0 & -2 & 0 & 0 \rule[-7pt]{0pt}{5pt} \\ 
\hline 
\end{array} $ \end{table}

\begin{table}[ht]  \caption{Character table of $C_{S_4}((1,2,3)) \cong \langle (1,2,3) \rangle \cong C_3$}
$\begin{array}{r|ccc} \hline 
 & () & (1,2,3) & (1,3,2) 
\rule[-7pt]{0pt}{20pt} \\ \hline 
\eta_{4.1} & 1 & 1 & 1 \rule[0pt]{0pt}{13pt} \\ 
\eta_{4.2} & 1 & E(3) & E(3)^2 \\ 
\eta_{4.3} & 1 & E(3)^2 & E(3) \rule[-7pt]{0pt}{5pt} \\ 
\hline 
\end{array} $ \end{table} 

\clearpage

\begin{table}[ht]  \caption{Character table of $C_{S_4}((1,2,3,4)) \cong  \langle (1,2,3,4)\rangle \cong C_4$}
$\begin{array}{r|cccc} \hline 
 & () & (1,2,3,4) & (1,3)(2,4) & (1,4,3,2)
\rule[-7pt]{0pt}{20pt} \\ \hline 
\eta_{5.1} & 1 & 1 & 1 & 1 \rule[0pt]{0pt}{13pt} \\ 
\eta_{5.2} & 1 & -1 & 1 & -1 \\ 
\eta_{5.3} & 1 & E(4) & -1 & -E(4) \\ 
\eta_{5.4} & 1 & -E(4) & -1 & E(4) \rule[-7pt]{0pt}{5pt} \\ 
\hline 
\end{array} $ \end{table} 

%\newpage

\subsection{}
\bf Character Tables of the Centralizers of $S_5$ \rm

The character tabels of the centralizers in $S_5$ of conjugacy class representatives of $S_5$ are given below.  These tables were used to calculate the indicators of $D(S_5)$ recorded in Section 4.3.

\begin{table}[ht]  \caption{Character table of $C_{S_5}(())\cong S_5$}
$
\begin{array}{r|ccccccc} \hline 
 &  () & (1,2) & (1,2)(3,4) & (1,2,3) & (1,2,3)(4,5) & (1,2,3,4) & (1,2,3,4,5)
\rule[-7pt]{0pt}{20pt} \\ \hline 
\eta_{1.1} & 1 & -1 & 1 & 1 & -1 & -1 & 1 \rule[0pt]{0pt}{13pt} \\ 
\eta_{1.2} & 4 & -2 & 0 & 1 & 1 & 0 & -1 \\ 
\eta_{1.3} & 5 & -1 & 1 & -1 & -1 & 1 & 0 \\ 
\eta_{1.4} & 6 & 0 & -2 & 0 & 0 & 0 & 1 \\ 
\eta_{1.5} & 5 & 1 & 1 & -1 & 1 & -1 & 0 \\ 
\eta_{1.6} & 4 & 2 & 0 & 1 & -1 & 0 & -1 \\ 
\eta_{1.7} & 1 & 1 & 1 & 1 & 1 & 1 & 1 \rule[-7pt]{0pt}{5pt} \\ 
\hline 
\end{array} 
$ \end{table}
%\clearpage

\begin{table}[ht]  \caption{Character table of $C_{S_5}((1,2))\cong \langle (1,2), (3,5), (4,5)\rangle \cong  D_{12}$}
$
\begin{array}{r|cccccc} \hline 
 &  () & (4,5) & (3,4,5) & (1,2) & (1,2)(4,5) & (1,2)(3,4,5)
\rule[-7pt]{0pt}{20pt} \\ \hline 
\eta_{2.1} & 1 & 1 & 1 & 1 & 1 & 1 \rule[0pt]{0pt}{13pt} \\ 
\eta_{2.2} & 1 & -1 & 1 & -1 & 1 & -1 \\ 
\eta_{2.3} & 1 & -1 & 1 & 1 & -1 & 1 \\ 
\eta_{2.4} & 1 & 1 & 1 & -1 & -1 & -1 \\ 
\eta_{2.5} & 2 & 0 & -1 & -2 & 0 & 1 \\ 
\eta_{2.6} & 2 & 0 & -1 & 2 & 0 & -1 \rule[-7pt]{0pt}{5pt} \\ 
\hline 
\end{array} 
$ \end{table}
%\newpage

\begin{table}[ht]  \caption{Character table of $C_{S_5}((1,2)(3,4))\cong \langle (1,2), (1,3)(2,4), (3,4) \rangle \cong D_8$}
$
\begin{array}{r|ccccc} \hline 
 & () & (3,4) & (1,2)(3,4) & (1,3)(2,4) & (1,3,2,4)
\rule[-7pt]{0pt}{20pt} \\ \hline 
\eta_{3.1} & 1 & 1 & 1 & 1 & 1 \rule[0pt]{0pt}{13pt} \\ 
\eta_{3.2} & 1 & -1 & 1 & -1 & 1 \\ 
\eta_{3.3} & 1 & -1 & 1 & 1 & -1 \\ 
\eta_{3.4} & 1 & 1 & 1 & -1 & -1 \\ 
\eta_{3.5} & 2 & 0 & -2 & 0 & 0 \rule[-7pt]{0pt}{5pt} \\ 
\hline 
\end{array} 
$\end{table}

\begin{table}[ht]  \caption{Character table of $C_{S_5}((1,2,3)) \cong \langle (1,2,3), (4,5) \rangle \cong C_6$}
$
\begin{array}{r|cccccc} \hline 
 & () & (4,5) & (1,2,3) & (1,2,3)(4,5) & (1,3,2) & (1,3,2)(4,5)
\rule[-7pt]{0pt}{20pt} \\ \hline 
\eta_{4.1} & 1 & 1 & 1 & 1 & 1 & 1 \rule[0pt]{0pt}{13pt} \\ 
\eta_{4.2} & 1 & -1 & 1 & -1 & 1 & -1 \\ 
\eta_{4.3} & 1 & -1 & E(3)^2 & -E(3)^2 & E(3) & -E(3) \\ 
\eta_{4.4} & 1 & -1 & E(3) & -E(3) & E(3)^2 & -E(3)^2 \\ 
\eta_{4.5} & 1 & 1 & E(3)^2 & E(3)^2 & E(3) & E(3) \\ 
\eta_{4.6} & 1 & 1 & E(3) & E(3) & E(3)^2 & E(3)^2 \rule[-7pt]{0pt}{5pt} \\ 
\hline 
\end{array} 
$\end{table}

The character table of $C_{S_5}((1,2,3)(4,5)) \cong \langle (1,2,3), (4,5) \rangle \cong C_6$ is identical to the table given just above, with the only exception being the characters are indexed by $5.j$.  Each of the 6 irreducible characters $\eta_{5.j}$ has the same character values as the corresponding irreducible characters $\eta_{4.j}$ listed in the table immediately above.

\begin{table}[ht]  \caption{Character table of $C_{S_5}((1,2,3,4))\cong  \langle (1,2,3,4)\rangle \cong C_4$}
$
\begin{array}{r|cccc} \hline 
 &  () & (1,2,3,4) & (1,3)(2,4) & (1,4,3,2)
\rule[-7pt]{0pt}{20pt} \\ \hline 
\eta_{6.1} & 1 & 1 & 1 & 1 \rule[0pt]{0pt}{13pt} \\ 
\eta_{6.2} & 1 & -1 & 1 & -1 \\ 
\eta_{6.3} & 1 & E(4) & -1 & -E(4) \\ 
\eta_{6.4} & 1 & -E(4) & -1 & E(4) \rule[-7pt]{0pt}{5pt} \\ 
\hline 
\end{array} 
$\end{table}

\begin{table}[ht]  \caption{Character table of $C_{S_5}((1,2,3,4,5)) \cong \langle (1,2,3,4,5) \rangle \cong C_5$}
$
\begin{array}{r|ccccc} \hline 
 &  () & (1,2,3,4,5) & (1,3,5,2,4) & (1,4,2,5,3) & (1,5,4,3,2)
\rule[-7pt]{0pt}{20pt} \\ \hline 
\eta_{7.1} & 1 & 1 & 1 & 1 & 1 \rule[0pt]{0pt}{13pt} \\ 
\eta_{7.2} & 1 & E(5) & E(5)^2 & E(5)^3 & E(5)^4 \\ 
\eta_{7.3} & 1 & E(5)^2 & E(5)^4 & E(5) & E(5)^3 \\ 
\eta_{7.4} & 1 & E(5)^3 & E(5) & E(5)^4 & E(5)^2 \\ 
\eta_{7.5} & 1 & E(5)^4 & E(5)^3 & E(5)^2 & E(5) \rule[-7pt]{0pt}{5pt} \\ 
\hline 
\end{array} 
$\end{table}
%\newpage
%

\subsection{}
\bf Character Tables of the Centralizers of $S_6$ \rm

The character tabels of the centralizers in $S_6$ of conjugacy class representatives of $S_6$ are given below.  These tables were used to calculate the indicators of $D(S_6)$ recorded in Section 4.4.
\clearpage

\begin{table}[ht]  \caption{Character table of $C_{S_6}(())\cong S_6$}
$
\begin{array}{r|ccccccccccc} \hline 
 & 1a 
 & 2a 
 & 2b 
 & 2c 
 & 3a 
 & 6a 
 & 3b 
 & 4a 
 & 4b 
 & 5a 
 & 6b 
\rule[-7pt]{0pt}{20pt} \\ \hline 
\eta_{1.1} & 1 & -1 & 1 & -1 & 1 & -1 & 1 & -1 & 1 & 1 & 
-1 \rule[0pt]{0pt}{13pt} \\ 
\eta_{1.2} & 5 & -3 & 1 & 1 & 2 & 0 & -1 & -1 & -1 & 0 & 1 \\ 
\eta_{1.3} & 9 & -3 & 1 & -3 & 0 & 0 & 0 & 1 & 1 & -1 & 0 \\ 
\eta_{1.4} & 5 & -1 & 1 & 3 & -1 & -1 & 2 & 1 & -1 & 0 & 0 \\ 
\eta_{1.5} & 10 & -2 & -2 & 2 & 1 & 1 & 1 & 0 & 0 & 0 & -1 \\ 
\eta_{1.6} & 16 & 0 & 0 & 0 & -2 & 0 & -2 & 0 & 0 & 1 & 0 \\ 
\eta_{1.7} & 5 & 1 & 1 & -3 & -1 & 1 & 2 & -1 & -1 & 0 & 0 \\ 
\eta_{1.8} & 10 & 2 & -2 & -2 & 1 & -1 & 1 & 0 & 0 & 0 & 1 \\ 
\eta_{1.9} & 9 & 3 & 1 & 3 & 0 & 0 & 0 & -1 & 1 & -1 & 0 \\ 
\eta_{1.10} & 5 & 3 & 1 & -1 & 2 & 0 & -1 & 1 & -1 & 0 & -1 \\ 
\eta_{1.11} & 1 & 1 & 1 & 1 & 1 & 1 & 1 & 1 & 1 & 1 & 
1 \rule[-7pt]{0pt}{5pt} \\ 
\hline 
\end{array}  
$ \end{table}
Here 1a = (), 2a = (1,2), 2b = (1,2)(3,4), 2c = (1,2)(3,4)
(5,6), 3a = (1,2,3), 6a = (1,2,3)(4,5), 3b = (1,2,3)(4,5,6), 4a = 
(1,2,3,4), 4b = (1,2,3,4)(5,6), 5a = (1,2,3,4,5), 6b = (1,2,3,4,5,6).\newline
%\newpage

\begin{table}[ht]  \caption{Character table of $C_{S_6}((1,2))\cong \langle (1,2), (3,6), (4,6), (5,6)\rangle$}
$
\begin{array}{r|cccccccccc} \hline 
 & 1a 
 & 2a 
 & 3a 
 & 2b 
 & 4a 
 & 2c 
 & 2d 
 & 6a 
 & 2e 
 & 4b 
\rule[-7pt]{0pt}{20pt} \\ \hline 
\eta_{2.1} & 1 & 1 & 1 & 1 & 1 & 1 & 1 & 1 & 1 & 1 \rule[0pt]{0pt}{13pt} \\ 
\eta_{2.2} & 1 & -1 & 1 & 1 & -1 & -1 & 1 & -1 & -1 & 1 \\ 
\eta_{2.3} & 1 & -1 & 1 & 1 & -1 & 1 & -1 & 1 & 1 & -1 \\ 
\eta_{2.4} & 1 & 1 & 1 & 1 & 1 & -1 & -1 & -1 & -1 & -1 \\ 
\eta_{2.5} & 2 & 0 & -1 & 2 & 0 & -2 & 0 & 1 & -2 & 0 \\ 
\eta_{2.6} & 2 & 0 & -1 & 2 & 0 & 2 & 0 & -1 & 2 & 0 \\ 
\eta_{2.7} & 3 & -1 & 0 & -1 & 1 & -3 & 1 & 0 & 1 & -1 \\ 
\eta_{2.8} & 3 & -1 & 0 & -1 & 1 & 3 & -1 & 0 & -1 & 1 \\ 
\eta_{2.9} & 3 & 1 & 0 & -1 & -1 & -3 & -1 & 0 & 1 & 1 \\ 
\eta_{2.10} & 3 & 1 & 0 & -1 & -1 & 3 & 1 & 0 & -1 & 
-1 \rule[-7pt]{0pt}{5pt} \\ 
\hline 
\end{array} 
$ \end{table}
Here 1a = (), 2a = (5,6), 3a = (4,5,6), 2b = (3,4)(5,6), 4a = (3,4,5,6), 2c = (1,2), 2d = (1,2)(5,6), 6a = (1,2)(4,5,6), 2e = (1,2)(3,4)(5,6), 4b = (1,2)(3,4,5,6).

\clearpage

\begin{table}[ht]  \caption{Character table of $C_{S_6}((1,2)(3,4))\cong \langle (1,2), (1,3)(2,4), (3,4), (5,6) \rangle \cong D_8  \times C_2$}
$
\begin{array}{r|cccccccccc} \hline 
 & 1a 
 & 2a 
 & 2b 
 & 2c 
 & 2d 
 & 2e 
 & 2f 
 & 2g 
 & 4a 
 & 4b 
\rule[-7pt]{0pt}{20pt} \\ \hline 
\eta_{3.1} & 1 & 1 & 1 & 1 & 1 & 1 & 1 & 1 & 1 & 1 \rule[0pt]{0pt}{13pt} \\ 
\eta_{3.2} & 1 & -1 & -1 & 1 & 1 & -1 & -1 & 1 & 1 & -1 \\ 
\eta_{3.3} & 1 & -1 & -1 & 1 & 1 & -1 & 1 & -1 & -1 & 1 \\ 
\eta_{3.4} & 1 & -1 & 1 & -1 & 1 & -1 & -1 & 1 & -1 & 1 \\ 
\eta_{3.5} & 1 & -1 & 1 & -1 & 1 & -1 & 1 & -1 & 1 & -1 \\ 
\eta_{3.6} & 1 & 1 & -1 & -1 & 1 & 1 & -1 & -1 & 1 & 1 \\ 
\eta_{3.7} & 1 & 1 & -1 & -1 & 1 & 1 & 1 & 1 & -1 & -1 \\ 
\eta_{3.8} & 1 & 1 & 1 & 1 & 1 & 1 & -1 & -1 & -1 & -1 \\ 
\eta_{3.9} & 2 & 2 & 0 & 0 & -2 & -2 & 0 & 0 & 0 & 0 \\ 
\eta_{3.10} & 2 & -2 & 0 & 0 & -2 & 2 & 0 & 0 & 0 & 0 \rule[-7pt]{0pt}{5pt} \\ 
\hline 
\end{array} 
$\end{table}
Here 1a = (), 2a = (5,6), 2b = (3,4), 2c = (3,4)(5,6), 2d = (1,2)(3,4), 2e = (1,2)(3,4)(5,6), 2f = (1,3)(2,4), 2g = (1,3)(2,4)(5,6), 4a = (1,3,2,4), 4b = (1,3,2,4)(5,6).
%\newpage

\begin{table}[ht]  \caption{Character table of $C_{S_6}((1,2)(3,4)(5,6)) \cong \langle (1,2), (1,5)(2,6), (3,4), (3,5)(4,6), (5,6) \rangle$}
$
\begin{array}{r|cccccccccc} \hline 
 & 1a 
 & 2a 
 & 2b 
 & 2c 
 & 4a 
 & 2d 
 & 2e 
 & 4b 
 & 3a 
 & 6a 
\rule[-7pt]{0pt}{20pt} \\ \hline 
\eta_{4.1} & 1 & 1 & 1 & 1 & 1 & 1 & 1 & 1 & 1 & 1 \rule[0pt]{0pt}{13pt} \\ 
\eta_{4.2} & 1 & -1 & 1 & -1 & 1 & -1 & 1 & -1 & 1 & -1 \\ 
\eta_{4.3} & 1 & -1 & 1 & 1 & -1 & -1 & -1 & 1 & 1 & -1 \\ 
\eta_{4.4} & 1 & 1 & 1 & -1 & -1 & 1 & -1 & -1 & 1 & 1 \\ 
\eta_{4.5} & 2 & -2 & 2 & 0 & 0 & -2 & 0 & 0 & -1 & 1 \\ 
\eta_{4.6} & 2 & 2 & 2 & 0 & 0 & 2 & 0 & 0 & -1 & -1 \\ 
\eta_{4.7} & 3 & -1 & -1 & -1 & 1 & 3 & -1 & 1 & 0 & 0 \\ 
\eta_{4.8} & 3 & -1 & -1 & 1 & -1 & 3 & 1 & -1 & 0 & 0 \\ 
\eta_{4.9} & 3 & 1 & -1 & -1 & -1 & -3 & 1 & 1 & 0 & 0 \\ 
\eta_{4.10} & 3 & 1 & -1 & 1 & 1 & -3 & -1 & -1 & 0 & 
0 \rule[-7pt]{0pt}{5pt} \\ 
\hline 
\end{array} 
$\end{table}
Here 1a = (), 2a = (5,6), 2b = (3,4)(5,6), 2c = (3,5)(4,6), 4a = (3,5,4,6), 2d = (1,2)(3,4)(5,6), 2e = (1,2)(3,5)(4,6), 4b = (1,2)(3,5,4,6), 3a = (1,3,5)(2,4,6), 6a = (1,3,5,2,4,6).
\newpage

\begin{table}[ht]  \caption{Character table of $C_{S_6}((1,2,3))\cong \langle (1,2,3), (4,6), (5,6) \rangle$}
$
\begin{array}{r|ccccccccc} \hline 
 & 1a 
 & 2a 
 & 3a 
 & 3b 
 & 6a 
 & 3c 
 & 3d 
 & 6b 
 & 3e 
\rule[-7pt]{0pt}{20pt} \\ \hline 
\eta_{5.1} & 1 & 1 & 1 & 1 & 1 & 1 & 1 & 1 & 1 \rule[0pt]{0pt}{13pt} \\ 
\eta_{5.2} & 1 & -1 & 1 & 1 & -1 & 1 & 1 & -1 & 1 \\ 
\eta_{5.3} & 1 & -1 & 1 & E(3)^2 & -E(3)^2 & E(3)^2 & E(3) & -E(3) & E(3) \\ 
\eta_{5.4} & 1 & -1 & 1 & E(3) & -E(3) & E(3) & E(3)^2 & -E(3)^2 & E(3)^2 \\ 
\eta_{5.5} & 1 & 1 & 1 & E(3)^2 & E(3)^2 & E(3)^2 & E(3) & E(3) & E(3) \\ 
\eta_{5.6} & 1 & 1 & 1 & E(3) & E(3) & E(3) & E(3)^2 & E(3)^2 & E(3)^2 \\ 
\eta_{5.7} & 2 & 0 & -1 & 2 & 0 & -1 & 2 & 0 & -1 \\ 
\eta_{5.8} & 2 & 0 & -1 & 2*E(3) & 0 & -E(3) & 2*E(3)^2 & 0 & -E(3)^2 \\ 
\eta_{5.9} & 2 & 0 & -1 & 2*E(3)^2 & 0 & -E(3)^2 & 2*E(3) & 0 & 
-E(3) \rule[-7pt]{0pt}{5pt} \\ 
\hline 
\end{array} 
$\end{table}
Here 1a = (), 2a = (5,6), 3a = (4,5,6), 3b = (1,2,3), 6a = (1,2,3)(5,6), 3c = (1,2,3)(4,5,6), 3d = (1,3,2), 6b = (1,3,2)(5,6), 3e = (1,3,2)(4,5,6).

\begin{table}[ht]  \caption{Character table of $C_{S_6}((1,2,3)(4,5))\cong \langle (1,2,3), (4,5) \rangle \cong C_6$}
$
\begin{array}{r|cccccc} \hline 
 & 1a 
 & 2a 
 & 3a 
 & 6a 
 & 3b 
 & 6b 
\rule[-7pt]{0pt}{20pt} \\ \hline 
\eta_{6.1} & 1 & 1 & 1 & 1 & 1 & 1 \rule[0pt]{0pt}{13pt} \\ 
\eta_{6.2} & 1 & -1 & 1 & -1 & 1 & -1 \\ 
\eta_{6.3} & 1 & -1 & E(3)^2 & -E(3)^2 & E(3) & -E(3) \\ 
\eta_{6.4} & 1 & -1 & E(3) & -E(3) & E(3)^2 & -E(3)^2 \\ 
\eta_{6.5} & 1 & 1 & E(3)^2 & E(3)^2 & E(3) & E(3) \\ 
\eta_{6.6} & 1 & 1 & E(3) & E(3) & E(3)^2 & E(3)^2 \rule[-7pt]{0pt}{5pt} \\ 
\hline 
\end{array} 
$\end{table}
Here 1a = (), 2a = (4,5), 3a = (1,2,3), 6a = (1,2,3)(4,5), 3b = (1,3,2), 6b = (1,3,2)(4,5). 
%\newpage

\begin{table}[ht]  \caption{Character table of $C_{S_6}((1,2,3)(4,5,6))\cong \langle (1,2,3), (1,4)(2,5)(3,6), (4,5,6) \rangle$}
$
\begin{array}{r|ccccccccc} \hline 
 & 1a 
 & 3a 
 & 3b 
 & 3c 
 & 3d 
 & 3e 
 & 2a 
 & 6a 
 & 6b 
\rule[-7pt]{0pt}{20pt} \\ \hline 
\eta_{7.1} & 1 & 1 & 1 & 1 & 1 & 1 & 1 & 1 & 1 \rule[0pt]{0pt}{13pt} \\ 
\eta_{7.2} & 1 & 1 & 1 & 1 & 1 & 1 & -1 & -1 & -1 \\ 
\eta_{7.3} & 1 & E(3)^2 & E(3) & E(3) & 1 & E(3)^2 & -1 & -E(3)^2 & -E(3) \\ 
\eta_{7.4} & 1 & E(3) & E(3)^2 & E(3)^2 & 1 & E(3) & -1 & -E(3) & -E(3)^2 \\ 
\eta_{7.5} & 1 & E(3)^2 & E(3) & E(3) & 1 & E(3)^2 & 1 & E(3)^2 & E(3) \\ 
\eta_{7.6} & 1 & E(3) & E(3)^2 & E(3)^2 & 1 & E(3) & 1 & E(3) & E(3)^2 \\ 
\eta_{7.7} & 2 & -1 & -1 & 2 & -1 & 2 & 0 & 0 & 0 \\ 
\eta_{7.8} & 2 & -E(3) & -E(3)^2 & 2*E(3)^2 & -1 & 2*E(3) & 0 & 0 & 0 \\ 
\eta_{7.9} & 2 & -E(3)^2 & -E(3) & 2*E(3) & -1 & 2*E(3)^2 & 0 & 0 & 
0 \rule[-7pt]{0pt}{5pt} \\ 
\hline 
\end{array} 
$\end{table}
Here 1a = (), 3a = (4,5,6), 3b = (4,6,5), 3c = (1,2,3)(4,5,6), 3d = (1,2,3)(4,6,5), 3e = (1,3,2)(4,6,5), 2a = (1,4)(2,5)(3,6), 6a = (1,4,2,5,3,6), 6b = (1,4,3,6,2,5).
%\newpage
\clearpage

\begin{table}[ht]  \caption{Character table of $C_{S_6}((1,2,3,4))\cong \langle (1,2,3,4), (5,6) \rangle \cong C_4 \times C_2$}
$\begin{array}{r|cccccccc} \hline 
 & 1a 
 & 2a 
 & 4a 
 & 4b 
 & 2b 
 & 2c 
 & 4c 
 & 4d 
\rule[-7pt]{0pt}{20pt} \\ \hline 
\eta_{8.1} & 1 & 1 & 1 & 1 & 1 & 1 & 1 & 1 \rule[0pt]{0pt}{13pt} \\ 
\eta_{8.2} & 1 & -1 & -1 & 1 & 1 & -1 & -1 & 1 \\ 
\eta_{8.3} & 1 & -1 & 1 & -1 & 1 & -1 & 1 & -1 \\ 
\eta_{8.4} & 1 & 1 & -1 & -1 & 1 & 1 & -1 & -1 \\ 
\eta_{8.5} & 1 & -1 & -E(4) & E(4) & -1 & 1 & E(4) & -E(4) \\ 
\eta_{8.6} & 1 & -1 & E(4) & -E(4) & -1 & 1 & -E(4) & E(4) \\ 
\eta_{8.7} & 1 & 1 & -E(4) & -E(4) & -1 & -1 & E(4) & E(4) \\ 
\eta_{8.8} & 1 & 1 & E(4) & E(4) & -1 & -1 & -E(4) & 
-E(4) \rule[-7pt]{0pt}{5pt} \\ 
\hline 
\end{array} 
$\end{table}
Here 1a = (), 2a = (5,6), 4a = (1,2,3,4), 4b = (1,2,3,4)(5,6), 2b = (1,3)(2,4), 2c = (1,3)(2,4)(5,6), 4c = (1,4,3,2), 4d = (1,4,3,2)(5,6). \newline

The character table of $C_{S_6}((1,2,3,4)(5,6)) \cong \langle (1,2,3,4), (5,6) \rangle \cong C_4 \times C_2$ is identical to the table given just above, with the only exception being the characters are indexed by $9.j$.  Each of the 8 irreducible characters $\eta_{9.j}$ has the same character values as the corresponding irreducible characters $\eta_{8.j}$ listed in the table immediately above.
%\newpage

\begin{table}[ht]  \caption{Character table of $C_{S_6}((1,2,3,4,5))\cong \langle (1,2,3,4,5) \rangle \cong C_5$}
$
\begin{array}{r|ccccc} \hline 
 & 1a 
 & 5a 
 & 5b 
 & 5c 
 & 5d 
\rule[-7pt]{0pt}{20pt} \\ \hline 
\eta_{10.1} & 1 & 1 & 1 & 1 & 1 \rule[0pt]{0pt}{13pt} \\ 
\eta_{10.2} & 1 & E(5) & E(5)^2 & E(5)^3 & E(5)^4 \\ 
\eta_{10.3} & 1 & E(5)^2 & E(5)^4 & E(5) & E(5)^3 \\ 
\eta_{10.4} & 1 & E(5)^3 & E(5) & E(5)^4 & E(5)^2 \\ 
\eta_{10.5} & 1 & E(5)^4 & E(5)^3 & E(5)^2 & E(5) \rule[-7pt]{0pt}{5pt} \\ 
\hline 
\end{array} 
$
\end{table}
Here 1a = (), 5a = (1,2,3,4,5), 5b = (1,3,5,2,4), 5c = (1,4,2,5,3), 5d = (1,5,4,3,2).\newline

\begin{table}[ht]  \caption{Character table of $C_{S_6}((1,2,3,4,5,6))\cong \langle (1,2,3,4,5,6) \rangle \cong C_6$}
$\begin{array}{r|cccccc} \hline 
 & 1a 
 & 6a 
 & 3a 
 & 2a 
 & 3b 
 & 6b 
\rule[-7pt]{0pt}{20pt} \\ \hline 
\eta_{11.1} & 1 & 1 & 1 & 1 & 1 & 1 \rule[0pt]{0pt}{13pt} \\ 
\eta_{11.2} & 1 & -1 & 1 & -1 & 1 & -1 \\ 
\eta_{11.3} & 1 & E(3)^2 & E(3) & 1 & E(3)^2 & E(3) \\ 
\eta_{11.4} & 1 & -E(3)^2 & E(3) & -1 & E(3)^2 & -E(3) \\ 
\eta_{11.5} & 1 & E(3) & E(3)^2 & 1 & E(3) & E(3)^2 \\ 
\eta_{11.6} & 1 & -E(3) & E(3)^2 & -1 & E(3) & -E(3)^2 \rule[-7pt]{0pt}{5pt} \\ 
\hline 
\end{array} 
$\end{table}
Here 1a = (), 6a = (1,2,3,4,5,6), 3a = (1,3,5)(2,4,6), 2a = (1,4)(2,5)(3,6), 3b = (1,5,3)(2,6,4), 6b = (1,6,5,4,3,2).

{\bf Acknowledgement:} This paper is part of the author's Ph.D. dissertation at the University of Southern California, under the direction of S. Montgomery.  The author would like to thank S. Montgomery and also M. Iovanov, for helpful comments. 

%------------------------Bibliography-------------------------------------

\end{document}